\documentclass[preprint, reqno, 11pt]{imsart}
\pdfoutput=1
\usepackage[utf8]{inputenc}
\usepackage{amsmath}
\usepackage{amssymb}
\usepackage{graphicx}
\usepackage[numbers]{natbib}
\usepackage{bbm}
\usepackage{appendix}
\usepackage{amsthm}
\usepackage{xcolor}
\usepackage{tabularx,multicol,multirow,booktabs,csquotes,comment}
\usepackage[margin=0.9in]{geometry}
\usepackage{verbatim}
\usepackage{hyperref}
\hypersetup{
  colorlinks,
  linkcolor={red!50!black},
  citecolor={blue},
  urlcolor={blue!80!black}
}
\RequirePackage{hypernat}

\RequirePackage{algorithm}
\RequirePackage{algpseudocode}

\usepackage[shortlabels]{enumitem}
\usepackage{subfigure}
\usepackage[noabbrev,capitalize]{cleveref}

\newtheorem{thm}{Theorem}[section]
\newtheorem*{thm*}{Theorem}
\newtheorem{cor}[thm]{Corollary}
\newtheorem{defi}{Definition}[section]
\newtheorem{assume}{Assumption}[section]
\newtheorem{prop}[thm]{Proposition}

\newtheorem{lemma}[thm]{Lemma}
\newtheorem{ex}{Example}[section]
\newtheorem{remark}{Remark}[section]

\newcommand{\diag}{\text{\normalfont{Diag}}}

\DeclareMathOperator{\an}{\alpha_{\mathnormal p}}

\DeclareMathOperator{\mcp}{\mathsf{\normalfont{Poly}}}

\DeclareMathOperator{\EE}{\mathbb{E}}
\DeclareMathOperator{\Var}{\textnormal{Var}}
\DeclareMathOperator{\Cov}{\textnormal{Cov}}
\DeclareMathOperator{\PP}{\mathbb P}
\DeclareMathOperator{\RR}{\mathbb R}

\DeclareMathOperator*{\xp}{\xrightarrow{\mathnormal P}}
\DeclareMathOperator*{\xd}{\xrightarrow{\mathnormal d}}
\DeclareMathOperator{\bbeta}{\beta^\star}

\DeclareMathOperator*{\argmax}{arg\,max}
\DeclareMathOperator*{\argmin}{arg\,min}
\DeclareMathOperator{\bZ}{Z}
\DeclareMathOperator{\EEB}{\EE_{\bbst}}
\DeclareMathOperator{\sumik}{\sum_{i=1}^k}

\newcommand{\F}{\mathcal{F}}
\newcommand{\thetaveb}{\hat{\theta}^{\mathsf{vEB}}}

\newcommand{\QQ}{\mathbb{Q}}
\newcommand{\R}{\mathbb{R}}

\newcommand{\dkl}{\mathrm{D_{KL} }}
\newcommand{\dtv}{\mathrm{D_{TV} }}
\newcommand{\Flm}{F_{\mathrm{lim}}}

\newcommand{\sumin}{\sum_{\mathnormal i=1}^{\mathnormal p}}
\newcommand{\sumij}{\sum_{1\mathnormal{ \le i\neq j \le p}}}
\newcommand{\maxin}{\max_{\mathnormal i=1}^{\mathnormal p}}
\newcommand{\sumjn}{\sum_{\mathnormal j=1}^{\mathnormal p}}
\newcommand{\sumji}{\sum_{\mathnormal j\neq i}^{\mathnormal p}}
\newcommand{\OPX}{\mathnormal O_{\mathnormal P,X}}
\newcommand{\bbst}{\beta^\star}
\newcommand{\OPB}{\mathnormal O_{\mathnormal P,\bbst}}

\newcommand{\thetahat}{\hat{\theta}_{\mathnormal p}}

\begin{document}
	\begin{frontmatter}
		\title{Parametric Mean-Field Empirical Bayes in High-dimensional Linear Regression}
		\runtitle{Parametric Empirical Bayes in Regression}
        \runauthor{Lee and Deb}
		
		 \begin{aug}
			\author{\fnms{Seunghyun} \snm{Lee}\ead[label=e1]{sl4963@columbia.edu}}
		 	\and
		 	\author{\fnms{Nabarun} \snm{Deb}\ead[label=e2]{nabarun.deb@chicagobooth.edu}}
		 \end{aug}
        \address{Columbia University\printead[presep={,\ }]{e1}}
        \address{University of Chicago\printead[presep={,\ }]{e2}}
  
		\begin{abstract}
          In this paper, we consider the problem of parametric empirical Bayes estimation of an i.i.d. prior in high-dimensional Bayesian linear regression, with random design. We obtain the asymptotic distribution of the variational Empirical Bayes (vEB) estimator (motivated from \cite{carbonetto2012vi} and \cite{kim2024flexible}) which approximately maximizes a variational lower bound of the intractable marginal likelihood. We characterize a sharp phase transition behavior for the vEB estimator --- namely that it is information theoretically optimal (in terms of limiting variance) up to $p=o(n^{2/3})$ while it suffers from a sub-optimal convergence rate in higher dimensions. In the first regime, i.e., when $p=o(n^{2/3})$, we show how the estimated prior can be calibrated to enable valid coordinate-wise and delocalized inference, both under the \emph{empirical Bayes posterior} and the oracle posterior. In the second regime, we propose a debiasing technique as a way to improve the performance of the vEB estimator beyond $p=o(n^{2/3})$. Extensive numerical experiments corroborate our theoretical findings.
        \end{abstract}
		

		
	\end{frontmatter}
	
	\maketitle
	
\section{Introduction}
Consider a canonical linear regression model
\begin{align}\label{eq:regression model}
    y \mid X, \beta \sim N(X\beta, \sigma^2 I_n),
\end{align}
where $y \in \RR^n$ denotes the responses, $X \in \RR^{n \times p}$ denotes the design matrix, $\beta \in \RR^p$ is the unknown regression coefficients, and $\sigma^2 > 0$ is the error variance (potentially unknown). We work under a high-dimensional asymptotic setting, where both the sample size $n$ and the number of parameters $p$ grow to infinity. 
Here, we additionally posit an i.i.d. prior on the coefficient vector $\beta=(\beta_1,\ldots ,\beta_p)$, where the (random) coefficients ($\beta_i$) are generated i.i.d. from an \emph{unknown} prior distribution $\mu_{\theta_0}$, and $\theta_0 \in \Theta \subset \R^k$ is a (deterministic) finite-dimensional parameter. These models arise naturally in estimating effect size distributions and in applications involving genetic associations; see e.g.~\cite{zhang2018estimation,o2021distribution,zhou2021fast}. 
Our goal in this paper is two-fold: 

\begin{enumerate}[(a)]
\item To \emph{estimate the prior} $\mu_\theta$ in a data-driven manner, and 

\item To use the \emph{estimated prior for downstream inference} under the \emph{oracle posterior} on $\beta$. 
\end{enumerate}

 While traditional Bayesian inference assumes that the prior $\mu_\theta$ is specified by the statistician, this can be problematic for high-dimensional noisy data. For example, consider a toy scenario where $p=n$ and $X$ is the $n\times n$ identity matrix. This is the classical Gaussian sequence model $$y_i \mid \beta_i \overset{ind}{\sim} N(\beta_i,\sigma^2).$$ 
 In this case, the posterior distribution of $\beta \mid y,X$ witnesses proportional contribution from the prior and the likelihood, which implies that there is no posterior contraction. In other words, all posterior based inference depends heavily on the prior, even asymptotically as $n\to\infty$. A misspecified prior from the statistician would therefore result in inconsistent posterior inference. This motivates the development of procedures that can consistently estimate $\mu_{\theta_0}$. 
 
 \emph{Empirical Bayes} (EB) addresses this problem by using the observed data to estimate the prior. For simplicity, suppose the variance $\sigma^2$ is known. A natural strategy for estimating $\mu_{\theta_0}$ is to maximize the marginalized likelihood
\begin{align}\label{eq:marginal}
    m_\theta (y):=&\Big(\frac{1}{2\pi\sigma^2}\Big)^{\frac{n}{2}} \int \exp\Big(-\frac{1}{2\sigma^2}\|{y}-{X}{\beta}\|^2\Big) \prod_{i=1}^p d\mu_\theta(\beta_i)
\end{align}
with respect to the `hyperparameter' $\theta$. However, except in extreme cases (such as when $\mu_\theta$ is a normal distribution or $X^\top X$ is diagonal), the integral in \eqref{eq:marginal} does not have an explicit form. The problem is particularly challenging under the high-dimensional setting with $p \to \infty$, as the marginal likelihood in \eqref{eq:marginal} is extremely difficult to compute by standard numerical integration or sampling-based approaches. 

Recent works (see \cite{carbonetto2012vi, kim2024flexible, mukherjee2023mean, fan2023gradient}) have addressed this issue by maximizing a more tractable \emph{variational lower bound}, also commonly known as the evidence lower bound (ELBO), instead of \eqref{eq:marginal}. Commonly known as \emph{variational empirical Bayes} (vEB), this procedure has been extremely popular in modern machine learning problems (such as Latent Dirichlet Allocation in \cite{blei2003latent} and variational autoencoders in \cite{kingma2013auto}), but with limited theoretical understanding. While necessary conditions for consistency was recently established in \cite{mukherjee2023mean}, the explicit convergence rates, limiting distributions, and asymptotic efficiency of vEB estimators have not been studied in the literature, to our knowledge. In this work, we take a first step in this direction by analyzing an approximate vEB estimator in the parametric setting (see \cite{carbonetto2012vi,fan2025dynamical,fan2025dynamical2}).  We provide a complete characterization of its asymptotic behavior when $p=o(n)$, under appropriate scaling of the design matrix. We then use the estimated prior for coordinate-wise and delocalized posterior inference.

\subsection{Organization and summary of main results}
This work first analyzes \textit{frequentist} inferential guarantees of the vEB estimator $\thetaveb$ that maximizes a variational approximation of the marginal likelihood $m_\theta(y)$, assuming a random design matrix $X$ and regularity conditions for the prior parametric family. Interestingly, the optimality of $\thetaveb$ exhibits a phase transition depending on the asymptotic scaling of the problem. When $p \ll n^{2/3}$, the vEB estimator is $\sqrt{p}$-consistent, asymptotically normal, and attains the information-theoretically optimal variance. In contrast, all of these properties fail when $p \gg n^{2/3}$. This is stated in the following theorem, which combines the separate upper and lower bound claims in \cref{sec:mainres}. While the following result is stated under a correctly specified parametric family for the prior, we also derive analogous limit distributions under misspecified priors (see \cref{cor:misspecified}). 

\begin{thm*}
    Suppose $y$ is generated from the Bayesian regression model \eqref{eq:regression model} with a true prior parameter $\theta_0$. Assume a random design matrix with a low signal to noise ratio (SNR) scaling, alongside standard regularity conditions on the prior (see Assumptions \ref{assmp:random design}--\ref{assmp:prior technical}).
    Let $\thetaveb$ be the vEB estimator (see \eqref{eq:est def}). Under the asymptotic regime $p = o(n^{2/3})$, $\thetaveb$ is $\sqrt{p}$-consistent with a limiting normal distribution, and attains the information-theoretic optimal limiting variance. However, when $p \gg n^{2/3}$, $\thetaveb$ is no longer optimal and is only $n/p$-consistent with a non-normal limit.
\end{thm*}

Moving beyond frequentist recovery of the prior hyperparameter $\theta$, we illustrate in \cref{sec:consequences} that the vEB estimator can be used for downstream \textit{Bayesian} inference on the regression coefficients $\beta$, under the asymptotic regime $p = o(n^{2/3})$. First, we bound the total variation distance between coordinate-wise distributions under the \textit{EB posterior} (constructed using the estimated prior $\mu_{\thetaveb}$) and the \textit{oracle posterior} (constructed using the oracle prior $\mu_{\theta_0}$). This in particular implies that the mean/variance  computed under the variational approximation of the EB posterior are consistent.
Next, we address uncertainty quantification under the two posteriors. We mainly focus on the one-dimensional projection of the posterior, that is $q^\top \beta \mid y, X$, for some fixed vector $q \in \R^p$. Contrary to previous studies in high SNR settings which show that inference under the oracle posterior and the EB posterior are asymptotically equivalent (the so-called higher-order merging phenomena in \cite{rizzelli2024empirical}), we illustrate that this is not guaranteed in the low SNR setting. More specifically, the one-dimensional projection of the oracle posterior and the EB posterior merge if and only if the projection direction is an approximate contrast vector (i.e. $\frac{\sumin q_i}{\|q\|} = o(\sqrt{p})$), illustrating the need for caution for utilizing the EB posterior. Nonetheless, we provide EB adjusted credible intervals for such one-dimensional projections which provide valid  marginal coverage under the \emph{oracle posterior}, using the estimated prior hyperparameter  $\thetaveb$. 

In \cref{sec:simulations}, we verify satisfactory finite-sample performance of our theoretical results, under various priors that include common choices such as spike-and-slab and Gaussian mixtures. In \cref{sec:debiasing}, we also illustrate that the slower convergence rate of $\thetaveb$ in higher-dimensions can be improved by debiasing. To elaborate, we show that the debiased vEB estimator relaxes the requirement $p \ll n^{2/3}$ for the aforementioned normal limit to $p \ll n^{3/4}$. We conclude the paper in \cref{sec:conclusion} by discussing avenues for future work.

\subsection{Related works}

\subsubsection{Standard EB in denoising}
The standard EB setup for normal means \citep[this corresponds to taking $X = I_n$ in \eqref{eq:regression model}, see e.g.,][]{robbins1956empirical,james1992estimation,jiang2009general,castillo2020spike,soloff2025multivariate} lead to a product posterior, and hence results in \textit{independent observations} under the posterior. This allows the marginal likelihood to factorize, and does not require a high-dimensional integral computation as in \eqref{eq:marginal}. Additionally, the exact factorization of the posterior results in tractable representations for individual coordinates of the posterior mean and variance, using Tweedie's formula (see e.g. \cite{robbins1956empirical}).  In contrast, our regression setup does not have any such representations of the posterior moments due to the intractable normalizing constant, creating additional theoretical challenges for downstream inference.

\subsubsection{Statistical efficiency of MML for EB}
As stated earlier in \eqref{eq:marginal}, maximizing the marginal likelihood (MML) is a natural approach to empirical Bayes. Consequently, asymptotic properties of parametric MML was previously studied in the literature \citep{petrone2014bayes,szabo2013empirical,rousseau2017asymptotic,castillo2020spike,rizzelli2024empirical,zhang2020convergence}. 
In particular, \cite{petrone2014bayes} established consistency of the MML estimator to an ``oracle set'' as well as the consistency of the EB posterior. The recent work \cite{rizzelli2024empirical} additionally shows a higher-order merging of the EB posterior to the oracle posterior, which indicates that the EB posterior credible intervals enjoy coverage under the oracle posterior. 
However, this line of work %
considers a more classical scaling with high SNR where the EB posterior contracts. %
This is in sharp contrast to our high-dimensional low SNR scaling. Indeed, we will show that such higher-order merging may not happen under our setup. We also stress that going beyond consistency, no asymptotic normality or efficiency results have previously been established for the MML estimator, beyond simple settings.

\subsubsection{Gradient flow based methods}
Another direction for empirical Bayes methods in regression involves using gradient flows to solve the complex optimization problem in \eqref{eq:marginal} \citep{kuntz2023particle,fan2023gradient,fan2025dynamical,fan2025dynamical2}. In particular, the recent line of works \cite{fan2025dynamical,fan2025dynamical2} consider a very similar problem setup as ours, but instead analyze an adaptive Langevin dynamics estimator under proportional asymptotics $p \asymp n$. %
We believe these results complement ours, as we analyze the vEB estimator, consider the scaling $p = o(n)$, and impose different assumptions for priors. In \cref{sec:simulations}, we compare our vEB estimator with the adaptive Langevin algorithm from \cite{fan2025dynamical,fan2025dynamical2}. We also note that in the parametric EB setup, no consistency guarantees have been established for the gradient flow estimator.

\subsubsection{vEB for general latent variable models}
To our knowledge, vEB dates back to \cite{blei2003latent}, which first proposed the notion of maximizing the variational approximation lower bound of the intractable likelihood in the context of topic modeling. This idea was also utilized for unsupervised learning in the context of variational autoencoders (VAEs, \cite{kingma2013auto}) and item response theory \citep{cho2021gaussian} as well as various modern machine learning problems such as matrix factorization \citep{wang2021empirical,zhong2022empirical}, due to their excellent scalability. However, the theoretical understanding of such vEB procedures is limited, in particular for precise convergence rates and limiting distributions, which is precisely the gap we aim to fill in this paper.

\subsection{Notation}\label{subsec:notation}
\vspace{-1mm}
For two measures $\PP,\QQ$ on the same probability space, define the Kullback-Leibler (KL) divergence between $\QQ$ and $\PP$ as
 $$\dkl(\QQ\mid\PP) := \begin{cases} \EE_{\QQ}\log{\frac{d\QQ}{d\PP}} & \mbox{if } \QQ \ll \PP, \\  \infty & \mbox{otherwise},\end{cases}$$
and the Total Variation (TV) distance as 
 $$\dtv (\PP, \QQ) := \sup_{E \subseteq \R} \left|\PP(E) - \QQ(E) \right|.$$
 Let $I_n$ denote the $n \times n$ identity matrix. For two matrices $V_1, V_2$, write $V_1 \succ V_2$ when $V_1 - V_2$ is positive definite.
   Let $\|\cdot\|$ denote the usual $L^2$ norm for matrices/vectors. We use the usual Bachmann-Landau notations. In particular, for deterministic sequences $\{a_p\}_{p\ge 1}$, $ \{c_p\}_{p\ge 1}$ (with $c_p > 0$) and a constant $r > 0$, we write $a_p \lesssim_r c_p$ when there exists a constant $K(r)$ such that $a_p \le K(r) c_p$. Write $a_p\ll c_p$ if $a_p = o(c_p)$. Given a matrix/tensor $M_p$ we will use $\|M_p\|_F$ to denote its Frobenius norm. We also extend the $o/O$ notation to vectors/matrices/tensors: for a sequence of vectors/matrices/tensors $\{M_p\}_{p\ge 1}$ (with fixed dimensions), write $M_p = O(c_p)$ when $\|M_p\|_F = O(c_p)$, and write $M_p = o(c_p)$ when $\|M_p\|_F = o(c_p)$. We also use the notations $o_P, O_P$, and in particular write $o_{P,X}(1)$ and $\OPX(1)$ to capture the randomness of a random variable $X$.
   The notations $\overset{d}{\longrightarrow}, \overset{P}{\longrightarrow}$ denotes convergence in distribution and convergence in probability, respectively.
   For a vector $v \in \R^k$, write $v^2 := v \otimes v$ and $v^3 := v\otimes v \otimes v$, where $\otimes$ denotes the outer product. For a random variable $Z$, define its sub-Gaussian norm $\|Z\|_{\psi_2}$ as the smallest $K\ge 0$ that satisfies
   $$\EE e^{t(Z - \EE Z)} \le e^{K^2 t^2/2}, \quad \forall t \in \R.$$

   Throughout the paper, we work on an asymptotic setting where both $n,p \to \infty$. Hence, we understand $n \equiv n(p)$ as a function of $p$ and write limits in terms of $p$. All probabilities and expectations are conditioned on $X, \bbst$ unless specified otherwise. When the expectation is respect to $X$ or $\bbst$, we write $\EE_X$ or $\EE_{\bbst}$, respectively. We also define the notion of linear-quadratic tilts for probability measures.
\begin{defi}[Linear-quadratic tilt]\label{def:quadratic tilt}
    Given a base-measure $\mu_\theta$ and constants $d_i$ for each $i = 0, 1, \ldots, p$ (defined below in \eqref{eq:notation} and \cref{def:information}), let $\mu_{i,\theta}$ be a quadratic tilt of $\mu_\theta$ by setting the likelihood ratio as
    $$\frac{d\mu_{i,\theta}}{d\mu_\theta}(b) := e^{- \frac{d_i}{2}b^2 - C_\theta(d_i)}, \quad \text{where} \quad C_\theta(d_i) := \log \int e^{-\frac{d_i}{2}b^2} d \mu_\theta(b).$$
    Next, for $t\in\R$, let $\mu_{i,t,\theta}$ be the linear tilt of $\mu_{i,\theta}$, i.e., $$\frac{d\mu_{i,t,\theta}}{d\mu_{i,\theta}}(b) := e^{t b- \psi_{i,\theta}(t)}, \quad \text{where} \quad \psi_{i,\theta}(t) := \log \int e^{t b} d\mu_{i,\theta}(b).$$
    By properties of exponential families, the mean and variance of $\mu_{i,t,\theta}$ are $\psi_{i,\theta}'(t)$ and $\psi_{i,\theta}''(t)$, respectively.
\end{defi}

\section{Variational empirical Bayes}\label{sec:vEB}

\noindent To estimate the prior parameter $\theta$, we maximize a variational lower bound of the marginal log-likelihood $\log m_{\theta}(y)$. %
We first define some transformations of the data $(X, y)$, which will be useful in simplifying this quantity.
\begin{defi}[Transformed variables]\label{def:A}
Define $p$-dimensional vectors $w = (w_1, \ldots, w_p)$ and $d=(d_1,\ldots,d_p)$ as follows:
\begin{equation}\label{eq:notation}
w:=\frac{X^\top y}{\sigma^2}, \quad d_i :=\frac{(X^\top X)(i,i)}{\sigma^2} \quad \text{for} \;\;  i=1,\ldots, p.
\end{equation}
Also, let $A_p \in \RR^{p\times p}$ be the off-diagonal part of the matrix $-\sigma^{-2} X^\top X$ (with diagonal entries set to zero), that is
\begin{equation*}
    A_p(i,j) = - \frac{(X^\top X)(i,j)}{\sigma^2} \quad \text{if} \quad 1 \le i \neq j \le p \quad \mbox{and} \quad A_p(i,i) = 0 \quad \text{if} \quad 1\le i\le p.
\end{equation*}
\end{defi}

As $d$ and $A_p$ respectively denote the diagonal and off-diagonal parts of $\sigma^{-2} X^\top X$, we can write
$$\sigma^{-2} X^\top X = \diag(d) - A_p.$$
Using these notations, we expand the quadratic term in the exponent of \eqref{eq:marginal}:
$$ \exp \Big(- \frac{1}{2\sigma^2} \|y - X\beta\|^2\Big) = \exp \Big(- \frac{\|y\|^2}{2 \sigma^2} - \frac{1}{2} \sumin d_i \beta_i^2 + \frac{\beta^\top A_p \beta}{2}  + w^\top \beta \Big).$$
Now, by absorbing each summand $d_i \beta_i^2$ to the base measure $\mu_\theta$ and using the notations $C_\theta(d_i), \mu_{i,\theta}$ in \cref{def:quadratic tilt}, 
we can re-write the scaled log-likelihood as
\begin{align}\label{eq:marginal likelihood simplified}
\ell_p(\theta) := \frac{1}{p}\log m_{\theta}(y) = -\frac{n}{2p}\log(2\pi\sigma^2) - \frac{1}{2p\sigma^2}\lVert y\rVert^2 + \frac{1}{p} \sumin C_\theta(d_i) + \frac{1}{p}\log Z_p(w, \theta),
\end{align}
where 
\begin{align}\label{eq:normalizing constant def}
    Z_p(w, \theta) := \int \exp \Big(\frac{\beta^\top A_p \beta}{2} + w^\top \beta \Big) \prod_{i=1}^p d \mu_{i,\theta}(\beta_i).
\end{align}

We motivate the vEB (variational empirical Bayes) estimator, initially assuming a known variance $\sigma^2 > 0$ (see \cref{rem:unkvar} for the relaxation). While maximizing $\ell_p(\theta)$ as a function of $\theta$ is tempting, the normalizing constant $Z_p(w,\theta)$ in \eqref{eq:normalizing constant def} involves an intractable high-dimensional integral which makes the maximization computationally challenging. To remedy this issue, we build our estimator of $\theta_0$ using ideas from Mean-Field variational inference. To wit, 
first observe that the Gibbs variational formula implies the following dual representation of $\log Z_p(w,\theta)$ (c.f. \cite{wainwright2008graphical,mukherjee2022variational}):
\begin{align}
    \log Z_p(w, \theta) &= \sup_{\QQ \in \mathcal{P}(\R^p)}\left(\EE_{\QQ}\left[\frac{1}{2}\beta^{\top} A_p\beta+ w^{\top}\beta\right]-\dkl\big(\QQ|\prod_{i=1}^p \mu_{i,\theta}\big)\right), 
    \notag
\end{align}
where $ \mathcal{P}(\R^p)$ denotes the set of all probability measures on $\R^p$. Now, by restricting $\mathcal{P}(\R^p)$ to the set of product measures $\QQ = \prod_{i=1}^p Q_i$, we get a lower bound for $\log Z_p(w, \theta)$:
\begin{align}
    \log Z_p(w, \theta) &\ge \sup_{\QQ=\prod_{i=1}^p Q_i} \left(\EE_{\QQ}\left[\frac{1}{2}\beta^{\top} A_p\beta+ w^{\top}\beta\right]-\dkl\big(\QQ|\prod_{i=1}^p \mu_{i,\theta}\big)\right) \label{eq:MF over product} \\
    &= \sup_{u \in \R^p} \Bigg(\underbrace{\frac{u^\top A_p u}{2} + w^\top u - \sumin I_{i,\theta}(u_i)}_{=:M_\theta(u)}\Bigg), \label{eq:M_theta(u) definition}
\end{align}
where $I_{i,\theta}(t):=\dkl(\mu_{i,(\psi_{i,\theta}')^{-1}(t),\theta} | \mu_{i,\theta})$ denotes the KL divergence between a linear tilt (with mean $t$) and the base measure of the $i$th coordinate. 
Here \eqref{eq:M_theta(u) definition} follows from the fact that the supremum in \eqref{eq:MF over product} is attained by linear tilts $\mu_{i,\cdot,\theta}$, which allows reducing the optimization over measures to that over reals (see \cite{mukherjee2022variational,yan2020nonlinear}). Note that the vector $u$ effectively denotes the coordinate-wise mean of the product measure $\QQ$.

Hence, by plugging-in \eqref{eq:M_theta(u) definition} to \eqref{eq:marginal likelihood simplified}, the re-scaled marginal log-likelihood $\ell_p(\theta)$ can be lower bounded by the ``Mean-Field'' log-likelihood $\ell_p^{\text{MF}}(\theta)$ (equivalently, a re-scaled ELBO):

\begin{align}\label{eq:Mean-Field objective function}
    \ell_p(\theta) \ge \ell_p^{\text{MF}} (\theta) := -\frac{n}{2p}\log(2\pi\sigma^2) - \frac{1}{2p\sigma^2}\lVert y\rVert^2 + \frac{1}{p} \sumin C_\theta(d_i) + \frac{1}{p} \sup_{u \in \R^p} M_\theta (u).
\end{align}
The supremum in \eqref{eq:Mean-Field objective function} is still an implicit function of $\theta$. We therefore approximate it further to reduce to a more tractable objective in $\theta$. 

\begin{defi}[Mean-Field optimizers]\label{def:fixed point}
    Under a random design matrix (see \cref{assmp:random design} below), by \cite[Lemma 2.1]{lee2025clt}, the supremum in \eqref{eq:Mean-Field objective function} has a unique optimizer. Let us denote it by $u_{\theta} = (u_{1,\theta}, \ldots, u_{p,\theta})$. Then $u_{\theta}$ satisfies
    \begin{align}\label{eq:fixed pt}
        {u}_{i,\theta} = \psi_{i,\theta}' ({s}_{i,\theta} + w_i), \quad \text{where} \quad s_{i,\theta} := \sum_{j \neq i} A_p(i,j) {u}_{j,\theta},\quad  \forall i \le p.
    \end{align}
\end{defi}
\noindent Here, the identity \eqref{eq:fixed pt} follows by writing the first order conditions of the supremum.

For heuristics, assume that $\mu_{\theta}$ is compactly supported. Then under our random design condition (see \cref{assmp:random design}), direct computations show that the empirical measure $p^{-1}\sum_{i=1}^p \delta_{s_{i,\theta}}$ converges weakly to $0$. It is then natural to approximate $s_{i,\theta} \approx 0$ in the identity ${u}_{i,\theta} = \psi_{i,\theta}' ({s}_{i,\theta} + w_i)$ and define the approximate optimizer $\tilde{u}_{\theta} = (\tilde{u}_{1,\theta}, \ldots, \tilde{u}_{n,\theta})$ as $\tilde{u}_{i,\theta} := \psi'_{i,\theta}(w_i).$ As a result, by spelling out the KL divergences in $I_{i,\theta}$, we can write
$$\frac{1}{p}\sup_{u \in \R^p} M_\theta (u) \approx  \frac{\tilde{u}_{\theta}^\top A_p \tilde{u}_{\theta}}{2p} + w^{\top}\tilde{u}_{\theta}  - \sumin I_{i,\theta}(\tilde{u}_{i,\theta}) \approx \sumin \psi_{i,\theta} (w_i),$$
where $\psi_{i,\theta}$ is defined in \cref{def:quadratic tilt}. Now, define functions $F_i:\Theta \to \R$ for each $i \le p$ as 
\begin{align}\label{eq:def_Fi}
F_i(\theta) := \log \int e^{w_i \beta_i - \frac{d_i \beta_i^2}{2}} d\mu_\theta(\beta_i) = C_{\theta}(d_i)+\psi_{i,\theta}(w_i).
\end{align} 
By plugging in the above observations into the lower bound \eqref{eq:Mean-Field objective function}, we have
\begin{align}\label{eq:app-likeli}
\ell_p(\theta) \ge \ell_p^{\text{MF}}(\theta) \approx -\frac{n}{2p}\log(2\pi\sigma^2) - \frac{1}{2p\sigma^2}\lVert y\rVert^2 + \frac{1}{p} \sumin F_i(\theta).
\end{align}
Therefore maximizing the above (approximate) lower bound, we get our vEB estimator 
\begin{align}\label{eq:est def}
    \boxed{\thetaveb = \argmax_{\theta\in\Theta} \frac{1}{p} \sumin F_i(\theta).}
\end{align}

\begin{remark}[Unknown variance]\label{rem:unkvar}
So far, the vEB estimator has required the knowledge of the variance $\sigma^2 > 0$. Fortunately, even though consistent estimation for $\beta$ is impossible under the low SNR scaling with $\sigma^{-2}\|X^\top X\| = O(1)$, it is straightforward to consistently estimate $\sigma^2$ using the usual MSE estimator:
\begin{align}\label{eq:sigma square estimator}
    \hat{\sigma}^2 := \frac{y^\top (I_n - X (X^\top X)^{-1} X^\top) y}{n-p}.
\end{align}
Thus, by plugging-in the above variance estimate, the vEB estimator can be further extended. Alternatively, one may choose to maximize the approximate likelihood in \eqref{eq:app-likeli} with respect to both $(\theta,\sigma^2)$, as suggested in \cite{kim2024flexible}.
\end{remark}

\section{Prior recovery via vEB and Fisher optimality}\label{sec:mainres}

Now, we illustrate our main results regarding the statistical performance of $\thetaveb$. We discuss its limiting distribution in \cref{subsec:upper bound}, extension to the misspecified setting in \cref{subsec:misspecified}, and its information-theoretic optimality in \cref{sec:lower bound}. To present our theoretical results, we require several assumptions. First, to understand dimension dependence, we work under the following random design matrix with independent entries.

\begin{assume}[Random design]\label{assmp:random design}
    For some distribution $P$ with mean zero, variance one, and sub-Gaussian norm bounded by $K>0$,  suppose that $X_{k,i} \overset{i.i.d.}{\sim} \frac{1}{\sqrt{n}} P$ for all $k \le n, i \le p$. We will work under the asymptotic scaling $p=o(n)$.
\end{assume}

Here the $n^{-1/2}$ scaling for $X$ is chosen so that $\|X^\top X\|/\sigma^2 = \Theta(1)$, placing the model in a low signal to noise ratio (SNR) regime where individual regression coefficients $\beta$ are not consistently estimable. In contrast, when the prior belongs to a low-dimensional parametric family, its hyperparameters remain identifiable through aggregation across coordinates. This makes our high-dimensional scaling particularly natural for empirical Bayes analysis. Such regimes are standard in theoretical studies of high-dimensional Bayesian regression \citep{qiu2024sub,celentano2023mean,barbier2020mutual,lee2025bayesregression,saenz2025characterizing}, and in particular appealing for EB problems that analyze the consequences of estimating the prior \citep{mukherjee2023mean,fan2023gradient,fan2025dynamical2}. For simplicity, we also assume that the regression variance $\sigma^2$ is known; we expect that all following results will go through when this is replaced by the plug-in estimator in \eqref{eq:sigma square estimator} (since it is consistent at a much faster $\sqrt{n}$-rate; also see \cref{rem:unkvar}). 

Next, we impose the following regularity conditions on the prior, which are standard assumptions for efficient estimation in parametric models \citep[e.g. see Theorems 2.6 or 5.1 in][]{lehmann2006theory}. For example, exponential families satisfy these conditions.
\begin{assume}[Prior regularity]\label{assmp:prior}
    Suppose that the prior $\mu_\theta$ is a parametric family with log-likelihood $\ell(\theta; \beta)$ with respect to a non-degenerate base measure $\nu$, such that it satisfies the following regularity conditions:
    \begin{enumerate}
        \item[R0.] $\mu_\theta$ is identifiable and has common support.
        \item[R1.] The parameter space $\Theta \subset \R^k$ is a compact set for some fixed integer $k$, and the true parameter $\theta_0$ lives in the interior of $\Theta$.
        \item[R2.] %
        For all $\beta$ and $\theta \in \Theta$, the log-likelihood $\ell(\theta; \beta)$ is three-times differentiable.
        \item[R3.] The Bartlett's identities for $\nabla \ell := \frac{\partial \ell}{\partial \theta}$ and $\nabla^2 \ell := \frac{\partial^2 \ell}{\partial \theta^2}$ hold (see conditions (g), (h) in Theorem 2.6, \cite{lehmann2006theory}), and the Fisher information $I(\theta) := \Var_{B_\theta \sim \mu_\theta}[\nabla \ell(\theta; B_\theta)]$ is positive definite for all $\theta \in U$.
    \end{enumerate}
\end{assume}
\noindent In passing, we mention that R2 requires the derivatives to exist for all $\theta$, compared to the usual regularity condition that assumes differentiability only near $\theta_0$. This is mainly for technical convenience.

Finally, in this Section, we assume a well-specified model, in the sense that the prior parametric family $\mu_{\theta}$ and the regression variance $\sigma^2$ are both correct. We will discuss extensions to misspecified priors later in \cref{subsec:misspecified}.
\begin{assume}[Well-specified model]\label{assmp:well-specified}
    Assume that the data $y$ is generated from a correctly specified model with true prior parameter $\theta_0 \in \textnormal{int}(\Theta)$. In other words, assume that $y$ is generated from the same model used to construct the posterior:
    $$\beta_i^\star \stackrel{i.i.d.}{\sim} \mu_{\theta_0}, \quad y \mid X, \beta^\star \sim N(X \beta^\star, \sigma^2 I_n).$$
\end{assume}

\subsection{Upper bound}\label{subsec:upper bound}

To analyze $\thetaveb$, we first introduce additional notations that will be used to describe the limiting distribution.
\begin{defi}\label{def:information}
\begin{enumerate}[(a)]
    \item Recalling the regression variance $\sigma^2$, set $d_0 := \sigma^{-2}$. 
    \item For any $\theta \in \Theta$, define a pair of random variables $(B_\theta, W_\theta)$ by setting 
    $$B_\theta \sim \mu_\theta, \quad W_\theta \mid B_\theta \sim N(d_0 B_\theta, ~d_0).$$
    Also define a deterministic mapping $V:\Theta \to \R^{k\times k}$ by setting
    \begin{equation}\label{eq:V_theta}
        V(\theta) := I(\theta) - \EE \Big[\Var \big(\nabla \ell(\theta; B_\theta) \mid W_\theta \big) \Big] = \Var \Big[ \EE \big(\nabla \ell(\theta; B_{\theta}) \mid W_{\theta}\big) \Big] \succ 0_{k \times k}.
    \end{equation}
    \item For $t\in \R,~d >0$, let $B_{t,d,\theta}$ be a linear-quadratic tilted probability measure on $\R$ defined as $\frac{d B_{t,d,\theta}}{d B_\theta}(b) \propto e^{tb - \frac{d b^2}{2}},$ so that $B_\theta \mid (W_{\theta} = t) \sim B_{t,d_0,\theta}$.
    Finally, set $G_\theta: \R \to \R^k$ as $G_\theta(t) := \EE \nabla \ell(\theta; B_{t,d_0,\theta})$ and define a mapping $\kappa: \Theta \to \R^k$ as
    $$\kappa(\theta) := \frac{1}{2}\EE [B_\theta^2] \EE [G_\theta''(W_{\theta})].$$
\end{enumerate}
\end{defi}

We motivate the notation $W_\theta$ and $V(\theta)$ in part (b) above, which will later play a crucial role as the Fisher information. First, recall from R3 in \cref{assmp:prior} that $I(\theta) = \Var [\nabla \ell(\theta; B_\theta)]$ is exactly the Fisher information for a single \emph{observed} coefficient $B_{\theta} \sim \mu_\theta$. However, under the hierarchical model, we cannot directly observe the regression coefficient $\beta$ and we instead observe $y$. This is roughly equivalent to observing the transformed data $w$ in \eqref{eq:notation}: 
\begin{equation}\label{eq:w distribution}
    w = \sigma^{-2} X^\top y \mid X, \bbst \stackrel{d}{\equiv} N(d_0 X^\top X \bbst, d_0 X^\top X).
\end{equation}
Approximating $X^\top X \approx I_p$, we can understand each coordinate $w_i$ as i.i.d. random variables with distribution $W_{\theta}$.
Then, $V(\theta)$ can be understood as the Fisher information of the single sample $W_\theta$, which is formalized in the following lemma. %
Note that $V(\theta) \prec I(\theta)$ since $W_\theta$ has less information than $B_\theta$.

\begin{lemma}[Fisher information]\label{lem:variance checking}
    Let $\tilde{\ell}(\theta; \cdot)$ be the log-likelihood of the random variable $W_\theta$. Then, the corresponding Fisher information is $V(\theta)$:
    $$-\EE[\nabla^2 \tilde{\ell}(\theta; W_{\theta})] = V(\theta).$$
\end{lemma}

The vector $\kappa(\theta)$ in part (c) is inevitably more technical. While we refrain from displaying the details, this quantity is crucial for understanding the mean of $\sumin \nabla F_i(\theta_0)$, which is the gradient of the objective function in \eqref{eq:est def}. We will further motivate $\kappa(\theta)$ after presenting the main result (see the proof overview that precedes \cref{rmk:rate}).

To analyze the limiting distribution, we require additional technical assumptions for the prior in addition to the usual ones in \cref{assmp:prior}.
\begin{assume}[Additional prior regularity]\label{assmp:prior technical}
Assume the following conditions on the prior likelihood $\ell(\theta; \beta)$ and tilted random variables $B_{t,d,\theta}$ (see part (iii) in \cref{def:information}).
\begin{enumerate}
    \item[R4.] The first three derivatives of the log-likelihood satisfy the following polynomial growth condition, where $r>0$ is some positive constant: $$\|\nabla \ell(\theta; \beta)\| + \|\nabla^2 \ell(\theta; \beta)\|_F + \|\nabla^3 \ell(\theta; \beta)\|_F\lesssim 1+|\beta|^r, \quad \forall \theta \in \Theta.$$
    \item[R5.] All finite moments of $B_{\theta_0}$ exist.
    \item[R6.] The tilted random variables $B_{t,d,\theta}$ satisfy: $$\sup_{d \in B_\epsilon(d_0)} \sup_{\theta \in \Theta} \big|\EE [B_{0,d,\theta}]\big| \lesssim 1, \quad \sup_{d \in B_\epsilon(d_0)} \sup_{\theta \in \Theta} \|B_{t,d,\theta}\|_{\psi_2} \lesssim 1 + |t|, \quad \forall t \in \R,$$
    where $B_\epsilon(d_0) := (d_0-\epsilon, d_0+\epsilon)$ for some $\epsilon < \frac{d_0}{2}$, and $\|\cdot\|_{\psi_2}$ denotes the sub-Gaussian norm.
\end{enumerate}
\end{assume}

\begin{remark}[On Conditions R4 and R5]
    R4 is a technical polynomial growth condition for the first three derivatives of the log-likelihood. Note that this condition is mild since the log-likelihood is typically polynomially bounded in $\beta$ for common parametric families. This is required to control the remainder term in the Taylor expansion for analyzing the first-order condition of \eqref{eq:est def}, and is reminiscent of the usual integrability assumption for $\nabla^3 \ell(\theta; \beta)$. Here we need the additional regularity for $\nabla \ell(\theta; \beta)$ and $\nabla^2 \ell(\theta; \beta)$ as the third derivative of $F_i(\theta)$ involves all three derivatives. Also, R5 is a standard moment assumption for the prior. Note that it is weaker than a sub-Gaussian assumption on the prior.
\end{remark}

\begin{remark}[On Condition R6]
R6 requires sub-Gaussian tails for linear and quadratic tilts of $B_\theta$. This is weaker than assuming sub-Gaussianity directly on the prior $\mu_\theta$. Here, as $\Theta$ and $B_\epsilon(d_0)$ are bounded sets, the uniform control over all $\theta, d$ is natural. 
On the other hand, the linear tilt parameter $t$ is unbounded, so instead of taking a supremum over $t$, we allow the sub-Gaussian norm to depend on $|t|$. This flexibility allows us to cover log-concave as well as non log-concave mixture priors (see \cref{lem:prior condition checking} below). We note that our results go through with minor modifications if we replace the linear growth of the sub-Gaussian norm in R6 with a general polynomial order growth condition. The linear growth is primarily for expositional simplicity. 
\end{remark}

The following lemma illustrates concrete instances where conditions R5, R6 in \cref{assmp:prior technical} hold, which already generalizes existing assumptions in the literature. In particular, we can accommodate priors with unbounded support, mixture priors, and sparsity inducing priors. On the other hand, existing works on low SNR high-dimensional Bayesian regression require the prior to have (a) bounded support \citep{mukherjee2023mean, lee2025bayesregression,fan2023gradient}, (b) log-concave densities \citep{lacker2024mean,sheng2025stability}, or (c) densities that satisfy the log-Sobolev inequality alongside a stronger version of our R4 \citep{fan2025dynamical}. These assumptions in the existing literature exclude common choices in practical EB methods, such as adaptive shrinkage priors (this is essentially scale mixtures of Gaussians, see \cite{stephens2017false}). %

\begin{lemma}\label{lem:prior condition checking}
    Conditions R5, R6 in \cref{assmp:prior technical} hold when the prior $\mu_\theta$ satisfies any one of the following conditions.
    \begin{enumerate}[(a)]
        \item $\mu_\theta$ has bounded support.
        \item %
        $\mu_\theta$ is a strongly log-concave density. %
        \item $\mu_\theta$ is a finite mixture of $\delta_0$ and Lebesgue densities which satisfy (R5), the first condition in (R6), and the stronger second condition in (R6):
        $$\sup_{t \in \R} \sup_{d \in B_\epsilon(d_0)} \sup_{\theta \in \Theta} \|B_{t,d,\theta}\|_{\psi_2} \lesssim 1.$$
        In particular, these conditions for the densities (individual mixture components) are satisfied by distributions from parts (a) and (b) above, so Gaussian mixture priors under any parametrization satisfy R5 and R6.
    \end{enumerate}
\end{lemma}

Now, we state our main result on limiting distributions of $\thetaveb$.

\begin{thm}\label{thm:upper bound}
    Suppose Assumptions \ref{assmp:random design}, \ref{assmp:prior}, \ref{assmp:well-specified}, \ref{assmp:prior technical} hold. Then, for general $p = o(n)$, we have  $\thetaveb - \theta_0 = O_P \Big(\frac{1}{\sqrt{p}} + \frac{p}{n}\Big)$. 
    In particular, the following limits hold.
    \begin{enumerate}[(a)]
        \item When $p \ll n^{2/3}$, we have
        \begin{align}\label{eq:normal limit}
            \sqrt{p}(\thetaveb - \theta_0) \mid X \xd N\Big(0, V(\theta_0)^{-1}\Big).
        \end{align}
        \item When $\frac{p^{3/2}}{n} \to \delta \in (0, \infty)$, we have
        $$\sqrt{p}(\thetaveb - \theta_0) \mid X \xd N\Big(\delta V(\theta_0)^{-1} \kappa(\theta_0), V(\theta_0)^{-1}\Big).$$
        \item In contrast, when $n^{2/3} \ll p \ll n$, we have
        $$\frac{n}{p}(\thetaveb - \theta_0) \mid X \xp V(\theta_0)^{-1}\kappa(\theta_0).$$
    \end{enumerate}
\end{thm}

We elaborate on each part of \cref{thm:upper bound}. First consider the setting of part (a) which shows $\sqrt{p}$-consistency of the estimator, with a normal limit. This is the typical parametric rate expected in this problem, as justified in \cref{rmk:rate} below.
In our hierarchical model, as we only have indirect measurements of $\beta^\star = (\beta_1^\star, \ldots, \beta_p^\star) \stackrel{i.i.d.}{\sim} \mu_{\theta_0}$, the limiting variance $V(\theta_0)^{-1}$ is strictly larger than $I(\theta_0)^{-1}$, the limiting variance for estimating $\theta_0$ directly from $\beta^\star$. We will further discuss the optimality of this limiting variance by deriving a matching lower bound in \cref{sec:lower bound}.

Part (b) considers higher dimensions with $p \propto n^{2/3}$. While the convergence rate and limiting variance is same as that in part (b), an additional bias appears in the limit. Note that this bias grows as $\delta = \lim_{p \to \infty}\frac{p^{3/2}}{n}$ increases. In part (c), we further increase the dimension so that $\delta \to \infty$, or equivalently $n^{2/3} \ll p$. The resulting limiting distribution exhibits a \emph{phase transition} and results in a slower rate of convergence (since $\frac{n}{p} \ll \sqrt{p}$ under this regime). This is because now the bias term in part (b) dominates the variance. 

To better understand this phase transition, we briefly discuss the proof strategy when $k=1$. Assuming that $\theta$ is a one-dimensional parameter, the first-order conditions for \eqref{eq:est def} gives
$$\sqrt{p} (\thetaveb-\theta_0) \approx \frac{\frac{1}{\sqrt{p}} \sumin F_i'(\theta_0)}{\frac{1}{p} \sumin F_i''(\theta_0)}.$$
We show that the denominator converges to $V(\theta_0)$ as long as $p = o(n)$. To analyze the numerator, we additionally condition on the true coefficients $\bbst$ and show that
\begin{align}\label{eq:conditional limit}
    \frac{1}{\sqrt{p}} \sumin F_i'(\theta_0) \mid X, \bbst \approx N\big(\mu_p(X,\bbst), T_2 \big), ~~ \text{where} ~~ \mu_p(X,\bbst) := \frac{1}{\sqrt{p}} \EE\Big[\sumin F_i'(\theta_0)\mid X, \bbst\Big]
\end{align}
and $T_2 > 0$ is some constant variance (see \eqref{eq:T_1, T_2}). The phase transition on $p$ arises from analyzing this mean term, where we view each summand $F_i'(\theta_0) = \EE \nabla \ell(\theta_0; B_{w_i,d_i,\theta_0})$ as a function of $w_i$ and use \emph{second-order} Taylor expansion $w_i = w_i(y) \approx d_0 \beta_i^\star$ to write
\begin{align}\label{eq:mu_p approximation}
    \mu_p(X,\bbst) \approx \frac{1}{\sqrt{p}} \sumin \EE \nabla \ell(\theta_0; B_{d_0 \beta_i^\star, d_0,\theta_0}) + \frac{p^{3/2}}{n} \kappa(\theta_0) = G_{\theta_0}(d_0 \beta_i^\star) + \frac{p^{3/2}}{n} \kappa(\theta_0).
\end{align}
Here, the bias term $\kappa(\theta_0)$ arises from expanding the second term in the Taylor expansion, as the first derivative term turns out to be always negligible.

When $p \ll n^{2/3}$, the first summand in \eqref{eq:mu_p approximation} dominates and gives the limit $\mu_p(X,\bbst) \mid X \xd N(0, T_1)$ for another constant $T_1>0$ (see \eqref{eq:T_1, T_2}). Combining this with the conditional normal limit \eqref{eq:conditional limit} using the fact that $T_1+T_2 = V(\theta_0)$, we get $\frac{1}{\sqrt{p}} \sumin F_i'(\theta_0) \mid X \xd N(0, V(\theta_0))$. The conclusion in part (a) follows by Slutsky's theorem.

In contrast, suppose $p \gg n^{2/3}$. Then, as the first term in the RHS of \eqref{eq:mu_p approximation} is $O_P(1)$, the second term dominates. The resulting non-degenerate limit of $\mu_p$ requires a slower convergence rate: $$\frac{n}{p^{3/2}}\mu_p(X,\bbst) \mid X \xd \kappa(\theta_0).$$
Consequently, we have $\frac{n}{p^2} \sumin F_i'(\theta_0) \mid X \xd \kappa(\theta_0),$ and the conclusion in part (c) follows again by Slutsky's theorem.
\begin{remark}[$\sqrt{p}$ is the optimal rate of convergence]\label{rmk:rate}
    The convergence rate $\sqrt{p}$ is information-theoretically tight. To see this, consider the setting where one directly observes all regression coefficients $\beta^\star = (\beta_1^\star, \ldots, \beta_p^\star) \stackrel{i.i.d.}{\sim} \mu_{\theta_0}$. The resulting convergence rate for estimating $\theta_0$ is $\sqrt{p}$, and hence a faster convergence rate is impossible given only $y$. 
    
    We can also show that under our general asymptotic scaling $p = o(n)$, the method of moments estimator for $\theta$ is $\sqrt{p}$-consistent, which illustrates that the slower rate in part (c) is sub-optimal. For illustration, assume a simple setting with $k = 1$ and suppose that the mean-mapping $M(\theta) := \EE B_\theta$ is bijective. Then, by setting $\hat{\beta} := (X^\top X)^{-1}X^\top y$ as the usual least squares estimator, we can define the first moment-based estimator as: $$\hat{\theta}^{\mathsf{MoM}}_p := M^{-1}\Big(\frac{1}{p}\sumin \hat{\beta}_i\Big).$$
    By applying the Delta method, we can show the following limit for $p=o(n)$:
    $$\sqrt{p}(\hat{\theta}^{\mathsf{MoM}}_p - \theta_0) \xd N\Big(0, \frac{\sigma^2 + \Var (B_{\theta_0})}{M'(\theta_0)^2}\Big).$$
    This limiting variance is sub-optimal and strictly greater than $V(\theta_0)^{-1}$.
\end{remark}

\begin{remark}[Faster rate of convergence at singularity]\label{rmk:singularity}
    In part (c) of \cref{thm:upper bound}, note that %
    the $k$-dimensional vector $\kappa(\theta_0)$ may be zero. While this singularity is typically a measure-zero event (see Figure \ref{fig:kappa} below for illustration), the convergence rate at such points is \emph{faster} than the $\frac{n}{p}$-rate in part (c). In particular, if $\kappa(\theta_0) = 0$, we can extend \cref{thm:upper bound} to show that $\thetaveb - \theta_0 = O_P \Big(\frac{1}{\sqrt{p}} + \big(\frac{p}{n}\big)^{3/2} \Big)$. Consequently, \eqref{eq:normal limit} extends to the higher-dimensional regime of $p = o(n^{3/4})$, and the phase transition no longer happens at $p \asymp n^{2/3}$.
    \begin{figure}[h!]
    \centering
    \includegraphics[width=\linewidth]{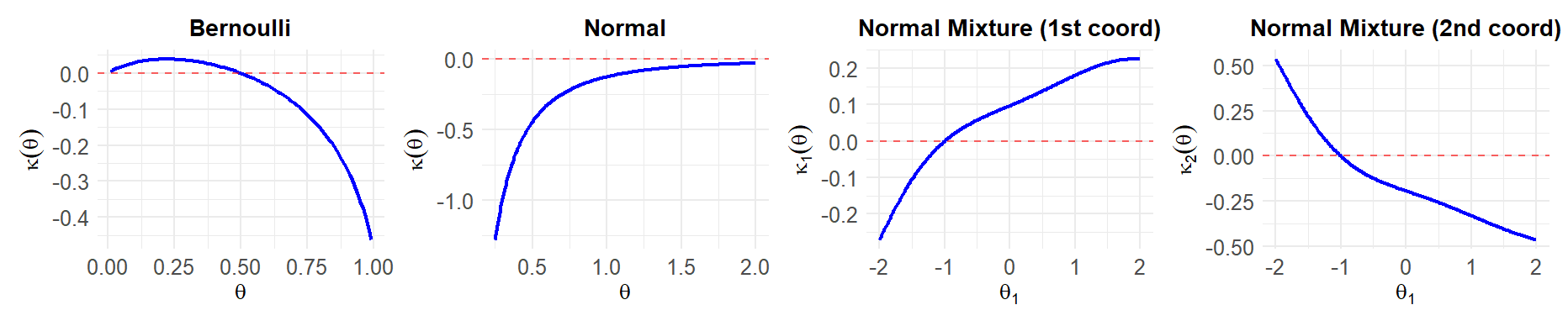}
    \vspace{-6mm}
    \caption{Visualization of $\kappa(\theta)$ under various priors: (first) $\textsf{Ber}(\theta)$, (second) Normal with unknown variance: $N(0, 1/\theta)$, (third/fourth) Gaussian mixtures $\frac{1}{2} N(\theta_1,1) + \frac{1}{2} N(\theta_2,0.25)$ with $\theta_2 = -1$ fixed throughout the visualization. For the mixture prior, $\kappa(\theta) = (\kappa_1(\theta), \kappa_2(\theta))$ is a two-dimensional vector, and each component is plotted in the third/fourth figure. We see that the singularity $\kappa(\theta) = 0$ is atypical. In particular, for the Gaussian mixture case, the singularity occurs when the two mixture components share a common mean $\theta_1 = \theta_2 = -1$.}
    \label{fig:kappa}
\end{figure}
\vspace{-5mm}
\end{remark}

To illustrate the limits in \cref{thm:upper bound}, we consider a toy example of the Gaussian prior with unknown mean. Here, conjugacy allows directly computing the limit distribution. %
\begin{ex}[Gaussian mean]\label{ex:gaussian}
    Suppose that the prior is $\mu_\theta \stackrel{d}{\equiv} N(\theta, 1)$. Then, the integral for defining the functions $F_i$ becomes tractable (see \eqref{eq:est def}), and we can write
    $$F_i(\theta) = \frac{1}{2}\log\Big(\frac{2\pi}{d_i+1}\Big) - \frac{\theta^2}{2} + \frac{(w_i+\theta)^2}{2(d_i+1)}.$$
    As this is a quadratic function in $\theta$, the vEB estimator in \eqref{eq:est def} can be explicitly written as
    $$\thetaveb = \frac{\sumin w_i/(d_i+1)}{\sumin d_i/(d_i+1)} = \frac{\sumin w_i}{p d_0} + O_P\Big(\frac{1}{\sqrt{n}}\Big),$$
    where the second equality used $d_i = d_0+O_{P,X}\Big(\frac{1}{\sqrt{n}}\Big).$
    Using the Gaussianity $w \mid X,\bbst \stackrel{d}{\equiv} N(\Sigma \bbst, \Sigma)$ where $\Sigma := d_0 X^\top X$ (where $d_0 = \sigma^{-2}$ is as in \cref{def:information}), we can write
    $$\thetaveb \mid X, \bbst \stackrel{d}{\equiv} \frac{1}{pd_0} N\Big({1_p^\top \Sigma \bbst}, 1_p^\top \Sigma 1_p\Big).$$
    Then, using the Gaussianity of $\bbst \stackrel{d}{\equiv} N(\theta_0 1_p, I_p)$, we can spell out the limit of the mean as 
    $$1_p^\top \Sigma \bbst \mid X \stackrel{d}{\equiv} N\Big(\theta_0 {1_p^\top \Sigma 1_p}, {1_p^\top \Sigma^2 1_p}\Big).$$
    By combining the two displays above and noting $\frac{1_p^\top \Sigma 1_p}{p} = d_0 + O_{P,X}\Big(\frac{1}{\sqrt{n}}\Big)$ (through standard moment computations), we have obtained the limit distribution of $\thetaveb$:
    $$\thetaveb \mid X \stackrel{d}{\equiv} N\Big(\theta_0, \frac{1+d_0}{p d_0}\Big) + O_P\Big(\frac{1}{\sqrt{n}}\Big).$$
    Under the asymptotic scaling $p = o(n)$, the error terms are negligible and we can write:
    \begin{equation}\label{eq:normal example limit}
        \sqrt{p}(\thetaveb - \theta_0) \mid X \xd N(0, 1+\sigma^2).
    \end{equation}

    To see that this matches \cref{thm:upper bound}, note that the corresponding log-likelihood derivative can be written as $\ell'(\theta; b) = b - \theta$ and $\Var(B_\theta \mid W_\theta) = \frac{1}{1+1/\sigma^2}$.
    We first compare with part (a) by computing the limiting variance under \cref{thm:upper bound}. Direct computations give 
    $$I(\theta) = \Var(B_\theta) = 1, \quad \EE \Var(B_\theta \mid W_\theta) = \frac{1}{1+1/\sigma^2},$$
    so we have $V(\theta) = \frac{1}{1+\sigma^2}$. This is exactly inverse of the limiting variance computed in \eqref{eq:normal example limit}.
    
    To check consistency with parts (b) and (c) in \cref{thm:upper bound}, note that $$G_\theta(t) = \EE \ell'(\theta; B_{t,d_0,\theta}) = \EE[B_{t,d_0,\theta}]-\theta = t$$ is exactly linear. Thus, $G_\theta'' \equiv 0$ for all $\theta$, and $\kappa \equiv 0$. In other words, every $\theta$ is a singular point in the sense of \cref{rmk:singularity}. The limit in part (c) becomes degenerate, and does not reflect the tight order as the actual limit in \eqref{eq:normal example limit}.
\end{ex}

\subsection{Extension to misspecified priors}\label{subsec:misspecified}
One natural question would be to go beyond the well-specified requirement in \cref{assmp:well-specified} by allowing $\bbst$ to be generated from a misspecified prior, say $\mu^\star \notin \{\mu_\theta\}_{\theta \in \Theta}$. Standard finite-dimensional Bayesian asymptotics \citep{berk1966limiting,kleijn2012bernstein} consider a \textit{misspecified likelihood} and showed that the resulting posterior contracts around a KL minimizer. Compared to such results, we consider a \textit{misspecified prior} while still assuming a correctly specified likelihood (see \eqref{eq:regression model}). Interestingly, as we shall illustrate, our results more resemble model misspecification in the frequentist literature \citep{white1982maximum}.

To formalize the setup, instead of requiring a well-specified class of priors as in \cref{assmp:well-specified}, assume that $\beta_i^\star \stackrel{i.i.d.}{\sim} \mu^\star$. Here, $\mu^\star$ is a true ``data-generating'' measure, which may not be parametrized as $\mu_\theta$. 
We also modify the notations in parts (b) and (c) of \cref{def:information} accordingly as follows.

\begin{defi}[Extension of \cref{def:information} to misspecified priors]

\begin{enumerate}
    \item[(a)] For any $\theta \in \Theta$, define random variables $(\beta^\star, W^\star, B_\theta^\star)$ as
    $$\beta^\star \sim \mu^\star, \quad W^\star \mid \beta^\star \sim N(d_0 \beta^\star, d_0), \quad B_\theta^\star \mid W^\star \sim  B_{W^\star,d_0,\theta}.$$
    Here, recall the notation $B_{t,d,\theta}$ from part (c) in \cref{def:information}. Set $V(\theta)$ as in \eqref{eq:V_theta}, and also define 
    $$V^\star(\theta) := V(\theta) - \EE\Big[\nabla^2 \ell(\theta; B_\theta^\star) + \big(\nabla \ell(\theta; B_\theta^\star) \big)\big(\nabla \ell(\theta; B_\theta^\star) \big)^\top \Big].$$
    \item[(ii)] Define a mapping $\kappa^\star: \Theta \to \R^k$ as
    $$\kappa^\star(\theta) := \frac{1}{2} \EE[(\beta^\star)^2] \EE[G_\theta''(W^\star)].$$
\end{enumerate}
\end{defi}
Recalling the discussion in \eqref{eq:w distribution}, $\beta^\star$ denotes a sample coefficient, and $W^\star$ denotes the approximation of each $w_i$. Also recalling from \cref{def:information} that $\frac{dB_{W^\star,d_0,\theta}}{d B_\theta}(b) \propto e^{-\frac{(W^\star-d_0 b)^2}{2d_0}}$, $B_\theta^\star$ denotes a draw from the posterior $\beta \mid W^\star$ (while using $\mu_\theta$ as the prior). Note that these modified notations reduce to that in \cref{def:information} under correct specification, in the sense that we have $B_\theta^\star \stackrel{d}{\equiv} B_\theta$ and $V^\star(\theta) = V(\theta)$ (by Bartlett's second identity) when $\mu^\star = \mu_\theta$. However such reductions are not possible for a general $\mu^\star \notin \{\mu_\theta\}_{\theta \in \Theta}$.

Finally, define $\theta^\star$ as the parameter value that minimizes the KL divergence between $W^\star$ and $W_\theta$, i.e.
\begin{equation}\label{eq:theta star}
    \theta^\star \in \argmin_{\theta \in \Theta} \dkl(W^\star \mid W_\theta). %
\end{equation}
Here, viewing $W^\star$ as an approximation of each $w_i$, $\theta^\star$ is the parameter value that best approximates this ``data-generating'' distribution $W^\star$. We formally state limit distributions of $\thetaveb$ under misspecified priors below.

    \begin{cor}[Generalization of \cref{thm:upper bound} to misspecified priors]\label{cor:misspecified}
        Suppose Assumptions \ref{assmp:random design}, \ref{assmp:prior} hold (where we replace all instances of $\theta_0$ to $\theta^\star$ in \cref{assmp:prior}), and suppose that \eqref{eq:theta star} has a unique minimizer. Assume R4,R6 for the class of priors $\mu_\theta$, and assume R5 for the random variable $B_{\theta^\star}^\star$. Then, for general $p = o(n)$,  $\thetaveb - \theta^\star = O_P \Big(\frac{1}{\sqrt{p}} + \frac{p}{n}\Big)$. 
        In particular, the following limits hold.
        \begin{enumerate}[(a)]
            \item When $p \ll n^{2/3}$, we have
            \begin{align*}
                \sqrt{p}(\thetaveb - \theta^\star) \mid X \xd N\Big(0, V^\star(\theta^\star)^{-1} V(\theta^\star) V^\star(\theta^\star)^{-1} \Big).
            \end{align*}
            \item When $\frac{p^{3/2}}{n} \to \delta \in (0, \infty)$, we have
            $$\sqrt{p}(\thetaveb - \theta^\star) \mid X \xd N\Big(\delta V^\star(\theta^\star)^{-1} \kappa^\star(\theta^\star), V^\star(\theta^\star)^{-1} V(\theta^\star) V^\star(\theta^\star)^{-1}\Big).$$
            \item When $n^{2/3} \ll p \ll n$, we have
            $$\frac{n}{p}(\thetaveb - \theta^\star) \mid X \xp V^\star(\theta^\star)^{-1}\kappa^\star(\theta^\star).$$
        \end{enumerate}
    \end{cor}

As a sanity check, we recover the conclusion of \cref{thm:upper bound} by taking $\mu = \mu^\star$ in \cref{cor:misspecified}. Here, note that the normal limit in part (a) is analogous to the limit theory for the frequentist MLE under likelihood misspecification (e.g. see Theorem 3.2 in \cite{white1982maximum}).

\subsection{Lower bound}\label{sec:lower bound}
We show that the limiting variance in \eqref{eq:normal limit} is optimal by deriving a matching lower bound. For this goal, define the Fisher information for the observed responses $y$. While $y$ is $n$-dimensional, the Fisher information matrix $\mathcal{I}_p(\theta)$ is still indexed by $p$ for notational consistency.
\begin{defi}[Fisher information]
    Let $\mathcal{I}_p(\theta)$ denote the $k \times k$ Fisher information computed from the marginal likelihood $m_\theta$ (see \eqref{eq:marginal}):
    $$\mathcal{I}_p(\theta) := - \EE\Big[\frac{\partial^2 \log m_\theta(y)}{\partial \theta^2} \mid X \Big].$$ %
    Here, the information is conditioned on a realized design matrix (as the above expectation is taken with respect to $y \mid X$), and thus implicitly depends on $X$.
    Also, note the distinction from the notation $I(\theta)$ in \cref{def:information}, which denotes the Fisher information for a \emph{single sample} from the prior $\mu_\theta$.
\end{defi}

By the multivariate Cramér–Rao bound \citep[e.g. see Theorem 2.6.1 in][]{lehmann2006theory}, for any unbiased estimator $\thetahat$, its variance is bounded below by the inverse Fisher information:
$$\EE \left[(\hat{\theta}_p - \theta_0) (\hat{\theta}_p - \theta_0)^\top \mid X\right] \succeq {\mathcal{I}_p(\theta_0)}^{-1}, \quad \forall \theta_0 \in \Theta.$$
Thus, understanding the fundamental limits of estimation boils down to evaluating the Fisher information $\mathcal{I}_p(\theta_0).$ Below we state our main lower bound result, under a technical bounded support assumption.\footnote{This assumption follows as we use the moment bounds developed in \cite{lee2025clt}; we believe this can be relaxed to assuming $\Var[B_{t,d,\theta}] \le K$ for some uniform constant $K$, which also allows log-concave priors.}

\begin{assume}[Bounded support]\label{assmp:cpt support}
    Assume that the common support of the prior $\mu_\theta$ is bounded and that the log-likelihood $\ell(\theta; b)$ is continuous with respect to both $\theta, b$. %
    Note that this assumption is one sufficient condition for \cref{assmp:prior technical}.
\end{assume}

\begin{thm}\label{thm:lower bound}
    Suppose Assumptions \ref{assmp:random design}, \ref{assmp:prior}, \ref{assmp:well-specified}, \ref{assmp:cpt support} hold. Under the asymptotic scaling $p \ll n^{2/3}$, we have $\frac{\mathcal{I}_p(\theta_0)}{p} \mid X \xp V(\theta_0).$
\end{thm}

As the scaled limiting Fisher information matches the limiting variance in \eqref{eq:normal limit}, we can conclude that the vEB estimator $\thetaveb$ attains the optimal limiting variance when $p \ll n^{2/3}$. The proof of \cref{thm:lower bound} builds upon the crucial observation that the Fisher information $\mathcal{I}_p(\theta_0)$ can be expressed as moments of \emph{Random Field Ising Models} (RFIMs, see \cref{def:rfim}). Utilizing recent developments for RFIMs in \cite{lee2025clt}, we establish moment bounds for the generalized summand $\sumin g(\beta_i)$ (see \cref{thm:CLT}) for some smooth function $g(\cdot)$, which may be of independent interest.

\section{Consequences of EB in downstream posterior inference}\label{sec:consequences}
Moving on from the frequentist guarantees for recovering the prior $\mu_{\theta_0}$ discussed in \cref{sec:mainres}, now we investigate consequences of EB (Empirical Bayes) in downstream \emph{posterior inference} for the regression coefficients $\beta = (\beta_1, \ldots, \beta_p)$. To set notation, we formally define the oracle posterior and the EB posterior. 

\begin{defi}[Oracle and EB posterior]\label{def:posterior}
    Given a prior $\mu_\theta$, let $\PP_{\theta}(\beta)$ denote the posterior distribution under the likelihood in \eqref{eq:regression model}:
    $$\PP_{\theta}(\beta) = \PP_{\theta}(\beta \mid y, X) \propto \PP(y \mid X, \beta) \prod_{i=1}^p d\mu_{\theta}(\beta_i).$$
    In particular, we say that $\PP_{\theta_0}(\beta)$ is the \emph{oracle posterior}. For \emph{any estimator} $\thetahat = \thetahat(y,X)$ of $\theta_0$, we say that $\PP_{\thetahat}(\beta)$ is the \emph{EB posterior}.
\end{defi}
\noindent Here note that the EB posterior $\PP_{\thetahat}$ depends on the estimated prior $\thetahat$. So, the data $y$ is double-dipped to construct $\PP_{\thetahat}$; once to write the likelihood $\PP(y \mid X, \beta)$ and second to estimate the prior $\thetahat$. Such notions of EB posteriors have been of interest for a wide range of statistical problems, including sparse linear regression \citep{johnstone2004needles,zhang2020convergence} as well as model selection \citep{rousseau2020asymptotic}.

In this section, the notation $\thetahat$ will denote any generic estimator, and $\thetaveb$ will continue to denote the vEB estimator \eqref{eq:est def}. Also, $\beta$ will denote a sample from the posterior (or its variational approximation), as opposed to the true regression coefficients $\beta^\star$ (see \cref{assmp:well-specified}).

Our main goal is to understand whether the EB posterior is a good approximation of the oracle posterior. In this section, we answer this question in terms of \emph{posterior inference} for (a) each coordinate, and (b) a one-dimensional projection. Due to technical reasons (mainly to use dependent concentration inequalities in \cite{lee2025clt}), this section will work under \cref{assmp:cpt support} and the asymptotic scaling $p = o(n^{2/3})$. %

\subsection{Estimating marginal distributions}

Here, note that even understanding the oracle posterior itself is a nontrivial task due to its dependent and high-dimensional structure. In fact, it is extremely challenging to even compute the posterior mean and variance due to the intractable normalizing constant. Fortunately, the variational approximation turns out to be useful for understanding these intractable quantities. Below, we introduce the notion of \emph{variational posteriors}, which are product measure approximations of the corresponding posterior.

\begin{defi}[Variational posteriors]\label{def:MF posterior}
    Recall from \cref{def:posterior} that $\PP_\theta$ denotes the usual posterior given a prior $\mu_\theta$. Define $\QQ_\theta$ as the Mean-Field variational approximation of $\PP_\theta$, i.e.,
    $$\QQ_\theta(\beta) := \argmin_{\QQ = \prod_{i=1}^p Q_i} \dkl(\QQ\mid \PP_\theta ),$$
    where each $Q_i$ has identical support as $\mu_\theta$. Note that $\QQ_\theta$ is uniquely determined when the design matrix satisfies \cref{assmp:random design}, as discussed in \cref{def:fixed point}. 
    In particular, we say that $\QQ_{\theta_0}$ is the \emph{variational oracle posterior}, and $\QQ_{\thetahat}$ is the \emph{variational EB posterior} (given an estimator $\thetahat$\footnote{Here, $\thetahat$ is not necessarily the vEB estimator $\thetaveb$. Note that $\QQ_{\thetahat}$ should be understood as the ``variational approximation'' of the EB posterior defined in \cref{def:posterior}, and does not have to use the vEB estimator.}). Under the variational oracle posterior, define the coordinate-wise mean and variance as
    $$u_i := \EE_{\beta \sim \QQ_{\theta_0}}[\beta_i], \quad \tau_i^2 := \Var_{\beta \sim \QQ_{\theta_0}}[\beta_i].$$
    Similarly, define analogous quantities under the variational EB posterior as
    $$\hat{u}_i := \EE_{\beta \sim \QQ_{\thetahat}}[\beta_i], \quad \hat{\tau}_i^2 := \Var_{\beta \sim \QQ_{\thetahat}}[\beta_i].$$
\end{defi}
Here, note that $u_i, \tau_i^2$ are oracle quantities that are functions of $\theta_0$, whereas $\hat{u}_i, \hat{\tau}_i^2$ depend on the estimated hyperparameter $\thetahat$. Also, note that these quantities are easily computable \citep[see e.g.][]{lee2025clt,lee2025bayesregression}, as opposed to the exact posterior mean and variance. In fact, these quantities connect to the Mean-Field optimizers from \cref{def:fixed point} as $u_i = u_{i,\theta_0}, \hat{u}_{i} = u_{i,\thetahat}$ (we will use the new notations for brevity).

The following proposition establishes high probability bounds between these quantities.
\begin{thm}\label{prop:mean variance}
    Suppose $p \ll n^{2/3}$ and assume Assumptions \ref{assmp:random design}, \ref{assmp:cpt support}. Consider any estimator $\thetahat$. For each $i \le p$, the following conclusions hold.
    \begin{enumerate}[(a)]
        \item For both the oracle posterior $\PP_{\theta_0}$ and EB posterior $\PP_{\thetahat}$, their variational approximation is accurate for each coordinate:
        $$\dtv\big(\PP_{\theta_0}(\beta_i), \QQ_{\theta_0}(\beta_i)\big) = O_{P,X}\Big(\sqrt{\frac{p}{n}}\Big), \quad \dtv\big(\PP_{\thetahat}(\beta_i), \QQ_{\thetahat}(\beta_i)\big) = O_{P,X}\Big(\sqrt{\frac{p}{n}}\Big), \quad \forall i \le p.$$
        In particular, the coordinate-wise mean and variance are consistently estimated:
        \begin{align}
            \EE_{\PP_{\theta_0}} [\beta_i] &= u_i + O_{P,X}\Big(\sqrt{\frac{p}{n}}\Big), \quad \Var_{\PP_{\theta_0}} [\beta_i] = \tau_i^2 + O_{P,X}\Big(\sqrt{\frac{p}{n}}\Big), \label{eq:mean oracle posterior} \\
            \EE_{\PP_{\thetahat}} [\beta_i] &= \hat{u}_i + O_{P,X}\Big(\sqrt{\frac{p}{n}}\Big), \quad \Var_{\PP_{\thetahat}} [\beta_i] = \hat{\tau}_i^2 + O_{P,X}\Big(\sqrt{\frac{p}{n}}\Big). \label{eq:mean EB posterior}
        \end{align}

        \item %
        The vEB posterior $\QQ_{\thetahat}$ is close to the oracle posterior $\PP_{\theta_0}$ for each coordinate:
        $$\dtv\big(\PP_{\theta_0}(\beta_i), \QQ_{\thetahat}(\beta_i)\big) = O_{P,X}\Big(\sqrt{\frac{p}{n}} + \big(\frac{p}{\sqrt{n}}+1\big) \|\thetahat - \theta_0\|\Big), \quad \forall i \le p.$$
        In particular, the coordinate-wise mean and variance satisfy the following:
        \begin{align}\label{eq:veb error}
            \EE_{\PP_{\theta_0}} [\beta_i] &= \hat{u}_i + O_{P}\Big(\sqrt{\frac{p}{n}} + \big(\frac{p}{\sqrt{n}}+1\big) \|\thetahat - \theta_0\|\Big), \\
            \Var_{\PP_{\theta_0}} [\beta_i] &= \hat{\tau}_i^2 + O_{P}\Big(\sqrt{\frac{p}{n}} + \big(\frac{p}{\sqrt{n}}+1\big) \|\thetahat - \theta_0\|\Big). \notag
        \end{align}
    \end{enumerate}
\end{thm}

Here, part (a) illustrates that the variational approximations give meaningful estimates of the coordinate-wise mean and variance, for both the oracle posterior and EB posterior. Note that this claim justifies the variational approximation of the posterior, but does not discuss the effect of estimating the prior. Part (b) addresses this, by providing error bounds for moments of the vEB posterior and corresponding oracle quantities. For illustration, take $\thetahat = \thetaveb$. Then, still under $p \ll n^{2/3}$, part (a) of \cref{thm:upper bound} gives $\|\thetaveb -\theta_0\|=O_P \Big(\frac{1}{\sqrt{p}}\Big),$ and \eqref{eq:veb error} simplifies to:
$$\EE_{\PP_{\theta_0}} [\beta_i] = \hat{u}_i + O_{P}\Big(\sqrt{\frac{p}{n}} + \frac{1}{\sqrt{p}}\Big), \quad \Var_{\PP_{\theta_0}} [\beta_i] = \hat{\tau}_i^2 + O_{P}\Big(\sqrt{\frac{p}{n}} + \frac{1}{\sqrt{p}}\Big).$$
Thus, the vEB posterior constructed under $\thetaveb$ consistently estimates the marginal mean and variance. 

\subsection{Uncertainty quantification for linear statistics}

Now we move beyond pointwise estimation and discuss the consequences of vEB on uncertainty quantification. We first consider the oracle case. Due to our low SNR scaling where the posterior crucially depends on the choice of prior, one cannot establish a normal limit for each coordinate of the posterior \citep{castillo2020spike,saenz2025characterizing}. Recently, jointly with Sumit Mukherjee, the authors established normal limiting distributions for \emph{one-dimensional projections} of the posterior \citep{lee2025bayesregression}. We restate part (a) of Theorem 3.1 and Proposition 3.2 in \cite{lee2025bayesregression} below.
\begin{prop}[CLT under the oracle posterior]\label{prop:oracle CLT}
    Let $q \in \R^p$ be a deterministic vector normalized as $\|q\|=1$, satisfying $\|q\|_\infty=o(1)$. Assume that the prior $\mu_{\theta_0}$ is compactly supported. Set $\upsilon_p = \upsilon_p(y) := \Var_{\QQ_{\theta_0}} (\sumin q_i \beta_i) = \sumin q_i^2 \tau_i^2$. Under the scaling $p \ll n^{2/3}$, we have
    \begin{align}\label{eq:upsilon defn}
        \upsilon_p \mid X \xp \upsilon(\theta_0) := \EE \Var(W_{\theta_0} \mid B_{\theta_0}),
    \end{align}
    and the following limit holds under both the oracle and variational posteriors:
    $$\frac{\sumin q_i(\beta_i - u_i)}{\sqrt{\upsilon_p}} \mid y, X \xrightarrow[\PP_{\theta_0}]{d} N(0,1), \quad \frac{\sumin q_i(\beta_i - u_i)}{\sqrt{\upsilon_p}} \mid y, X \xrightarrow[\QQ_{\theta_0}]{d} N(0,1).$$
\end{prop}
In other words, under the oracle posterior, a CLT with a normal limiting distribution holds for the statistic $\sumin q_i \beta_i$. Interestingly, the above result shows that one may replace the intractable mean and variance of the oracle posterior $\PP_{\theta_0}$ with the tractable counterparts that arise from the Mean-Field approximation $\QQ_{\theta_0}$. This allows one to perform posterior inference, such as constructing credible intervals with explicit expressions. Note that establishing limit distributions of the oracle posterior requires the asymptotic scaling $p \ll n^{2/3}$.

Armed with this result, our goal is to perform inference for one-dimensional projections $\sumin q_i\beta_i$ when $\theta_0$ is unavailable. Given an EB estimator $\thetahat$, two immediate questions are:
\begin{itemize}
    \item Does \cref{prop:oracle CLT} generalize to the EB posterior?
    \item Can one modify the centering and variance in \cref{prop:oracle CLT} to data-driven quantities (for example the plug-in counterparts using $\thetahat$) to get a limit under the oracle posterior?
\end{itemize}
The following result answers both questions: the first holds (see part (b)), whereas the plug-in version of the second generally fails (see part (c)). Recalling the notation $B_{t,d,\theta}$ and $W_{\theta_0}$ from \cref{def:information}, define the following Jacobian $J:\textnormal{int}(\Theta) \to \R^k$:
\begin{align}\label{eq:J}
    J(\theta) := \EE \left[\frac{\partial \EE[B_{W_\theta,d_0,\theta} \mid W_\theta]}{\partial \theta} \right], \quad \forall \theta \in \textnormal{int}(\Theta).
\end{align}

\begin{thm}\label{thm:EB}
    Let $q = (q_1, \ldots, q_p) \in \R^p$ denote a standardized deterministic vector with $\|q\|_\infty = o(1)$, and set $\hat{\upsilon}_p := \sumin q_i^2 \hat{\tau}_i^2$.
    Suppose $p \ll n^{2/3}$ and assume $\thetahat \xp \theta_0$.
    \begin{enumerate}[(a)]
    \item We have $\hat{\upsilon}_p \mid X \xp \upsilon(\theta_0),$ where $\upsilon(\theta_0)$ is the value defined in \eqref{eq:upsilon defn}.
    
    \item Assuming $\sumin q_i^4 = o\Big(\frac{n}{p^2}\Big)$, the following holds under the EB posterior ${\PP}_{\thetahat}$ (and also under the vEB posterior ${\QQ}_{\thetahat}$):
    \begin{align}\label{eq:CLT EB posterior}
        \frac{\sumin q_i(\beta_i - \hat{u}_i)}{\sqrt{\hat{\upsilon}_p}} \mid y, X \xrightarrow[\PP_{\thetahat}]{d} N(0,1).
    \end{align}
    
    \item Additionally assume that $\sqrt{p}(\thetahat-\theta_0) \mid X \xd N(0, S(\theta_0))$ for some continuous function $S:\Theta \to \R^{k \times k}$, and that $\frac{1}{\sqrt{p}} \sumin q_i \to \gamma$ for some constant $-1 \le \gamma \le 1$. 
    Then, when $\gamma = 0$, \eqref{eq:CLT EB posterior} holds under the oracle posterior ${\PP}_{\theta_0}$. For general $\gamma$, the following holds:
    \begin{align}\label{eq:CLT oracle posterior}
        \frac{\sumin q_i(\beta_i - \hat{u}_i)}{\sqrt{\hat{\upsilon}_p + \gamma^2 J(\thetahat)^\top S(\thetahat) J(\thetahat)}} \mid X \xrightarrow[\PP_{\theta_0}]{d} N(0,1).
    \end{align}
    Here, note that the LHS is a sample-based quantity that only depends on $y, X$. 

    \end{enumerate}
\end{thm}

\begin{remark}
    Note that Cauchy-Schwartz gives $\frac{1}{\sqrt{p}} |\sumin q_i| \le \sqrt{\sumin q_i^2} = 1$, which justifies the range $|\gamma|\le 1$ in part (c). In particular, when $\gamma = 0$, one may understand $q$ as an approximate contrast vector.
    We also note in passing that the centering $\sumin q_i \hat{u}_i$ in parts (b), (c) can be replaced by a more direct value of $\sumin q_i \psi_{i,\thetahat}'(w_i)$.
\end{remark}

We elaborate on each conclusion. Part (a) extends the limit of the variational oracle posterior variance in \eqref{eq:upsilon defn} to the variance of the vEB posterior. Part (b) considers the EB posterior and shows that the oracle CLT in \cref{prop:oracle CLT} can be extended by replacing $\theta_0$ with $\thetahat$. Part (c) instead considers the oracle posterior, %
and shows that the limit \eqref{eq:CLT EB posterior} extends when $\gamma = 0$, or equivalently when $q$ is a contrast vector. However, for a general $\gamma \neq 0$, the limit no longer holds as the vEB posterior mean is biased. In other words, the estimation error of $\thetahat$ results in a bias term with nontrivial fluctuations: $$\EE_{\PP_{\theta_0}} [q^\top \beta] - \EE_{\PP_{\thetahat}} [q^\top \beta] = \sumin q_i u_i - \sumin q_i \hat{u}_i +o_P(1) = \Theta_P(1).$$
This is why we have to marginalize over $y$ to get a standard normal limit in \eqref{eq:CLT oracle posterior}. The notation $J(\theta)$ in \eqref{eq:J} also arises from this bias, as we use a first-order Taylor expansion to approximate $$\hat{u}_i - u_i = \EE_{\QQ_{\hat{\theta}_p}}[\beta_i] -\EE_{\QQ_{{\theta}_0}}[\beta_i] \approx (\hat{\theta}_p-\theta_0) J(\hat{\theta}_p).$$

The following corollary summarizes inferential consequences of \cref{thm:EB} by \textit{using the vEB estimator} $\thetaveb$ to construct data-driven $100(1-\alpha)\%$ credible intervals with (conditional) coverage under the EB posterior, and (marginal) coverage under the oracle posterior. Here, for $\alpha \in (0,1)$, let $c_{\alpha/2}$ denote a constant such that $\PP(N(0,1) > c_{\alpha/2}) = \alpha/2$.
\begin{cor}\label{cor:coverage}
\begin{enumerate}[(a)]
    \item The interval 
    $$\mathcal{I}(y, X, \thetaveb) = \Big[\sumin q_i \hat{u}_i \pm c_{\alpha/2} \sqrt{\hat{\upsilon}_p}\Big]$$
    has asymptotic $1-\alpha$ coverage under the EB posterior ${\PP}_{\thetaveb}$:
    $$\PP_{\thetaveb}\Big(\sumin q_i \beta_i \in {\mathcal{I}} (y, X, \thetaveb) \mid y,X\Big) \to 1-\alpha.$$
    However, $\mathcal{I}(y, X, \thetaveb)$ has undercoverage under the oracle posterior $\PP_{\theta_0}$ if $\gamma \neq 0$.
    
    \item The following interval with modified variance:
    \begin{align}\label{eq:vEB CI}
        \tilde{\mathcal{I}} (y, X, \thetaveb) = \Big[\sumin q_i \hat{u}_i \pm c_{\alpha/2} \sqrt{\hat{\upsilon}_p + \gamma^2 J(\thetaveb)^\top V(\thetaveb)^{-1} J(\thetaveb)}\Big]
    \end{align}
    has asymptotic $1-\alpha$ average coverage under the oracle posterior $\PP_{\theta_0}$:
    $$\PP_{\theta_0}\Big(\sumin q_i \beta_i \in \tilde{\mathcal{I}} (y, X, \thetaveb) \mid X\Big) \to 1-\alpha.$$
\end{enumerate}
\end{cor}

\begin{remark}[Comparison with high SNR posteriors]
    As noted in the introduction, this is in sharp contrast with the ``merging'' properties of the EB posterior in high SNR settings where the EB posterior is a higher-order approximation of the oracle posterior \citep[c.f. Prop 3.1,][]{rizzelli2024empirical}. In other words, the cited result states that the EB credible region can serve as an oracle credible region, without the need for variance inflation as in part (b).
\end{remark}

\begin{remark}[Singularity]\label{rmk:singularity J}
    Recall \cref{rmk:singularity}, where we exhibit desirable properties under $\kappa(\theta_0) = 0$. A similar situation may happen when $J(\theta_0) = 0$, under which the CIs $\mathcal{I}$ and $\tilde{\mathcal{I}}$ in parts (a) and (b) asymptotically coincide. We will further discuss this in the simulations in \cref{sec:oracle posterior simulation}.
\end{remark}

\section{Simulation studies}\label{sec:simulations}

We conduct simulation studies to assess (i) the accuracy for recovering the prior hyperparameter $\theta_0$ and (ii) Bayesian uncertainty quantification for the regression coefficients $\beta$. {After presenting the main simulation results in Sections \ref{subsec:sim-1}-\ref{sec:oracle posterior simulation}, we also propose a way to improve the vEB estimator by debiasing in \cref{sec:debiasing}.}

\subsection{Setup}\label{subsec:sim-1}
We vary the number of coefficients as $p = 25, 50, 100, 200$, and take $n = p^2$. All reported accuracy values that will follow are averaged over 400 replications, under a single realization of the random design $X$ with Gaussian entries (see \cref{assmp:random design}). We consider the following priors, which allow both discrete, continuous, and mixed distributions:
\begin{enumerate}[(a)]
    \item $\textsf{Ber}(\pi)$ with $\pi_0 = 0.5$, %
    \item $\textsf{Spike-Slab}(\pi,\tau^2) = \pi \delta_0 + (1-\pi) N(0, \tau^2)$ with $\pi_0 = 0.5, \tau_0 = 1$, %
    \item $\textsf{location-GMM}(\theta_1,\theta_2) = \frac{1}{2} N(\theta_1, 1) + \frac{1}{2} N(\theta_2, 0.25)$ with $\theta_{1,0} = -1, \theta_{2,0} = 1$\footnote{Here, the MoM estimator suffers from a fundamental label switching issue, as the moment equations cannot identify $(\theta_1,\theta_2) = (-1,1)$ and $(1,-1)$, so we additionally imposed $\theta_1 < \theta_2$. Interestingly, the same issue persists for the adaptive Langevin sampler, where $(1,-1)$ is a local minimizer of the free energy. Thus, we used a warm initialization for the Langevin sampler.}, %
\end{enumerate}
The first is a discrete sparse prior, the second is a version of the spike and slab prior (where the prior parameters are treated as fixed), and the third is a non-sparse, bimodal Gaussian location mixture prior. Further details on implementation, computational time comparisons,  and additional simulation results are provided in \cref{sec:sim details}.

\subsection{Accuracy of prior recovery}
We compare our vEB estimator $\thetaveb$ with the following alternative EB methods: 
\begin{enumerate}[(a)]
    \item Method of Moments (extension of \cref{rmk:rate} by considering higher-order moments),
    \item Adaptive Langevin Diffusion from \cite{fan2025dynamical,fan2025dynamical2} for the continuous \textsf{location-GMM} prior. We did not implement this for other priors as the available implementation requires a Lebesgue density.
\end{enumerate}
\begin{table}[h!]
    \centering
    \begin{tabular}{c|cc|cc|ccc}
    \toprule
    Prior & \multicolumn{2}{c}{\textsf{Ber}} & \multicolumn{2}{c}{\textsf{Spike-Slab}} & \multicolumn{3}{c}{\textsf{location-GMM}} \\ %
    $p$ \textbackslash{} Algo & vEB & MoM & vEB & MoM & vEB & MoM    & Langevin \\
    \midrule
    25 & 0.042 & 0.048 & 0.173 & 0.391 & 0.339 & 0.318 & 0.713 \\ %
    50 & 0.023 & 0.023 & 0.125 & 0.373 & 0.168 & 0.237 & 0.400 \\ %
    100& 0.011 & 0.012 & 0.086 & 0.335 & 0.065 & 0.203 & 0.182 \\ %
    200& 0.006 & 0.006 & 0.059 & 0.299 & 0.028 & 0.187 & 0.089 \\ %
    \bottomrule
    \end{tabular}
    \caption{MSE $\|\thetaveb - \theta_0\|^2$ for varying $p$ and $n = p^2$ (averaged over 400 replications). Note that the scaling varies across different parametric priors.}
    \vspace{-5mm}
    \label{tab:mse_theta}
\end{table}
\noindent \cref{tab:mse_theta} reports the Mean Squared Error (MSE) for estimating the prior parameter $\theta_0$ under the aforementioned three prior families. The results validate that the vEB estimator is consistent in the sense that the MSE decreases as the dimension $p$ increases. In particular, for the \textsf{Ber} and \textsf{location-GMM} families, the MSE for the vEB estimator is roughly proportional to $1/p$, which aligns with the theoretical $\sqrt{p}$-convergence rate (here the \textsf{Spike-Slab} family is ruled out as the MSE aggregates two parameters which are of different scale). 
Compared to other estimators, the vEB estimator attains a smaller error, and empirically validates our asymptotic efficiency claim.

We also remark that the vEB estimator is computationally appealing. For estimating the \textsf{location-GMM} prior under $p = 200$, the vEB estimator only took 0.05 seconds to compute, whereas the MoM and Langevin estimators took 1.6, 3.0 seconds, respectively\footnote{Here, the vEB and MoM estimators are implemented in \texttt{R} whereas the Langevin estimator is implemented in \texttt{Python}. See \cref{sec:sim details} for details.}. Complete results for computation time is reported in \cref{tab:time} in the supplement, which illustrate that computing the vEB estimator took less than 0.1 seconds in all simulation settings and was always the fastest.

\subsection{Accuracy of oracle posterior inference}\label{sec:oracle posterior simulation}
Next, we consider posterior inference and evaluate the accuracy of the proposed vEB credible intervals (CIs) in part (b) of \cref{cor:coverage}. We consider $90\%$ credible intervals for $q^\top \beta$, where $\beta$ is a draw from the oracle posterior $\PP_{\theta_0}$. We consider two choices for $q$: 
\begin{enumerate}[(a)]
    \item $q^{\textsf{avg}} = \frac{1_p}{\sqrt{p}}$, which corresponds to the coordinate-wise average,
    \item $q^{\textsf{con}} = \frac{1}{\sqrt{p}}(1_{p/2}^\top, -1_{p/2}^\top)^\top$, which is a contrast vector such that $\langle q^{\textsf{avg}}, q^{\textsf{con}}\rangle = 0$.
\end{enumerate}
Note that each choice of $q$ yields the value $\gamma = \lim_{p \to \infty} \langle q, q^{\textsf{avg}}\rangle = 1, 0$, respectively.
We compare the vEB credible interval $\tilde{\mathcal{I}}(y,X,\thetaveb)$ in \eqref{eq:vEB CI} against its oracle analog that builds upon the oracle CLT in \cref{prop:oracle CLT}. While both theoretical coverage guarantees for the vEB and oracle CIs require bounded priors, we believe this is a technical requirement due to our proof technique and did not truncate the priors here. In addition to the coverage probability, we also report the average width of each credible interval. 

Tables \ref{tab:posterior mean credible} and \ref{tab:posterior contrast credible} respectively considers $q = q^{\textsf{avg}}$ and $q^{\textsf{con}}$. Looking at the oracle coverage guarantees, we see that the asymptotic CLT kicks in even for moderate sample sizes in the sense that it gives the correct coverage even when $p = 25$. In terms of the vEB interval in \cref{tab:posterior mean credible}, it is wider than the oracle interval for the Bernoulli and GMM prior. This is as expected by \cref{cor:coverage}, since we take $\gamma = 1$. Interestingly, the widths are comparable under the Spike-Slab prior, which is because $J(\thetaveb) \approx J(\theta_0) = 0$ and the width of the CI does not have to be increased for correct coverage (as mentioned in \cref{rmk:singularity J}). Similar to this, the widths of the EB-CI and oracle-CI in \cref{tab:posterior contrast credible} are always comparable, as we have $\gamma = 0$. In both tables, the vEB credible interval enjoys close to $90\%$ coverage and indicates accurate finite-sample performance. This suggests that the bounded prior requirement in our results in \cref{sec:consequences} may potentially be relaxed. Finally, we note that the vEB CIs' coverage exhibits a larger standard error compared to the oracle CIs. This is more apparent in \cref{tab:posterior mean credible}, as the vEB CIs only exhibit marginal coverage and no conditional coverage, theoretically.

\begin{table}[h!]
    \centering
    \resizebox{\textwidth}{!}{
    \begin{tabular}{c|cccc|cccc|cccc}
    \toprule
    Prior & \multicolumn{4}{c}{\textsf{Ber}} & \multicolumn{4}{c}{\textsf{Spike-Slab}} & \multicolumn{4}{c}{\textsf{location-GMM}}\\
    CI & \multicolumn{2}{c}{\textsf{Oracle}} & \multicolumn{2}{c}{\textsf{vEB}} & \multicolumn{2}{c}{\textsf{Oracle}} & \multicolumn{2}{c}{\textsf{vEB}} & \multicolumn{2}{c}{\textsf{Oracle}} & \multicolumn{2}{c}{\textsf{vEB}} \\
    $p$ \textbackslash{} Object & width & cover & width & cover & width & cover & width & cover & width & cover & width & cover 
    \\
    \midrule
    25 & 1.47 & 0.902 (0.021) & 3.12 & 0.860 (0.259) & 1.87 & 0.899 (0.007) & 1.72 & 0.783 (0.270) & 2.50 & 0.897 (0.006) & 3.32 & 0.899 (0.107) \\
    50 & 1.47 & 0.901 (0.020) & 3.20 & 0.873 (0.231) & 1.87 & 0.900 (0.006) & 1.77 & 0.819 (0.220) & 2.52 & 0.898 (0.006) & 3.29 & 0.907 (0.088) \\
    100& 1.47 & 0.900 (0.018) & 3.22 & 0.890 (0.218) & 1.87 & 0.900 (0.006) & 1.82 & 0.862 (0.119) & 2.52 & 0.900 (0.006) & 3.26 & 0.894 (0.100) \\
    200& 1.47 & 0.901 (0.015) & 3.24 & 0.907 (0.184) & 1.87 & 0.900 (0.006) & 1.87 & 0.885 (0.059) & 2.52 & 0.900 (0.006) & 3.25 & 0.901 (0.092) \\
    \bottomrule
    \end{tabular}
    }
    \caption{Average credible interval (CI) width and average coverage probability (standard error) under the oracle posterior, where $q^{\textsf{avg}} = \frac{1_p}{\sqrt{p}}$ is the \emph{coordinate-wise average}.}
    \label{tab:posterior mean credible}
    \vspace{-7mm}
\end{table}

\begin{table}[h!]
    \centering
    \resizebox{\textwidth}{!}{
    \begin{tabular}{c|cccc|cccc|cccc}
    \toprule
     Prior & \multicolumn{4}{c}{\textsf{Ber}} & \multicolumn{4}{c}{\textsf{Spike-Slab}} & \multicolumn{4}{c}{\textsf{location-GMM}}\\
    CI & \multicolumn{2}{c}{\textsf{Oracle}} & \multicolumn{2}{c}{\textsf{vEB}} & \multicolumn{2}{c}{\textsf{Oracle}} & \multicolumn{2}{c}{\textsf{vEB}} & \multicolumn{2}{c}{\textsf{Oracle}} & \multicolumn{2}{c}{\textsf{vEB}} \\
    $p$ \textbackslash{} Object & width & cover & width & cover & width & cover & width & cover & width & cover & width & cover \\
    \midrule
    25 & 1.47 & 0.899 (0.022) & 1.27 & 0.822 (0.169) &  1.87 & 0.892 (0.006) & 1.72 & 0.774 (0.268) & 2.50 & 0.890 (0.006) & 2.47 & 0.874 (0.041) \\
    50 & 1.47 & 0.900 (0.020) & 1.39 & 0.876 (0.060) & 1.87 & 0.899 (0.006) & 1.77 & 0.815 (0.220) & 2.52 & 0.899 (0.006) & 2.46 & 0.884 (0.032) \\
    100& 1.47 & 0.900 (0.017) & 1.44 & 0.891 (0.024) & 1.87 & 0.899 (0.006) & 1.82 & 0.860 (0.120) & 2.52 & 0.900 (0.006) & 2.49 & 0.893 (0.021) \\
    200& 1.47 & 0.900 (0.016) & 1.45 & 0.898 (0.016) & 1.87 & 0.899 (0.006) & 1.87 & 0.885 (0.060) & 2.52 & 0.900 (0.006) & 2.50 & 0.893 (0.017) \\
    \bottomrule
    \end{tabular}
    }
    \caption{Average credible interval (CI) width and average coverage probability (standard error) under the oracle posterior, where $q^{\textsf{con}} = \frac{1}{\sqrt{p}}(1_{p/2}^\top, -1_{p/2}^\top)^\top$ is the \emph{contrast vector}.}
    \label{tab:posterior contrast credible}
    \vspace{-7mm}
\end{table}

\subsection{Improving the naive vEB estimator by debiasing}\label{sec:debiasing}
As the vEB estimator exhibits a sharp phase transition at $p\sim n^{2/3}$, a natural question arises: can one modify the vEB estimator to recover a $p^{-1/2}$ convergence rate beyond $p\sim n^{2/3}$? To this end, one natural strategy is debiasing, especially noting that the limiting distribution in part (c) of \cref{thm:upper bound} is essentially a degenerate bias. The following proposition shows that debiasing the vEB estimator can allow the normal limit \eqref{eq:normal limit} to extend to higher dimensions.
\begin{prop}\label{prop:debiasing}
    Assume the usual assumptions \ref{assmp:random design}--\ref{assmp:prior technical}, and consider the debiased estimator $$\tilde{\theta}^{\mathsf{d}} := \thetaveb - \frac{p}{n} V(\thetaveb)^{-1} \kappa(\thetaveb).$$ Then, the limit \eqref{eq:normal limit} holds for $\tilde{\theta}^{\mathsf{d}}$ as long as $p \ll n^{3/4}$. %
\end{prop}

The following \cref{fig:debias} considers the one-dimensional symmetric Gaussian mixture prior $\mu_\theta \equiv \frac{1}{2} N(\theta, 1) + \frac{1}{2} N(-\theta, 1)$, and illustrates that bias correction leads to better finite sample performance, which is more apparent for larger $p$. Our debiasing formula arises from correcting the second-order expansion of the likelihood. More specifically, noting that the slower rate of convergence of $\thetaveb$ arises from the second term in the RHS of \eqref{eq:mu_p approximation}, we correct for this bias by replacing $\kappa(\theta_0)$ with $\kappa(\thetaveb)$. This is in contrast to the more common first-order debiasing in the frequentist literature for estimating high-dimensional regression coefficients $\beta$ \citep{Geer2014debiasing,Bellec2022debiasing,Bellec2021,li2023spectrum}. We conjecture that it may be possible to extend the limit \eqref{eq:normal limit} to higher dimensions $p \gg n^{3/4}$ by additionally correcting for higher-order error terms in \eqref{eq:mu_p approximation}.

\begin{figure}[h!]
    \centering
    \includegraphics[width=0.55\linewidth]{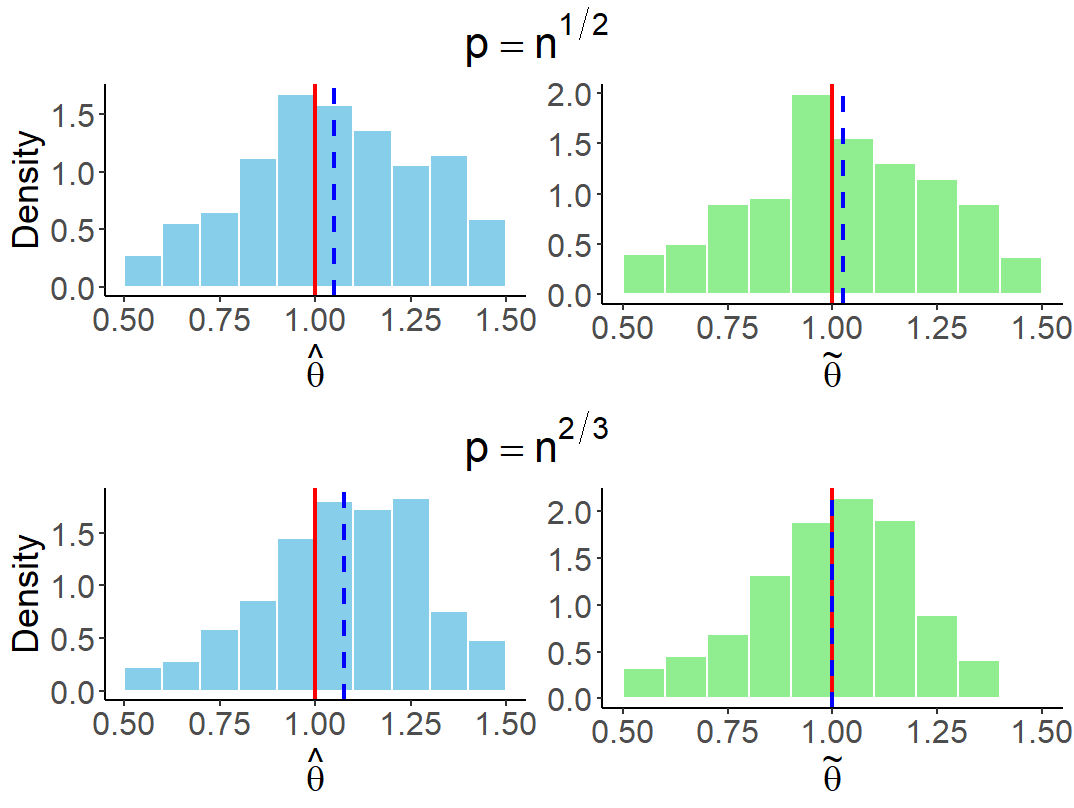}

    \includegraphics[width=0.4\linewidth]{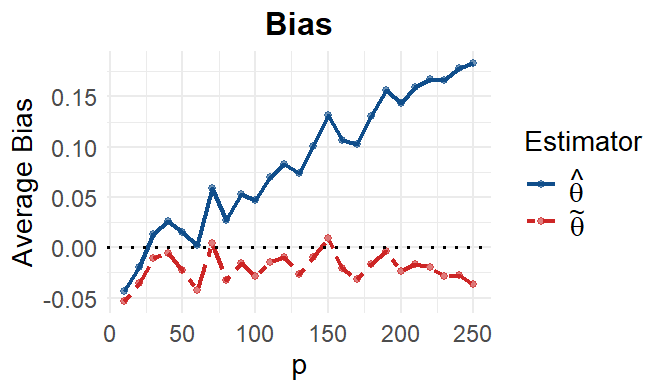}
    \quad
    \includegraphics[width=0.4\linewidth]{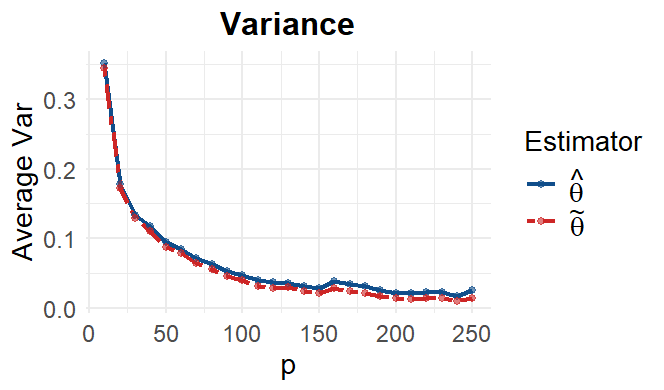}

    \caption{\textbf{(Top row)} Histogram of the estimators $\thetaveb$ and $\tilde{\theta}^{\mathsf{d}}$, where the prior is the symmetric Gaussian mixture $\frac{1}{2}N(\theta,1)+\frac{1}{2}N(-\theta,1)$. We fix $n = 1000$, and consider $p=\sqrt{n}$ and $n^{2/3}$ in each row. The solid red line shows the true value $\theta_0 = 1$ and the dotted blue line denotes the average across 400 replications (under a fixed realization of $X$). We see that $\thetaveb$ has a positive bias and is right-skewed whereas the debiased estimator $\tilde{\theta}^{\mathsf{d}}$ has smaller bias. \textbf{(Bottom row)} Average bias and variance of $\thetaveb$ and $\tilde{\theta}^{\mathsf{d}}$, where $p$ varies from $10$ to $250$ and $n=1000$ is fixed. As $p$ increases, the bias of $\thetahat$ also increases. %
    }
    \label{fig:debias}
\end{figure}

\section{Discussion}\label{sec:conclusion}
We discuss avenues for future research.
\begin{enumerate}[(a),itemsep=1em]
    \item \emph{Technical conditions.} \quad It would be important to improve upon various technical conditions. One immediate direction involves our design matrix assumption (see \cref{assmp:random design}), which requires entry-wise independence. While this is a natural starting point for theoretical understanding, it would be important to further relax this assumption to design matrices with less randomness, such as rotationally invariant matrices or fixed designs with correlated columns that were considered in the AMP literature \citep{li2023random,dudeja2024spectral}. We believe that the definition of the vEB estimator (see \eqref{eq:est def}) as well as the limit distributions and debiasing formula should change appropriately. %
    Additionally, some parts of our analysis require a bounded prior (see \cref{assmp:cpt support}). %
    As this restriction is inevitable due to limitations in existing literature for RFIMs \citep{lee2025clt}, it would be important to further understand RFIMs with relaxed conditions on the base measure.

    \item \emph{Singular or hierarchical priors.} \quad One common choice for practical empirical Bayes procedures is to use finite (Gaussian) mixtures to model the prior $\mu_\theta$. While our results allow such mixtures, we require standard regularity assumptions (see \cref{assmp:prior}), which rules out singular models such as overfitted mixtures. 
    Indeed, it is well-known that overfitted mixtures have a slower $p^{1/4}$-rate of convergence for parameter estimation, even under the non-hierarchical setting with $p$ i.i.d. observations \citep{ho2016strong}. It would be interesting to explore whether our phase transition for the vEB estimator persists under such singular priors. 
    On a related note, it would be interesting to view mixture priors as another layer of hierarchy and consider inference for each mixture component (for example, suppose $\beta_i \mid \Delta_i \sim P_{\Delta_i}$ and $\Delta_i \sim \textsf{Categorical}(\pi)$, where $P_1,\ldots,P_M$ denotes each mixture distribution, $\Delta_i$ denotes the mixture component, and $\pi$ is a $M$-dimensional weight vector). Related hierarchical modeling approaches were recently studied under the classical normal means EB problem \citep{kim2025eb} as well as oracle structured variational inference in regression \citep{sheng2025theory}.
    
    \item \emph{Going beyond naive MF.} \quad Finally, it would be interesting to utilize higher-order variational approximations (as opposed to the naive Mean-Field) such as the TAP free energy in statistical physics \citep{thouless1977solution,krzakala2014variational} to estimate the prior.
    It is known that the TAP free energy can capture the leading-order asymptotics of the log-partition function under a proportional scaling \citep{qiu2023tap}, but its tightness compared to NMF has not been investigated when $p = o(n)$. It would be interesting whether using the TAP free energy can improve prior recovery, and whether it has connections with the debiasing strategy mentioned in \cref{prop:debiasing}. %
\end{enumerate}

\section*{Acknowledgments}
The authors thank Sumit Mukherjee and Bodhisattva Sen for many helpful discussions throughout this work, Shunan Sheng and Bohan Wu for comments on the draft, and Zhou Fan for sharing the adaptive Langevin implementation in \cite{fan2025dynamical}.

\bibliographystyle{imsart-number}
\bibliography{reference}

@article{kingma2013auto,
  title={Auto-encoding variational bayes},
  author={Kingma, Diederik P and Welling, Max},
  journal={arXiv preprint arXiv:1312.6114},
  year={2013}
}

@article{castillo2020spike,
  title={Spike and slab empirical Bayes sparse credible sets},
  author={Castillo, Isma{\"e}l and Szab{\'o}, Botond},
  journal={Bernoulli},
  volume={26},
  number={1},
  pages={127--158},
  year={2020},
  publisher={JSTOR}
}

@article{fan2025dynamical,
  title={Dynamical mean-field analysis of adaptive Langevin diffusions: Propagation-of-chaos and convergence of the linear response},
  author={Fan, Zhou and Ko, Justin and Loureiro, Bruno and Lu, Yue M and Shen, Yandi},
  journal={arXiv preprint arXiv:2504.15556},
  year={2025}
}

@article{zhang2018estimation,
  title={Estimation of complex effect-size distributions using summary-level statistics from genome-wide association studies across 32 complex traits},
  author={Zhang, Yan and Qi, Guanghao and Park, Ju-Hyun and Chatterjee, Nilanjan},
  journal={Nature genetics},
  volume={50},
  number={9},
  pages={1318--1326},
  year={2018},
  publisher={Nature Publishing Group US New York}
}

@article{o2021distribution,
  title={The distribution of common-variant effect sizes},
  author={O’Connor, Luke J},
  journal={Nature genetics},
  volume={53},
  number={8},
  pages={1243--1249},
  year={2021},
  publisher={Nature Publishing Group US New York}
}

@article{amaro1991fluctuations,
  title={Fluctuations in the Curie-Weiss version of the random field Ising model},
  author={Amaro de Matos, JMG and Perez, J Fernando},
  journal={Journal of Statistical Physics},
  volume={62},
  number={3},
  pages={587--608},
  year={1991},
  publisher={Springer}
}

@incollection{nattermann1998theory,
  title={Theory of the random field Ising model},
  author={Nattermann, Thomas},
  booktitle={Spin glasses and random fields},
  pages={277--298},
  year={1998},
  publisher={World Scientific}
}

@article{bhattacharya2024ldp,
  title={LDP for inhomogeneous U-statistics},
  author={Bhattacharya, Sohom and Deb, Nabarun and Mukherjee, Sumit},
  journal={The Annals of Applied Probability},
  volume={34},
  number={6},
  pages={5769--5808},
  year={2024},
  publisher={Institute of Mathematical Statistics}
}

@article{deb2025pivotal,
  title={Pivotal CLTs for Pseudolikelihood via Conditional Centering in Dependent Random Fields},
  author={Deb, Nabarun},
  journal={arXiv preprint arXiv:2510.04972},
  year={2025}
}

@article{zhou2021fast,
  title={A fast and robust Bayesian nonparametric method for prediction of complex traits using summary statistics},
  author={Zhou, Geyu and Zhao, Hongyu},
  journal={PLoS genetics},
  volume={17},
  number={7},
  pages={e1009697},
  year={2021},
  publisher={Public Library of Science}
}

@article{petrone2014bayes,
  title={Bayes and empirical Bayes: do they merge?},
  author={Petrone, Sonia and Rousseau, Judith and Scricciolo, Catia},
  journal={Biometrika},
  volume={101},
  number={2},
  pages={285--302},
  year={2014},
  publisher={Oxford University Press}
}

@incollection{james1992estimation,
  title={Estimation with quadratic loss},
  author={James, William and Stein, Charles},
  booktitle={Proc. Fourth Berkeley Symp. Math. Statist. Prob},
  pages={361--379},
  year={1961},
  publisher={Springer}
}

@article{soloff2025multivariate,
  title={Multivariate, heteroscedastic empirical Bayes via nonparametric maximum likelihood},
  author={Soloff, Jake A and Guntuboyina, Adityanand and Sen, Bodhisattva},
  journal={Journal of the Royal Statistical Society Series B: Statistical Methodology},
  volume={87},
  number={1},
  pages={1--32},
  year={2025},
  publisher={Oxford University Press UK}
}

@article{jiang2009general,
  title={General maximum likelihood empirical Bayes estimation of normal means},
  author={Jiang, Wenhua and Zhang, Cun-Hui},
  year={2009}
}

@incollection{robbins1956empirical,
  title={{An empirical Bayes approach to statistics}},
  author={Robbins, Herbert E},
  booktitle={Breakthroughs in Statistics: Foundations and basic theory},
  pages={157–163},
  year={1956},
  publisher={Springer}
}

@article{ho2016strong,
  title={On strong identifiability and convergence rates of parameter estimation in finite mixtures},
  author={Ho, Nhat and Nguyen, XuanLong},
  year={2016}
}

@article {Deb2024detecting,
    AUTHOR = {Deb, Nabarun and Mukherjee, Rajarshi and Mukherjee, Sumit and
              Yuan, Ming},
     TITLE = {Detecting structured signals in {I}sing models},
   JOURNAL = {Ann. Appl. Probab.},
  FJOURNAL = {The Annals of Applied Probability},
    VOLUME = {34},
      YEAR = {2024},
    NUMBER = {1A},
     PAGES = {1--45},
      ISSN = {1050-5164,2168-8737},
   MRCLASS = {62G10 (62C20 62G20 82B20)},
  MRNUMBER = {4696272},
       DOI = {10.1214/23-aap1929},
       URL = {https://doi.org/10.1214/23-aap1929},
}

@article{Bellec2021,
author = {Pierre C. Bellec and Cun-Hui Zhang},
title = {{Debiasing convex regularized estimators and interval estimation in linear models}},
volume = {51},
journal = {The Annals of Statistics},
number = {2},
publisher = {Institute of Mathematical Statistics},
pages = {391 -- 436},
year = {2023},
doi = {10.1214/22-AOS2243},
URL = {https://doi.org/10.1214/22-AOS2243}
}

@article{mukherjee2022variational,
  title={Variational Inference in high-dimensional linear regression},
  author={Mukherjee, Sumit and Sen, Subhabrata},
  journal={The Journal of Machine Learning Research},
  volume={23},
  number={1},
  pages={13703--13758},
  year={2022},
  publisher={JMLRORG}
}

@article{chatterjee2011nonnormal,
  title={Nonnormal approximation by Stein’s method of exchangeable pairs with application to the Curie--Weiss model},
  author={Chatterjee, Sourav and Shao, Qi-Man},
  journal={The Annals of Applied Probability},
  volume={21},
  number={2},
  pages={464--483},
  year={2011},
  publisher={Institute of Mathematical Statistics}
}

@article{deb2020fluctuations,
  title={Fluctuations in mean-field Ising models},
  author={Deb, Nabarun and Mukherjee, Sumit},
  journal={The Annals of Applied Probability},
  volume={33},
  number={3},
  pages={1961--2003},
  year={2023},
  publisher={Institute of Mathematical Statistics}
}

@article{barbier2020mutual,
  title={Mutual information and optimality of approximate message-passing in random linear estimation},
  author={Barbier, Jean and Macris, Nicolas and Dia, Mohamad and Krzakala, Florent},
  journal={IEEE Transactions on Information Theory},
  volume={66},
  number={7},
  pages={4270--4303},
  year={2020},
  publisher={IEEE}
}

@article{qiu2024sub,
  title={Sub-optimality of the Naive Mean Field approximation for proportional high-dimensional Linear Regression},
  author={Qiu, Jiaze},
  journal={Advances in Neural Information Processing Systems},
  volume={36},
  year={2024}
}

@article{celentano2023mean,
  title={Mean-field variational inference with the TAP free energy: Geometric and statistical properties in linear models},
  author={Celentano, Michael and Fan, Zhou and Lin, Licong and Mei, Song},
  journal={arXiv preprint arXiv:2311.08442},
  year={2023}
}

@article{bhattacharya2023gibbs,
  title={Gibbs Measures with Multilinear Forms},
  author={Bhattacharya, Sohom and Deb, Nabarun and Mukherjee, Sumit},
  journal={arXiv preprint arXiv:2307.14600},
  year={2023}
}

@article{wainwright2008graphical,
  title={Graphical models, exponential families, and variational inference},
  author={Wainwright, Martin J and Jordan, Michael I and others},
  journal={Foundations and Trends{\textregistered} in Machine Learning},
  volume={1},
  number={1--2},
  pages={1--305},
  year={2008},
  publisher={Now Publishers, Inc.}
}

@book{vershynin2018high,
  title={High-dimensional probability: An introduction with applications in data science},
  author={Vershynin, Roman},
  volume={47},
  year={2018},
  publisher={Cambridge university press}
}

@article{van2014probability,
  title={Probability in high dimension},
  author={Van Handel, Ramon},
  journal={Lecture Notes (Princeton University)},
  year={2014}
}

@article{chatterjee2009fluctuations,
  title={Fluctuations of eigenvalues and second order Poincar{\'e} inequalities},
  author={Chatterjee, Sourav},
  journal={Probability Theory and Related Fields},
  volume={143},
  number={1},
  pages={1--40},
  year={2009},
  publisher={Springer}
}

@article{yan2020nonlinear,
author = {Jun Yan},
title = {{Nonlinear large deviations: Beyond the hypercube}},
volume = {30},
journal = {The Annals of Applied Probability},
number = {2},
publisher = {Institute of Mathematical Statistics},
pages = {812 -- 846},
year = {2020},
doi = {10.1214/19-AAP1516},
URL = {https://doi.org/10.1214/19-AAP1516}
}

@article{qiu2023tap,
  title={The TAP free energy for high-dimensional linear regression},
  author={Qiu, Jiaze and Sen, Subhabrata},
  journal={The Annals of Applied Probability},
  volume={33},
  number={4},
  pages={2643--2680},
  year={2023},
  publisher={Institute of Mathematical Statistics}
}

@article{mukherjee2023mean,
  title={A mean field approach to empirical bayes estimation in high-dimensional linear regression},
  author={Mukherjee, Sumit and Sen, Bodhisattva and Sen, Subhabrata},
  journal={arXiv preprint arXiv:2309.16843},
  year={2023}
}

@Manual{isingsampler,
    title = {IsingSampler: Sampling Methods and Distribution Functions
      for the Ising Model},
    author = {Sacha Epskamp},
    year = {2025},
    note = {R package version 0.2.4},
    url = {https://github.com/sachaepskamp/isingsampler},
  }

@article{fan2025dynamical2,
  title={Dynamical mean-field analysis of adaptive Langevin diffusions: Replica-symmetric fixed point and empirical Bayes},
  author={Fan, Zhou and Ko, Justin and Loureiro, Bruno and Lu, Yue M and Shen, Yandi},
  journal={arXiv preprint arXiv:2504.15558},
  year={2025}
}

@inproceedings{kuntz2023particle,
  title={Particle algorithms for maximum likelihood training of latent variable models},
  author={Kuntz, Juan and Lim, Jen Ning and Johansen, Adam M},
  booktitle={International Conference on Artificial Intelligence and Statistics},
  pages={5134--5180},
  year={2023},
  organization={PMLR}
}

@incollection{bakry2006diffusions,
  title={Diffusions hypercontractives},
  author={Bakry, Dominique and {\'E}mery, Michel},
  booktitle={S{\'e}minaire de Probabilit{\'e}s XIX 1983/84: Proceedings},
  pages={177--206},
  year={2006},
  publisher={Springer}
}

@article{johnstone2004needles,
author = {Iain M. Johnstone and Bernard W. Silverman},
title = {{Needles and straw in haystacks: Empirical Bayes estimates of possibly sparse sequences}},
volume = {32},
journal = {The Annals of Statistics},
number = {4},
publisher = {Institute of Mathematical Statistics},
pages = {1594 -- 1649},
year = {2004}
}

@article{rousseau2020asymptotic,
  title={Asymptotic frequentist coverage properties of Bayesian credible sets for sieve priors},
  author={Rousseau, Judith and Szabo, Botond},
  journal={The Annals of Statistics},
  volume={48},
  number={4},
  pages={2155--2179},
  year={2020},
  publisher={JSTOR}
}

@article{zhang2020convergence,
  title={Convergence Rates of Empirical Bayes Posterior Distributions: A Variational Perspective},
  author={Zhang, Fengshuo and Gao, Chao},
  journal={arXiv preprint arXiv:2009.03969},
  year={2020}
}

@article{cho2021gaussian,
  title={Gaussian variational estimation for multidimensional item response theory},
  author={Cho, April E and Wang, Chun and Zhang, Xue and Xu, Gongjun},
  journal={British Journal of Mathematical and Statistical Psychology},
  volume={74},
  pages={52--85},
  year={2021},
  publisher={Wiley Online Library}
}

@article{saenz2025characterizing,
  title={Characterizing Finite-Dimensional Posterior Marginals in High-Dimensional GLMs via Leave-One-Out},
  author={S{\'a}enz, Manuel and Sur, Pragya},
  journal={arXiv preprint arXiv:2601.00091},
  year={2025}
}

@article{lee2025clt,
  title={Fluctuations in random field {I}sing models},
  author={Lee, Seunghyun and Deb, Nabarun and  Mukherjee, Sumit},
  journal={Annals of Applied Probability},
  pages={just accepted},
  year={2026}
}

@article{blei2003latent,
  title={Latent Dirichlet Allocation},
  author={Blei, DM and Ng, AY and Jordan, MI},
  journal={Journal of Machine Learning Research},
  volume={3},
  year={2003}
}

@article{sheng2025theory,
  title={Theory and computation for structured variational inference},
  author={Sheng, Shunan and Wu, Bohan and Zhu, Bennett and Chewi, Sinho and Pooladian, Aram-Alexandre},
  journal={arXiv preprint arXiv:2511.09897},
  year={2025}
}

@article{zhong2022empirical,
  title={Empirical Bayes PCA in high dimensions},
  author={Zhong, Xinyi and Su, Chang and Fan, Zhou},
  journal={Journal of the Royal Statistical Society Series B: Statistical Methodology},
  volume={84},
  number={3},
  pages={853--878},
  year={2022},
  publisher={Oxford University Press}
}

@article{li2023random,
  title={Random linear estimation with rotationally-invariant designs: Asymptotics at high temperature},
  author={Li, Yufan and Fan, Zhou and Sen, Subhabrata and Wu, Yihong},
  journal={IEEE Transactions on Information Theory},
  volume={70},
  number={3},
  pages={2118--2153},
  year={2023},
  publisher={IEEE}
}

@article{dudeja2024spectral,
  title={Spectral universality in regularized linear regression with nearly deterministic sensing matrices},
  author={Dudeja, Rishabh and Sen, Subhabrata and Lu, Yue M},
  journal={IEEE Transactions on Information Theory},
  year={2024},
  publisher={IEEE}
}

@article{stephens2017false,
  title={False discovery rates: a new deal},
  author={Stephens, Matthew},
  journal={Biostatistics},
  volume={18},
  number={2},
  pages={275--294},
  year={2017},
  publisher={Oxford University Press}
}

@article{wang2021empirical,
  title={Empirical bayes matrix factorization},
  author={Wang, Wei and Stephens, Matthew},
  journal={Journal of Machine Learning Research},
  volume={22},
  number={120},
  pages={1--40},
  year={2021}
}

@article{lee2025bayesregression,
  title={{CLT in high-dimensional Bayesian linear regression with low SNR}},
  author={Lee, Seunghyun and Deb, Nabarun and Mukherjee, Sumit},
  journal={arXiv preprint arXiv:2507.23285},
  year={2025}
}

@article{rousseau2017asymptotic,
  title={Asymptotic behaviour of the empirical Bayes posteriors associated to maximum marginal likelihood estimator},
  author={Rousseau, Judith and Szabo, Botond},
  year={2017}
}

@article{szabo2013empirical,
    author = {B. T. Szab{\'o} and A. W. van der Vaart and J. H. van Zanten},
    title = {{Empirical Bayes scaling of Gaussian priors in the white noise model}},
    volume = {7},
    journal = {Electronic Journal of Statistics},
    number = {none},
    publisher = {Institute of Mathematical Statistics and Bernoulli Society},
    pages = {991 -- 1018},
    year = {2013},
    doi = {10.1214/13-EJS798}
}

@article{Bellec2022debiasing,
author = {Pierre C. Bellec and Cun-Hui Zhang},
title = {{De-biasing the lasso with degrees-of-freedom adjustment}},
volume = {28},
journal = {Bernoulli},
number = {2},
publisher = {Bernoulli Society for Mathematical Statistics and Probability},
pages = {713 -- 743},
year = {2022},
doi = {10.3150/21-BEJ1348},
URL = {https://doi.org/10.3150/21-BEJ1348}
}

@article{li2023spectrum,
  title={Spectrum-aware debiasing: A modern inference framework with applications to principal components regression},
  author={Li, Yufan and Sur, Pragya},
  journal={arXiv preprint arXiv:2309.07810},
  year={2023}
}

@article{berk1966limiting,
  title={Limiting behavior of posterior distributions when the model is incorrect},
  author={Berk, Robert H},
  journal={The Annals of Mathematical Statistics},
  volume={37},
  number={1},
  pages={51--58},
  year={1966},
  publisher={Institute of Mathematical Statistics}
}

@article{kim2025eb,
    author = {Kim, Taehyun and Sen, Bodhisattva},
    title = {Nonparametric Empirical {B}ayes Estimation and Inference via Hierarchical Model},
    journal = {preprint},
    year = {2025}
}

@article{white1982maximum,
  title={Maximum likelihood estimation of misspecified models},
  author={White, Halbert},
  journal={Econometrica: Journal of the econometric society},
  pages={1--25},
  year={1982},
  publisher={JSTOR}
}

@article{Geer2014debiasing,
author = {Sara van de Geer and Peter B{\"u}hlmann and Ya’acov Ritov and Ruben Dezeure},
title = {{On asymptotically optimal confidence regions and tests for high-dimensional models}},
volume = {42},
journal = {The Annals of Statistics},
number = {3},
publisher = {Institute of Mathematical Statistics},
pages = {1166 -- 1202},
year = {2014},
doi = {10.1214/14-AOS1221},
URL = {https://doi.org/10.1214/14-AOS1221}
}

@article{thouless1977solution,
  title={Solution of `solvable model of a spin glass'},
  author={Thouless, David J and Anderson, Philip W and Palmer, Robert G},
  journal={Philosophical Magazine},
  volume={35},
  number={3},
  pages={593--601},
  year={1977},
  publisher={Taylor \& Francis}
}

@article{kleijn2012bernstein,
  title={The Bernstein-von-Mises theorem under misspecification},
  author={Kleijn, Bas JK and Van der Vaart, Aad W},
  year={2012}
}

@inproceedings{krzakala2014variational,
  title={Variational free energies for compressed sensing},
  author={Krzakala, Florent and Manoel, Andre and Tramel, Eric W and Zdeborov{\'a}, Lenka},
  booktitle={2014 IEEE International Symposium on Information Theory},
  pages={1499--1503},
  year={2014},
  organization={IEEE}
}

@article{sheng2025stability,
  title={Stability of Mean-Field Variational Inference},
  author={Sheng, Shunan and Wu, Bohan and Gonz{\'a}lez-Sanz, Alberto and Nutz, Marcel},
  journal={arXiv preprint arXiv:2506.07856},
  year={2025}
}

@article{lacker2024mean,
  title={Mean field approximations via log-concavity},
  author={Lacker, Daniel and Mukherjee, Sumit and Yeung, Lane Chun},
  journal={International Mathematics Research Notices},
  volume={2024},
  number={7},
  pages={6008--6042},
  year={2024},
  publisher={Oxford University Press}
}

@article{rizzelli2024empirical,
  title={Empirical Bayes in Bayesian learning: understanding a common practice},
  author={Rizzelli, Stefano and Rousseau, Judith and Petrone, Sonia},
  journal={arXiv preprint arXiv:2402.19036},
  year={2024}
}

@article{carbonetto2012vi,
author = {Peter Carbonetto and Matthew  Stephens},
title = {{Scalable Variational Inference for Bayesian Variable Selection in Regression, and Its Accuracy in Genetic Association Studies}},
volume = {7},
journal = {Bayesian Analysis},
number = {1},
publisher = {International Society for Bayesian Analysis},
pages = {73 -- 108},
year = {2012}
}

@article{kim2024flexible,
  title={A flexible empirical Bayes approach to multiple linear regression and connections with penalized regression},
  author={Kim, Youngseok and Wang, Wei and Carbonetto, Peter and Stephens, Matthew},
  journal={Journal of Machine Learning Research},
  volume={25},
  number={185},
  pages={1--59},
  year={2024}
}

@book{lehmann2006theory,
  title={Theory of point estimation},
  author={Lehmann, Erich L and Casella, George},
  year={2006},
  publisher={Springer Science \& Business Media}
}

@book{van2000asymptotic,
  title={Asymptotic statistics},
  author={Van der Vaart, Aad W},
  volume={3},
  year={2000},
  publisher={Cambridge university press}
}

@article{fan2023gradient,
  title={Gradient flows for empirical Bayes in high-dimensional linear models},
  author={Fan, Zhou and Guan, Leying and Shen, Yandi and Wu, Yihong},
  journal={arXiv preprint arXiv:2312.12708},
  year={2023}
}

\newpage 

\setcounter{page}{1}

\begin{appendix}
\noindent \textbf{Appendix}

\setcounter{figure}{0}
\setcounter{table}{0}

\renewcommand{\thefigure}{\thesection.\arabic{figure}}
\renewcommand{\thetable}{\thesection.\arabic{table}}
\vspace{2mm}
\noindent The Appendix is organized as follows. We prove all main theorems in \cref{sec:pf main results}. We prove main lemmas (including those in the main text) in \cref{sec:pf lemmas}, and prove additional technical results in \cref{sec:proof technical lemmas}. \cref{sec:sim details} discusses the simulation setups from the main paper in more detail and also presents additional numerical  results.

\section{Proof of main theorems}\label{sec:pf main results}

\subsection{Proof of \cref{thm:upper bound} and \cref{cor:misspecified}}
We begin with a modified version of Wald's consistency result (see Theorem 5.14 in \cite{van2000asymptotic}), where we relax the independent sample requirement by directly assuming the pointwise LLNs that are used in the proof. This will enable us to establish the consistency of $\thetaveb$. 
\begin{lemma}[Wald's consistency theorem]\label{lem:Wald}
    Let $g:\R \times \Theta \to \R$ be a function such that the map $\theta \to g(v,\theta)$ is upper-semicontinuous for all $v$. Let $\{V_i\}_{i\ge 1}$ and $V_{\infty}$ be random variables such that satisfy the LLN
    $$\frac{1}{p} \sumin g(V_i, \theta) \xp \EE g(V_{\infty},\theta), \quad \forall \theta\, \in \Theta.$$
    Assume that the map $\theta \to \EE g(V_{\infty},\theta)$ is uniquely maximized at $\theta_0$, and that for every sufficiently small ball $U \in \Theta$ we have
    $$\EE \sup_{\theta \in U} g(V_{\infty},\theta) < \infty.$$
    Define $\mathcal{M}_p(\theta) := \frac{1}{p} \sumin g(V_i,\theta)$. Let $\hat{\theta}_p$ be any estimator such that $\mathcal{M}_p(\thetahat) \ge \mathcal{M}_p(\theta_0) -o_P(1)$. Then, $\thetahat \xp \theta_0$.
\end{lemma}

To apply \cref{lem:Wald}, we require several additional lemmas  whose proofs are deferred to \cref{sec:pfupbdlm}. Recall the notation $B_{\theta}$, $W_{\theta}$, and $B_{t,d,\theta}$ from \cref{def:information}. Also recall the definition of $w$ from \eqref{eq:notation}. Our first Lemma will be used to prove that the limit of $p^{-1}\sum_{i=1}^p F_i(\theta)$ (see \eqref{eq:est def}) equals $\Flm(\theta)$ where 
\begin{align}\label{eq:flm}
\Flm(\theta) := \EE \Big[\log \Big(\int e^{W_{\theta_0} b - \frac{d_0 b^2}{2}} d \mu_\theta(b) \Big) \Big].
\end{align}
To wit, let $\mcp(r)$ denote the set of all functions $h:\R \to \R$ with sub-polynomial growth $|h(z)| \lesssim 1+|z|^r$ for all $z\in\R$. Here $\lesssim$ hides constants free of $z$. 

\begin{lemma}[LLN]\label{lem:LLN}
Suppose R5 from Assumption \ref{assmp:prior technical} holds.
\begin{enumerate}[(a)]
    \item Let $g:\R \to \R$ be a $\mathcal{C}^1$ function with $g' \in \mcp(r)$ for some $r\ge 0$ and $\EE g(W_{\theta_0})<\infty$. Then, the following LLN holds:
    $$\frac{1}{p} \sumin g(w_i) \mid X \xp \EE g(W_{\theta_0}).$$
    \item Additionally assuming R6, for any function $h \in \mcp(r)$, we have 
    \begin{align*}
        \frac{1}{p} \sumin \EE h(B_{w_i,d_i,\theta_0}) \mid X &\xp \EE h(B_{\theta_0}), ~~ \frac{1}{p} \sumin \big[\EE h(B_{w_i,d_i,\theta_0})\big]^2 \mid X &\xp \EE \Big[ \EE \big(h(B_{\theta_0}) \mid W_{\theta_0}\big) \Big]^2.
    \end{align*}
\end{enumerate}
\end{lemma}

The next result shows how $\Flm(\cdot)$ identifies $\theta_0$. 
\begin{lemma}\label{lem:unique maximizer}
      Under the regularity condition R0 from \cref{assmp:prior}, $\Flm(\cdot)$ is uniquely maximized at $\theta = \theta_0$. 
\end{lemma}

To go beyond consistency and establish limiting distributions, we use the following key lemma.
\begin{lemma}[CLT]\label{lem:linear clt}
    Suppose Assumptions \ref{assmp:random design}--\ref{assmp:prior technical} hold.

    \begin{enumerate}[(a)]
        \item When $\frac{p^{3/2}}{n} \to \delta \in [0, \infty)$, we have
    \begin{align*}
        \frac{1}{\sqrt{p}} \sumin \nabla F_i(\theta_0) \mid X \xd N\big(\delta \kappa(\theta_0), V(\theta_0) \big), 
    \end{align*}
    where $\kappa(\theta_0)$ is defined in \cref{def:information}. Here we note that if $p \ll n^{2/3}$, then $\delta = 0$ and the mean of the limiting normal is $0$.

    \item When $p \gg n^{2/3}$, we have
    $$\frac{n}{p^2} \sumin \nabla F_i(\theta_0) \mid X \xp \kappa(\theta_0).$$

    \item When $p \ll {n^{3/4}}$, we have
    \begin{align*}
        \frac{1}{\sqrt{p}} \sumin \nabla F_i(\theta_0) - \frac{p^{3/2}}{n} \kappa(\theta_0) \mid X \xd N\big(0, V(\theta_0) \big).
    \end{align*}
    \end{enumerate}
\end{lemma}

Finally, we state three technical moment-related lemmas using the sub-Gaussian assumption of the tilted prior (see R6 in \cref{assmp:prior technical}). 
The following lemma provides convenient moment and cgf bounds under the sub-Gaussian assumption.
\begin{lemma}\label{lem:sub-Gaussianity}
    Under R6 from \cref{assmp:prior technical}, we have  
    $$\sup_{\theta\in\Theta} \log \int e^{t b - \frac{d b^2}{2}} d\mu_\theta(b) \lesssim 1+|t|^4, \quad \sup_{\theta\in\Theta} \EE |B_{t,d,\theta}|^r \lesssim_{r} 1+|t|^{3r}, \quad \forall r>0, ~ t \in \R, ~ d>0.$$
    Additionally assuming R5 from \cref{assmp:prior technical}, we have $\sup_{\theta\in\Theta}\sumin \EE |B_{w_i, d_i,\theta}|^r = O_P(p)$.
\end{lemma}

The next lemma generalizes Lemma B.4 in \cite{lee2025clt}, which required a bounded prior.

\begin{lemma}\label{lem:B4}
    Suppose R6 from \cref{assmp:prior technical} holds. For any $t \in \R, \theta\in\Theta$ and $h \in \mcp(r)$, we have
    \begin{align*}
        \Big| \log \Big[\int e^{t b - \frac{d b^2}{2}} d \mu_\theta(b) \Big] - \log \Big[\int e^{t b - \frac{d_0 b^2}{2}} d \mu_\theta(b) \Big] \Big| &\lesssim |d-d_0| \big(1+ |t|^6\big), \\
        |\EE h(B_{t,d,\theta}) - \EE h(B_{t,d_0,\theta})| &\lesssim_r |d-d_0| \big(1+ |t|^{3(r+2)} \big).
    \end{align*}
\end{lemma}

\begin{lemma}\label{lem:convergence for F_i}
     Suppose R5 from \cref{assmp:prior technical} holds. For $\F_i := \sum_{j \neq i} A_p(i,j)\beta_j^\star$ and any fixed constant $r>0$, we have
    $$\frac{1}{p} \sumin e^{|\F_i|} \mid X \xp 1, \quad \frac{1}{p} \sumin \EE[{|w_i|^r} \mid X,\bbst] = O_P(1).$$
\end{lemma}

\begin{proof}[Proof of \cref{thm:upper bound}]
   Under the general asymptotic scaling $p = o(n)$, we establish common bounds before proving the individual conclusions in parts (a), (b), and (c). Throughout the proof, the design matrix $X$ will always be conditioned on unless specified otherwise.

\begin{enumerate}[(i)]
    \item Consistency: We begin by proving $\thetaveb\xp \theta_0$ by checking the conditions in \cref{lem:Wald}. 
    To this extent, set $g(t,d,\theta) := \log \int e^{t b - \frac{d b^2}{2}} d\mu_{\theta}(b)$ and recall that $\thetaveb$ is a maximizer of $\tilde{M}_p(\theta) = \frac{1}{p} \sumin g(w_i,d_i,\theta)$. 
    Since 
    \begin{align}\label{eq:ddo}
    d_i = d_0 + \OPX\Big(\frac{1}{\sqrt{n}}\Big)    
    \end{align}
    by \cite[Equation (A.19)]{lee2025bayesregression}, define a simpler objective $M_p(\theta) := \frac{1}{p} \sumin g(w_i,d_0,\theta)$ to which we apply \cref{lem:Wald}. First note that the derivative of the map $t\mapsto g(t,d_0,\theta)$ belong to $\mcp(r)$ for all $\theta\in\Theta$, by \cref{lem:sub-Gaussianity}. Therefore, the first two conditions in \cref{lem:Wald}, LLN and unique maximizer, directly follow from Lemmas \ref{lem:LLN} and \ref{lem:unique maximizer} respectively. For the third condition, it suffices to show the following bound:
    $$\EE \left[ \sup_{\theta \in \Theta} \log \Big(\int e^{W_{\theta_0} b - \frac{d_0 b^2}{2}} d\mu_\theta(b) \Big) \right] < \infty.$$
    For this goal, applying the first bound in \cref{lem:sub-Gaussianity} gives $$\sup_{\theta\in\Theta}\log \int e^{t b - \frac{d_0 b^2}{2}} d\mu_\theta(b) \lesssim 1+ t^4, \quad \forall t.$$
    Hence, the expectation is upper bounded by $O(1 + \EE W_{\theta_0}^4) < \infty$. Finally, using \cref{lem:B4} coupled with \eqref{eq:ddo}, we can write $g(w_i, d_i, \theta) = g(w_i, d_0, \theta) + O_P\Big(\frac{1+|w_i|^6}{\sqrt{n}}\Big)$. By noting that $\sumin |w_i|^6 = O_P(p)$ (by \cref{lem:convergence for F_i}), we have
    $$M_p(\thetaveb) = \tilde{M}_p(\thetaveb) + o_{P}(1) \ge \tilde{M}_p(\theta_0) + o_{P}(1) = M_p(\theta_0) +o_{P}(1).$$
    Therefore consistency follows from \cref{lem:Wald}.

    \item Limiting distribution: 
    For a $k \times k \times k$ tensor $T$ and a vector $v$, define $T \times_a v \in \R^{k \times k}$ as the mode-$a$ product of $T$ and $v$.
    By the first order conditions and Taylor approximation, we have
    \begin{align*}
        0_k &= \frac{1}{p} \sumin \nabla F_i(\thetaveb) \\
        &= \frac{1}{p} \sumin \nabla F_i(\theta_0) + \frac{1}{p} \sumin \nabla^2 F_i(\theta_0) (\thetaveb - \theta_0) \\
        & \quad + \frac{1}{2p} \sumin \nabla^3 F_i(\tilde{\theta}_n) \times_2 (\thetaveb - \theta_0) \times_3 (\thetaveb - \theta_0)
    \end{align*}
    for some $\tilde{\theta}_n \in (\theta_0, \thetaveb)$.
    Hence, we can write
    \begin{align}\label{eq:thetahat Taylor}
        (\thetaveb - \theta_0) = - \Big(\frac{1}{p} \sumin \nabla^2 F_i(\theta_0) + \frac{1}{2p} \sumin \nabla^3 F_i(\tilde{\theta}_n) \times_3 (\thetaveb - \theta_0) \Big)^{-1} \Big(\frac{1}{p} \sumin \nabla F_i(\theta_0) \Big).
    \end{align}
   By direct computations, we can write the derivatives as:
    \begin{align*}
        \nabla F_i (\theta) &= \EE \nabla \ell(\theta; B_{w_i,d_i,\theta}), \\
        \nabla^2 F_i (\theta) %
        &= \EE [\nabla^2 \ell(\theta; B_{w_i,d_i,\theta})] + \EE [\nabla \ell(\theta; B_{w_i,d_i,\theta})^2] - \big[\EE \nabla \ell(\theta; B_{w_i,d_i,\theta})\big]^2, \\
        \nabla^3 F_i(\theta) &= \EE [\nabla^3 \ell(\theta; B_{w_i,d_i,\theta}) + \nabla \ell(\theta; B_{w_i,d_i,\theta})^3] + 2 [\EE \nabla \ell(\theta; B_{w_i,d_i,\theta})]^3 \\
        & ~~ - 3 \EE[\nabla \ell(\theta; B_{w_i,d_i,\theta})] \otimes \EE[\nabla^2 \ell(\theta; B_{w_i,d_i,\theta}) + \nabla \ell(\theta; B_{w_i,d_i,\theta})^2].
    \end{align*}
    Here, note that we use R2 to ensure that the derivatives in the RHS are well-defined.
    In the above expression, for a vector $v \in \R^k$, we have abbreviated $v^k$ as the $k$-fold outer product, i.e. $v^2 :=  v \otimes v, v^3 :=  v \otimes v \otimes v$.
    Recall that $B_{w_i,d_i,\theta}$ is defined in \cref{def:information}, and note that all expectations are conditional on $w_i$s. We claim that the third derivative term satisfies $$\frac{1}{p}\sumin \nabla^3 F_i(\tilde{\theta}_n) \times_3 (\thetaveb-\theta_0) = o_P(1),$$
    and is negligible. To see this, use the growth-condition R4 from \cref{assmp:prior technical}, to write:
    \begin{align*}
        \|\nabla^3 F_i(\theta)\|_F \lesssim 1+ \EE |B_{w_i,d_i,\theta}|^r,
    \end{align*}
    for some integer $r>0$, where the implied constant does not depend on $\theta$. Then, by the last bound in \cref{lem:sub-Gaussianity}, we have 
    $$\|\frac{1}{p}\sumin \nabla^3 F_i(\tilde{\theta}_n)\|_F \lesssim \frac{1}{p} \sumin \sup_{\theta} \|\nabla^3 F_i(\theta)\|_F \lesssim \sup_{\theta\in\Theta}\frac{1}{p} \sumin (1+\EE |B_{w_i,d_i,\theta}|^r)= O_P(1),$$
    so the claim follows by Slutsky's theorem.

    Regarding the second derivative term, we claim the limit
    $$\frac{1}{p} \sumin \nabla^2 F_i(\theta_0) \xp - V(\theta_0).$$
    Now, by applying the simplified LLNs in part (b) of \cref{lem:LLN} with $$h(b) = \nabla^2 \ell(\theta_0;b), [\nabla \ell(\theta_0;b)]^2, \nabla \ell(\theta_0;b)$$ (corresponding to each line below), the following limits hold:
    \begin{align}
        \frac{1}{p} \sumin \EE[\nabla^2 \ell(\theta_0; B_{w_i,d_i,\theta_0})] &\xp \EE \nabla^2 \ell(\theta_0; B_{\theta_0}), \label{eq:individual LLN} \\
         \frac{1}{p} \sumin \EE[\nabla \ell(\theta_0; B_{w_i,d_i,\theta_0})^2] &\xp \EE \nabla \ell(\theta_0; B_{\theta_0})^2, \notag \\
         -\frac{1}{p} \sumin [\EE \nabla \ell(\theta_0; B_{w_i,d_i,\theta_0})]^2 &\xp -\EE \Big[ \EE \big(\nabla \ell(\theta_0; B_{\theta_0}) \mid W_{\theta_0}\big) \Big]^2. \notag
    \end{align}
    Therefore, 
    \begin{align*}
    \frac{1}{p}\sumin \nabla^2 F_i(\theta_0) & \xp \EE \nabla^2 \ell(\theta_0; B_{\theta_0}) + \sumin [\EE \nabla \ell(\theta_0; B_{w_i,d_i,\theta_0})]^2 -\EE \Big[ \EE \big(\nabla \ell(\theta_0; B_{\theta_0}) \mid W_{\theta_0}\big) \Big]^2 \\ & = \EE[\nabla^2 \tilde{\ell}(\theta_0;W_{\theta_0})],
    \end{align*}
    where $\tilde{\ell}(\theta_0;\cdot)$ is the log-likelihood of the random variable $W_{\theta_0}$. The claim now follows via \cref{lem:variance checking}.

    By combining the above conclusions, we can simplify \eqref{eq:thetahat Taylor} as follows:
    \begin{equation}\label{eq:theta expansion}
        (\thetaveb - \theta_0) = \big(V(\theta_0)+o_P(1)\big)^{-1} \frac{1}{p}\sumin \nabla F_i(\theta_0).
    \end{equation}
\end{enumerate}

    Finally, the individual conclusions in the theorem follow by using the limiting distribution of $p^{-1}\sumin \nabla F_i(\theta_0)$ from \cref{lem:linear clt} alongside Slutsky's theorem.
\end{proof}

Next we prove \cref{cor:misspecified} by making minor changes in the above argument.
\begin{proof}[Proof of \cref{cor:misspecified}]
We first modify \cref{lem:unique maximizer} to the following claim:
$$\Flm^\star(\theta) := \EE \Big[\log \Big(\int e^{W^\star b - \frac{d_0 b^2}{2}} d \mu_\theta(b) \Big) \Big]$$
is uniquely maximized at $\theta = \theta^\star$. 
To show this, note that 
$$\dkl(W^\star \mid W_\theta) = \EE \Big[\log \Big(\int e^{W^\star b - \frac{d_0 b^2}{2}} d \mu^\star(b) \Big) \Big] - \EE \Big[\log \Big(\int e^{W^\star b - \frac{d_0 b^2}{2}} d \mu_\theta(b) \Big) \Big].$$
Thus, $\argmax_{\theta} \Flm^\star(\theta) = \argmin_\theta \dkl(W^\star \mid W_\theta) = \theta^\star$, where we use \eqref{eq:theta star} alongside the unique minimizer assumption. By re-writing the first order conditions for $\nabla F_0^\star(\theta^\star) = 0$, we get $\EE \nabla \ell (\theta^\star; B_{\theta^\star}^\star) = 0$, which is the misspecified analog of Bartlett's first identity.

Now, we can apply Wald's consistency theorem similar to that in part (a) of proof of \cref{thm:upper bound}, and get $\thetaveb \xp \theta^\star$. To establish the limit of the numerator, we again use \eqref{eq:thetahat Taylor} where all instances of $\theta_0$ are replaced with $\theta^\star$:
\begin{align}\label{eq:corollary taylor}
    (\thetaveb - \theta^\star) = - \Big(\frac{1}{p} \sumin \nabla^2 F_i(\theta^\star) + o_P(1) \Big)^{-1} \Big(\frac{1}{p} \sumin \nabla F_i(\theta^\star) \Big).
\end{align}
Noting that the proof of \cref{lem:linear clt} is still valid under the notation change from $(B_{\theta_0}, W_{\theta_0})$ to $(B_{\theta^\star}^\star,W^\star)$, parts (a) and (b) can be modified to:
\begin{itemize}
    \item When $\frac{p^{3/2}}{n} \to \delta \in [0, \infty)$, we have
    \begin{align*}
        \frac{1}{\sqrt{p}} \sumin \nabla F_i(\theta^\star) \mid X \xd N\big(\delta \kappa^\star(\theta^\star), V(\theta^\star) \big).
    \end{align*}
    \item When $p \gg n^{2/3}$, we have
    $$\frac{n}{p^2} \sumin \nabla F_i(\theta^\star) \mid X \xp \kappa^\star(\theta^\star).$$
\end{itemize}
Also, the limit of the denominator modifies to (see \eqref{eq:individual LLN} and \cref{lem:variance checking}):
$$\frac{1}{p} \sumin \nabla^2 F_i(\theta^\star) \mid X \xp \EE\Big[\nabla^2 \ell(\theta^\star; B_\theta^\star) + \big(\nabla \ell(\theta^\star; B_\theta^\star) \big)\big(\nabla \ell(\theta^\star; B_\theta^\star) \big)^\top \Big] - V(\theta^\star) = -V^\star(\theta^\star).$$
The only change here is that we cannot use Bartlett's second identity to simplify the final equality in \eqref{eq:W_theta Fisher information}. Now, the individual conclusions in each part of \cref{cor:misspecified} follow by combining the separate limits for the numerator and denominator of \eqref{eq:corollary taylor} via Slutsky's theorem.
\end{proof}

\subsection{Proof of \cref{thm:lower bound}}

\vspace{3mm}
Without loss of generality, assume that the bounded support in \cref{assmp:cpt support} is the interval $[-1,1]$. \cref{thm:lower bound} is based on two technical moment bounds for RFIMs. Before stating the lemmas, we introduce some  additional notation. 
\begin{defi}[Random field Ising model (see \cite{lee2025clt,nattermann1998theory,amaro1991fluctuations,bhattacharya2024ldp})]\label{def:rfim}
    Given $\theta \in \Theta$ and any $c = (c_1, \ldots, c_p) \in \R^p$, let $B = (B_1, \ldots, B_p) \stackrel{d}{\equiv} \tilde{\PP} = \tilde{\PP}^{A_p, c}_\theta$ be a sample from the following RFIM:
    $$\frac{d \tilde{\PP}^{A_p, c}_\theta}{d\prod_{i=1}^p \mu_{i,\theta}} (b) := \frac{\exp \Big(\frac{b^\top A_p b}{2} + c^\top b \Big)}{Z_p(c,\theta)}, \quad \forall \, b \in [-1,1]^p,$$
    where $Z_p(c,\theta)$ is the normalizing constant defined in \eqref{eq:normalizing constant def} and $\mu_{i,\theta}$ are quadratic tilts in \cref{def:quadratic tilt}.
\end{defi}

\noindent Recall from \cref{def:quadratic tilt} that $\mu_{i,c_i,\theta}$ denotes the exponential tilt of $\mu_{i,\theta}$ with $\frac{d\mu_{i,c_i,\theta}}{d\mu_{i,\theta}} (B_i) \propto \exp(c_i B_i)$. Denote the $i$th marginal of the Mean-Field approximation (that attains the supremum in \eqref{eq:MF over product}) 
as $${Q}_{i,\theta,c} := \mu_{i,s_{i,\theta}+c_i,\theta}, \quad \textnormal{where} \quad {s}_{i,\theta} := \sumjn A_p(i,j) {u}_{j,\theta},$$
where $u_{j,\theta}$s are as in \cref{def:fixed point}. For any function $f: [-1,1] \to \mathbb{R}^k$ and measure $Q_i$ on $[-1,1]$, write the expectation of $f$ under $Q_i$ (assuming it is well-defined) as:
$$\langle f\rangle_{Q_i}:= \EE^{B_i \sim Q_i} f(B_i).$$ 

\noindent The following lemma provides finite-sample error bounds for the first and second moment of general RFIMs (with arbitrary random fields $c$).
\begin{lemma}\label{thm:CLT}
Let $A_p$ be the off-diagonal parts of the random design in \cref{assmp:random design}, $c \in \R^p$ be an arbitrary vector, and consider the asymptotic regime $p = o(n^{2/3})$. For $B \sim \tilde{\PP}^{A_p, c}_\theta$ and any vector-valued $\mathcal{C}^2$ function $g:[-1,1] \to [-1,1]^k$, %
the following limits hold:
\begin{enumerate}[(a)]
    \item 
    \begin{align*}
        \frac{1}{p} \EE^{\tilde{\PP}^{A_p,c}_{\theta}} \Big[ \sumin \big[g(B_i) - \langle g \rangle_{{Q}_{i,\theta,c}} \big]\Big] &= \OPX\left({\frac{1}{\sqrt n}} + \frac{p}{n}\right), \\
        \frac{1}{p}\sumin \Big[\langle g \rangle_{{Q}_{i,\theta,c}} - \langle g \rangle_{\mu_{i,c_i,\theta}} \Big] &= O_{P,X}\Big(\sqrt{\frac{p}{n}}\Big).
    \end{align*}
    \item
    \begin{align*}
        \frac{1}{p} \EE^{\tilde{\PP}^{A_p,c}_{\theta}} \Big[\sumin \big[g(B_i) - \langle g \rangle_{{Q}_{i,\theta,c}}\big] \Big]^2 
        =& \frac{1}{p} \sumin \Big[\langle g^2 \rangle_{\mu_{i,c_i,\theta}} - \big(\langle g \rangle_{\mu_{i,c_i,\theta}}\big)^2 \Big] + \OPX\left(\frac{1}{\sqrt{p}} + \sqrt{\frac{p^3}{n^2}}\right).
    \end{align*}
\end{enumerate}
Here, the hidden constants are absolute constants, and part (b) abbreviates $v^2 := v \otimes v$ for vectors $v \in \R^k$.
\end{lemma}

\begin{proof}[Proof of \cref{thm:lower bound}]
Recall that the marginal likelihood in \eqref{eq:marginal}. 
By direct computations noting that
$$\tilde{\PP}_\theta^{A_p,w} (B) \propto \exp \Big(\frac{B^\top A_p B}{2} + w^\top B \Big) \prod_{i=1}^p \mu_{i,\theta}(B_i) \propto \exp \Big(-\frac{\|y-X B\|^2}{2\sigma^2}\Big) \prod_{i=1}^p \mu_{\theta}(B_i)$$
and $\mu_\theta(b) = e^{\ell(\theta;b)}$, we have
\begin{align*}
    \frac{\partial \log m_\theta}{\partial \theta} &= \EE^{\tilde{\PP}^{A_p, w}_\theta} \Big[ \sumin \nabla \ell(\theta; B_i) \Big], \\
    \frac{\partial^2 \log m_\theta}{\partial \theta^2} &= \EE^{\tilde{\PP}^{A_p, w}_\theta} \Big[ \sumin \nabla^2 \ell(\theta; B_i) \Big] + \Var^{\tilde{\PP}^{A_p, w}_\theta} \Big[ \sumin \nabla \ell(\theta; B_i) \Big].
\end{align*}

Define $g^{(1)}(b) := \nabla \ell(\theta_0; b)$, $g^{(2)}(b) := \nabla^2 \ell(\theta_0; b)$. Under \cref{assmp:cpt support}, $\ell(\theta; b)$ is a continuous function defined on a compact set, and hence uniformly bounded.
For $B \sim \tilde{\PP}_{\theta_0}^{A_p, w},$ abbreviate the average of $g^{(a)}(B_i)$ as 
$$\overline{g^{(a)}(B)} := \frac{1}{p} \sumin g^{(a)} (B_i), \quad a = 1,2.$$
Then, we can write the Fisher information $\mathcal{I}_p(\theta_0)$ as:
\begin{align*}
    \mathcal{I}_p(\theta_0) &= - \EE\left[ \frac{\partial^2 \log m_\theta}{\partial \theta^2} \mid_{\theta = \theta_0}\right] \\
    &= - \EE\left[ \EE^{\tilde{\PP}^{A_p, w}_{\theta_0}} \Big[ \sumin \nabla^2 \ell(\theta_0; B_i) \Big]\right] - \EE\left[ \Var^{\tilde{\PP}^{A_p, w}_{\theta_0}} \Big[ \sumin \nabla \ell(\theta_0; B_i) \Big]\right].
\end{align*}

We separately compute the limit of each term. To control the first term, we use part (a) of \cref{thm:CLT} to write
\begin{align*}
     \frac{1}{p}\EE^{\tilde{\PP}^{A_p, w}_{\theta_0}} \Big[ \sumin \nabla^2 \ell(\theta_0; B_i) \Big] &= \EE^{\tilde{\PP}^{A_p, w}_{\theta_0}} [ \overline{g^{(2)}(B)} ] \\
    &= \frac{1}{p} \sumin \langle g^{(2)} \rangle_{\mu_{i,w_i,\theta_0}} + o_{P,X}(1) \\
    &\xp \EE \Big[ \EE \big[\nabla^2 \ell(\theta_0; B_{\theta_0}) \mid W_{\theta_0} \big] \Big] = -I(\theta_0).
\end{align*}
The convergence in probability is conditional on the covariate matrix $X$ and it follows from limits in \eqref{eq:individual LLN}. By taking an outer expectation (over $y$), we get the final limit:
$$\EE \left[ \frac{1}{p} \EE^{\tilde{\PP}^{A_p, w}_{\theta_0}} \Big[ \sumin \nabla^2 \ell(\theta_0; B_i) \Big] \right] \mid X \xp -I(\theta_0).$$
Here we use that fact the inner term is bounded (by a uniform constant).

Next, we compute the limit of the second term of $\mathcal{I}_p(\theta_0)$. Using both parts of \cref{thm:CLT}, we can write
\begin{align}
&p \Var^{\tilde{\PP}_{\theta_0}^{A_p,w}} [\overline{g^{(1)}(B)}] \notag \\
=&  p \EE^{\tilde{\PP}_{\theta_0}^{A_p,w}} \Big[\overline{g^{(1)} (B)} - \frac{1}{p}\sumin\langle g^{(1)} \rangle_ {{Q}_{i,\theta_0,w_i}} \Big]^2 - \frac{1}{p} \Big[\sumin \langle g^{(1)} \rangle_ {{Q}_{i,\theta_0,w}} - \EE^{\tilde{\PP}_{\theta_0}^{A_p,w}} \overline{g^{(1)} (B)} \Big]^2 \notag \\
=& \frac{1}{p} \sumin \Big[\langle (g^{(1)})^2\rangle_{{\mu}_{i,w_i,\theta_0}} - \big( \langle g^{(1)} \rangle_{{\mu}_{i,w_i,\theta_0}}\big)^2 \Big] + o_{P,X}(1) \nonumber \\ \xp & I(\theta)-V(\theta_0), \label{eq:lhs variance}
\end{align}
where the convergence in probability in the last line is conditional on $X$. Here, the second line uses the basic identity $\Var(V) = \EE[V -v]^2 - [v-\EE V]^2$, where $V$ is a random vector and $v$ is a deterministic vector. The third line follows from \cref{thm:CLT}. The last limit follows from \eqref{eq:individual LLN}. Finally, taking an outer expectation (over $y$) noting that the LHS of \eqref{eq:lhs variance} is bounded, we get 
$$\EE\left[ \frac{1}{p} \Var_{\tilde{\PP}^{A_p, w}_{\theta_0}} \Big[ \sumin \nabla \ell(\theta_0; B_i) \Big]\right] \mid X \xp I(\theta)-V(\theta_0).$$
The proof is complete by combining the two limits.
\end{proof}

\subsection{Proofs from \cref{sec:consequences}}
    We first state a stability result with respect to the prior parameter $\theta$.
    \begin{lemma}\label{lem:psi' bound}
    Suppose \cref{assmp:cpt support} holds. For any $t \in \R, i \le p$, $\theta, \theta' \in \Theta$, and a bounded function $h :[-1,1] \to [-1,1]$, we have
    $$|\langle h \rangle_{\mu_{i,t,\theta}} - \langle h \rangle_{\mu_{i,t,\theta'}} | \lesssim \|\theta - \theta'\|,$$
        where the hidden constant only depends on the prior likelihood $\ell$. In particular, we have
        $$|\psi_{i,\theta}'(t) - \psi_{i,\theta'}'(t)|, ~|\psi_{i,\theta}''(t) - \psi_{i,\theta'}''(t)| = O \big(\|\theta - \theta'\|\big)$$
    \end{lemma}

    The next technical result provides a way to quantify the distance between the true posterior and the EB posterior. As it involves the use of several technical lemmas from \cite{lee2025clt} to compute moment bounds for RFIMs, we state them in the following lemma. Each part follows from Theorem 2.3, Lemma 3.2(b), and Lemma 3.1 in \cite{lee2025clt}, respectively. We also refer the reader to \cite[Lemma 2,10]{bhattacharya2023gibbs} and \cite[Lemma A.1]{deb2025pivotal} for related results.

\begin{lemma}[Facts about RFIMs]\label{lem:RFIM lemma}
    For $B \sim \tilde{\PP}_\theta^{A_p,c}$, let $m_i := m_i(B) = \sum_{j \neq i} A_p(i,j) B_j$ denote the ``local field'' for each coordinate $i$. Then, under \cref{assmp:random design} and $p \ll n^{2/3}$, the following conclusions hold.
    \begin{enumerate}[(a)]
        \item For a constant vector $\mathfrak{\zeta} \in \R^p$, we have $\EE\Big[\sumin \zeta_i (B_i-u_{i,\theta})\Big]^2 = \OPX\big(\|\zeta\|^2 \big).$
        \item $\EE\Big[\sumin (m_i-s_{i,\theta})^2\Big]^2 = \OPX \Big( \frac{p^3}{n^2} \Big).$
        \item For any measurable function $h: [-1,1] \to [-1,1]$, we have 
    \end{enumerate}
        \begin{align*}
        \EE\Big[\frac{1}{p} \sumin \big(h(B_i) - \langle h \rangle_{\mu_{i,m_i+c_i,\theta}}\big) \Big] & = O\Big(\frac{1}{\sqrt{p}}\Big), \\ \quad \EE\Big[\frac{1}{p} \sumin \Big(h(B_i)-\langle h \rangle_{\mu_{i,m_i+c_i,\theta}}\Big) \langle h \rangle_{\mu_{i,m_i+c_i,\theta}} \Big] & = O\Big(\frac{1}{\sqrt{p}}\Big).
        \end{align*}
\end{lemma}

The final technical result provides some basic properties of the design matrix $X$ that will be useful in the sequel. The proof follows from Lemmas 4.1, 4.2 in \cite{lee2025clt}.

\begin{lemma}[Design matrix properties]\label{lem:alpha_p}
    Let $X$ be the random design from \cref{assmp:random design} and recall $A_p$ from \cref{def:A}. Then, for any $r \ge 2$, the following properties hold:
    $$\|A_p\|_{r \to r } =O_{P,X}\Big(\frac{p^{1-1/r}}{\sqrt{n}}\Big), \quad \alpha_p = \alpha_p(X) := \maxin \sumjn A_p(i,j)^2 = O_{P,X}\Big(\frac{p}{n}\Big).$$
    Recall from \cref{subsec:notation} that $O_{P,X}$ denotes the sole randomness of $X$.
\end{lemma}

\begin{proof}[Proof of \cref{prop:mean variance}]
Recall the notation $u_i, \hat{u}_i$ from \cref{def:MF posterior}. We also use the notations $s_i = \sum_{j\neq i} A_p(i,j) u_j$ and $\hat{s}_i = \sum_{j\neq i} A_p(i,j) \hat{u}_j$ throughout this proof.

\begin{enumerate}[(a)]
    \item We prove a general claim under a general parameter value $\theta$ that may depend on the data $y,X$. Let $h$ be any 
    bounded function with $\|h\|_{\infty} \le 1$. For $t \in \R$ and fixed indices $i,\theta$, define 
    $H_{i,\theta}(t):=\langle h\rangle_{\mu_{i,t,\theta}}$. 
    Using the properties of the complete conditional distribution of $\beta_i$ (see \eqref{eq:conditional distribution of B_i}), for $\beta \sim \PP_{\theta}$ and $m_i = m_i(\beta_{(-i)}) = \sum_{j \neq i} A_p(i,j) \beta_j$, we have
    \begin{align}\label{eq:conditional expectation}
        \EE[h(\beta_i)] = \EE[\EE[h(\beta_i) \mid \beta_{(-i)}]] = \EE \langle h \rangle_{\mu_{i,m_i+w_i,\theta}} = \EE H_{i,\theta}(m_i+w_i).
    \end{align}

    Our key argument is noting that  
    \begin{align}\label{eq:h expectation}
        |\EE H_{i,\theta}(m_i+w_i) - H_{i,\theta}(s_i+w_i)| = O(\EE|m_i-s_i|) = O_{P,X}(\sqrt{\an}).
    \end{align}
    Here, the first equality follows from a Taylor expansion and the fact that $h$ is bounded, and the second follows from writing $m_i - s_i = \sum_{j \neq i} A_p(i,j) (\beta_j-u_j)$ and using a second moment bound for RFIMs (see \cref{lem:RFIM lemma}(a)). The hidden constants above do not depend on the prior hyperparameter $\theta$, and only on the compact support in \cref{assmp:cpt support}. 
    \vspace{2mm}

    The identities for the mean and variance are immediate by taking $h(b) = b, b^2$, respectively in \eqref{eq:h expectation}. The conclusion for the TV distance follows by spelling out the definition and using \eqref{eq:h expectation}:
    $$\dtv \big(\PP_{\theta}(\beta_i), \QQ_{\theta}(\beta_i)\big) = \sup_{h(\cdot) = I(\cdot \in E), ~E \subset \R} \Big|\EE \langle h\rangle_{\mu_{i,m_i+w_i,\theta}} - \langle h \rangle_{\mu_{i,s_i+w_i,\theta}} \Big| = O_{P,X}(\sqrt{\an}).$$
    Finally we use the bound on $\an$ from \cref{lem:alpha_p}.

    \vspace{2mm}
    \item Following part (a), let $h$ be any bounded function. Use two triangle inequalities (first and third line) to bound 
    \begin{align}
        |\EE_{\PP_{\theta_0}} h(\beta_i) &- \EE_{\QQ_{\thetahat}} h(\beta_i)| \le |\EE_{\PP_{\theta_0}} h(\beta_i) - \EE_{\QQ_{\theta_0}} h(\beta_i)| + |\EE_{\QQ_{\theta_0}} h(\beta_i) - \EE_{\QQ_{\thetahat}} h(\beta_i)| \notag \\
        &= O_{P,X}(\sqrt{\an}) + |H_{i,\theta_0}(s_i + w_i) - H_{i,\thetahat}(\hat{s}_i + w_i)| \notag \\
        &\le O_{P,X}(\sqrt{\an}) + |s_i - \hat{s}_i| + |H_{i,\theta_0}(\hat{s}_i + w_i) - H_{i,\thetahat}(\hat{s}_i + w_i)| \notag \\
        &= O_{P,X}(\sqrt{\an}) + |s_i - \hat{s}_i| + O(\|\thetahat - \theta_0\|).\label{eq:middle bound}
    \end{align}
    Here, the second line used \eqref{eq:h expectation} from part (a), and the fourth line used \cref{lem:psi' bound}.

    \vspace{2mm}
    To control $|s_i - \hat{s}_i|$ in \eqref{eq:middle bound}, let $v_i := \hat{u}_i - u_i$. Again, by \cref{lem:psi' bound} (with $h(b) = b$ and $t = \hat{s}_i+w_i$), we have
    $$\hat{u}_i = \psi_{i,\thetahat}'\big(\hat{s}_i + w_i\big) = \psi_{i,\theta_0}'\big(\hat{s}_i + w_i\big) + E_i, \quad \text{where} \quad |E_i| \lesssim \|\thetahat-\theta_0\|.$$
    Hence, using the definition of $\hat{u}_i, u_i$ in \cref{def:MF posterior}, we can write
    \begin{align}
        v_i = \psi_{i,\theta_0}'\big(\hat{s}_i + w_i\big) - \psi_{i,\theta_0}'\big(s_i + w_i\big) &= \psi_{i,\theta_0}''(\xi_i) (\hat{s}_i-s_i) + E_i \notag \\
        &= \psi_{i,\theta_0}''(\xi_i) \sum_{j \neq i} A_p(i,j) v_j + E_i, \label{eq:v_i bound}
    \end{align}
    for some $\xi_i \in (s_i,\hat{s}_i)$. By setting $D_p$ as a $p \times p$ matrix with $D_p(i,j) := \psi_{i,\theta_0}''(\xi_i) A_p(i,j)$, we get a vector equation $v = D_p v + E$ (with $E = (E_1,\ldots, E_p$)). As
    $$\|D_p\| \le \|\diag(\psi_{i,\theta_0}''(\xi_i))\| \|A_p\| < \frac{1}{2}$$
    with high probability (using the operator norm bound for $A_p$ in \cref{lem:alpha_p}), we have $$\|v\| \le \|(I_p - D_p)^{-1}\| \|E\| \le  2 \|E\| = O_{P,X}(\sqrt{p}\|\thetahat-\theta_0\|).$$
    Then, by a Cauchy-Schwartz, we have
    $$|s_i - \hat{s}_i| = \Big|\sum_{j \neq i} A_p(i,j) v_j\Big| \le \sqrt{\sumjn A_p(i,j)^2} \|v\| = O_{P,X}\Big(\frac{p}{\sqrt{n}}\|\thetahat-\theta_0\|\Big).$$
    The desired conclusions follow by plugging this into \eqref{eq:middle bound}. In particular, the TV distance claim follows by taking $h$ as indicator functions, and the claims for the mean and variance follows by taking $h(b) = b$ and $b^2$.
\end{enumerate}
\end{proof}

    The proof of \cref{thm:EB} builds upon recently developed finite-sample error bounds for the KS distance between the projected posterior and a Gaussian, which we state below.
    \begin{lemma}[Theorem 2.4 in \cite{lee2025clt}]\label{lem:posterior clt}
        Suppose $\beta = (\beta_1, \ldots, \beta_p) \in [-1,1]^p \stackrel{d}{\equiv} \tilde{\PP}_{\thetahat}^{A_p,w}$ (see \cref{def:rfim}). Define error terms $R_{1}, R_{2}$ as:
        \begin{align*}
            R_{1} := \sumin \Big(\sumji A_p(i,j) q_i (\psi_{j,\thetahat}''(w_j)-\hat{\upsilon}_p) \Big)^2, \quad R_{2} := \sumin \Big(\sumji A_p(i,j) \psi_{j,\thetahat}'(w_j)\Big)^2,
        \end{align*}
        where $\hat{\upsilon}_p$ is defined in \cref{thm:EB}. Then, under \cref{assmp:random design}, the following bound holds, where the hidden constant is an universal:
        {\small\begin{align*}
            d_{KS}\left(\sumin q_i(\beta_i - \hat{u}_i), N(0, \hat{\upsilon}_p) \mid y, X \right) \lesssim \sqrt{R_{1}} + \sqrt{\alpha_p R_{2}} + \frac{\sqrt{p}\alpha_p + \sqrt{\sumin q_i^4} \sqrt{R_{2}}+\|q\|_\infty}{\hat{\upsilon}_p} + \|A_p\|.
        \end{align*}
        }
    \end{lemma}
    As the original result in \cite{lee2025clt} allows any choice of prior/base-measure $\mu$, it also holds under the estimated prior $\mu = \mu_{\thetahat}$.
    Here, we have slightly modified the statement in \cite{lee2025clt} (i.e. replaced an additional error term $R_{3p}$ there in terms of $R_{2}$) by replacing the bound in the third display in \citep[pg 17,][]{lee2025clt} to the following:
    $$\Big|\sumin q_i^2 (\psi_i''(s_i+w_i) - \psi_i''(w_i))\Big| \le \sumin q_i^2 |s_i| \le \sqrt{\sumin q_i^4} \sqrt{\sumin s_i^2}\lesssim \sqrt{\sumin q_i^4} \sqrt{R_2},$$

    We also need the following technical lemmas.
    \begin{lemma}\label{lem:sum q'u approximate}
        Under Assumptions \ref{assmp:random design} and \ref{assmp:prior technical}, for any $\thetahat$ with $\|\thetahat - \theta_0\| = O_P\Big(\frac{1}{\sqrt{p}}\Big)$, we have
        $$\sumin q_i \psi_{i,\thetahat}'(s_{i,\thetahat}+w_i) = \sumin q_i \psi_{0,\thetahat}'(w_i) + o_P(1).$$
    \end{lemma}

    \begin{lemma}[LLN with general weights]\label{lem:LLN modified}
    Assume R5. Let $g:\R\to\R$ be a $\mathcal{C}^1$ function where $g'\in  \mcp(r)$, and let $r_i$ be constants such that $\sumin r_i \to \gamma \in \R$ and $\sumin r_i^2 = o(1)$. Then, we have 
    $$\sumin r_i g(w_i) \mid X \xp \gamma \EE g(W_{\theta_0}).$$
    \end{lemma}
    \noindent The proof of \cref{lem:LLN modified} follows from direct modifications from \cref{lem:LLN} (which is when $r_i \equiv \frac{1}{p}$). So we omit the proof for brevity.

\begin{proof}[Proof of \cref{thm:EB}]
\begin{enumerate}[(a),itemsep=0.5em]
    \item We show that $\hat{\upsilon}_p \mid X \xp \upsilon(\theta_0)$ under the double-dipping of the data $y$ on $w_i$ and $\thetahat$. For this goal, we already know from \cref{prop:oracle CLT} that $$\upsilon_p \mid X \xp \upsilon(\theta_0),$$
    so it suffices to show $\hat{\upsilon}_p - \upsilon_p  = o_P(1)$. Indeed, this follows from \cref{lem:psi' bound} as
    $$|\hat{\upsilon}_p - \upsilon_p| = |\sumin q_i^2 [\psi_{i,\thetahat}''(w_i) - \psi_{i,\theta_0}''(w_i)]| \le \sup_{w \in \R} |\psi_{i,\thetahat}''(w) - \psi_{i,\theta_0}''(w)| = O(\|\thetahat - \theta_0\|).$$
    The RHS is $o_P(1)$ as we assume that $\thetahat$ is consistent.

    \item It suffices to show that the finite-sample bound in the RHS of \cref{lem:posterior clt} is small. Here we use the bounds $\an = \OPX\Big(\frac{p}{n}\Big), \|A_p\| = \OPX\Big(\sqrt{\frac{p}{n}}\Big)$, which are stated in \cref{lem:alpha_p}. By these bounds, it also follows that
    \begin{align}\label{eq:R_1, R_2}
        R_{1} = \OPX\Big(\frac{p}{n}\Big), \quad R_{2} = \OPX\Big(\frac{p^2}{n}\Big)
    \end{align}
    where the hidden constants do not depend on the estimated prior $\thetahat$. The proof is complete by plugging these into \cref{lem:posterior clt}, alongside the assumption on $\sumin q_i^4$.

    \item Recall from \cref{prop:oracle CLT} that
    $$\sumin q_i (\beta_i- u_i) \mid y, X \xrightarrow[\PP_{\theta_0}]{d} N(0,\upsilon(\theta_0)).$$
    By the conditional converging together lemma (see \cref{lem:converge together}), it suffices to show the following limit:
    \begin{align}\label{eq:eb posterior suffices}
        \sumin q_i (u_i - \hat{u}_i) \mid X \xd N\Big(0, \gamma^2 J(\theta_0)^\top S(\theta_0) J(\theta_0)\Big),
    \end{align}
    as the two displays above imply
    $$\sumin q_i(\beta_i - \hat{u}_i) \mid X \xd N\Big(0, \upsilon(\theta_0) + \gamma^2 J(\theta_0)^\top S(\theta_0) J(\theta_0)\Big).$$
    Then, \eqref{eq:CLT oracle posterior} follows by modifying the variance to the sample-based quantities, using the conclusion of part (a) alongside the continuous mapping theorem.
    The claim under $\gamma = 0$ also follows from \eqref{eq:eb posterior suffices}, as \eqref{eq:eb posterior suffices} implies
    $$\sumin q_i (u_i - \hat{u}_i) \mid y, X \xp 0.$$
    
    It remains to show \eqref{eq:eb posterior suffices}. By applying \cref{lem:sum q'u approximate} for $\sumin q_i \hat{u}_i$ and $\sumin q_i u_i$, we can write $$\sumin q_i (\hat{u}_i - u_i) = \sumin q_i \Big[\psi_{0,\thetahat}'(w_i) - \psi_{0,\theta_0}'(w_i)\Big] + o_P(1).$$
    For notational simplicity, set $g(\theta, w) := \psi_{0,\theta}'(w)$.
    Then, a Taylor approximation gives
    \begin{align*}
        g(\thetahat, w_i) - g(\theta_0, w_i) = (\thetahat-\theta_0)^\top \frac{\partial g}{\partial \theta}(\theta_0,w_i) + O(\|\thetahat-\theta_0\|^2),
    \end{align*}
    where we used the fact that $\frac{\partial^2 g}{\partial \theta^2}$ is uniformly bounded under \cref{assmp:cpt support}.
    By using the LLN in \cref{lem:LLN modified}, we get
    $$\frac{1}{\sqrt{p}}\sumin q_i \frac{\partial g}{\partial \theta}(\theta_0,w_i) \mid X \xp \gamma \EE \frac{\partial g}{\partial \theta}(\theta_0, W_{\theta_0}) = \gamma J(\theta_0).$$
    Combining this alongside the limiting distribution $\sqrt{p}(\thetahat - \theta_0) \xd N(0, S(\theta_0)),$ we have
    \begin{align*}
        \sumin q_i \Big[\psi_{0,\thetahat}'(w_i) - \psi_{0,\theta_0}'(w_i)\Big] &= (\thetahat-\theta_0)^\top \sumin q_i \frac{\partial g}{\partial \theta}(\theta_0, w_i) + o_P(1) \\
        &\xd N\Big(0, \gamma^2 J(\theta_0)^\top S(\theta_0) J(\theta_0) \Big),
    \end{align*}
    completing the proof of \eqref{eq:eb posterior suffices}.
\end{enumerate}
\end{proof}

\subsection{Proofs from \cref{sec:simulations}}
\begin{proof}[Proof of \cref{prop:debiasing}]
    Throughout the proof, all convergence (in distribution/probability) statements are conditioned on $X$.
    Consider the asymptotic scaling $p \ll n^{3/4}$ and define a function $G(\theta) := V(\theta)^{-1} \kappa(\theta)$. By writing $$\sqrt{p}(\tilde{\theta}^{\mathsf{d}}-\theta_0) = \sqrt{p}\big(\thetaveb-\theta_0 - \frac{p}{n} G(\theta_0)\big)  - \frac{p\sqrt{p}}{n} \big(G(\thetaveb) - G(\theta_0)\big),$$
    it suffices to show the following limits:
    \begin{equation}\label{eq:debiasing second eq}
        \sqrt{p}\big(\thetaveb-\theta_0 - \frac{p}{n} G(\theta_0)\big) \xd N(0, V(\theta_0)^{-1}), \quad \frac{p\sqrt{p}}{n} \big(G(\thetaveb) - G(\theta_0)\big) \xp 0.
    \end{equation}
    
    Recall from \eqref{eq:theta expansion} that
    $$\sqrt{p}(\thetaveb - \theta_0) = \big(V(\theta_0)^{-1} + o_P(1)\big) \frac{1}{\sqrt{p}}\sumin \nabla F_i(\theta_0),$$
    so we can write the left term of \eqref{eq:debiasing second eq} as
    $$\sqrt{p}\big(\thetaveb - \theta_0 - \frac{p}{n}G(\theta_0) \big) = V(\theta_0)^{-1} \Big[\frac{1}{\sqrt{p}}\sumin \nabla F_i(\theta_0) - \frac{p\sqrt{p}}{n} \kappa(\theta_0)\Big] +  o_P(1).$$
    The first limit in \eqref{eq:debiasing second eq} follows by recalling from part (c) of \cref{lem:linear clt} that the limit
    $$\frac{1}{\sqrt{p}}\sumin \nabla F_i(\theta_0) - \frac{p\sqrt{p}}{n} \kappa(\theta_0) \xd N(0, V(\theta_0)).$$

    \noindent To show the second claim in \eqref{eq:debiasing second eq}, use a Taylor expansion to write
    $$G(\thetaveb) - G(\theta_0) =  G'(\xi) (\thetaveb - \theta_0) = O_p\Big(\frac{1}{\sqrt{p}} + \frac{p}{n}\Big),$$
    for some $\xi\in (\thetaveb,\theta_0)$. 
    Here we used (a) the fact that $G'(\cdot)$ is bounded (since it is a continuous function and $\Theta$ is compact), (b) the bound for $\thetaveb - \theta_0$ in \cref{thm:upper bound}. Now, the proof is complete by multiplying both sides by $\frac{p\sqrt{p}}{n}$.
\end{proof}

\section{Proof of lemmas}\label{sec:pf lemmas}

\subsection{Proof of Lemmas from the main paper}

\begin{proof}[Proof of \cref{lem:variance checking}]
    By marginalizing out the joint distribution of $(B_\theta,W_\theta)$ in \cref{def:information}, we have
    $$e^{\tilde{\ell}(\theta;t)} = \PP_{W_\theta}(t) = \int \frac{e^{- \frac{(t-d_0 b)^2}{2 d_0}}}{\sqrt{2\pi d_0}} d\mu_\theta(b) = \frac{e^{-\frac{t^2}{2d_0}}}{\sqrt{2\pi d_0}} \int e^{t b - \frac{d_0 b^2}{2} + \ell(\theta; b)} d b.$$
    Recalling the notation $v^2$ for a vector $v \in \R^k$ (see \cref{subsec:notation}), direct computations give
    \begin{align*}
        \nabla^2 \tilde{\ell}(\theta; t) &= \frac{\partial^2 \big(\log \int e^{t b - \frac{d_0 b^2}{2} + \ell(\theta; b)} d b\big)}{\partial \theta^2} \\
        &= \EE \big[\nabla^2 \ell(\theta; B_{t,d_0,\theta})\big] + \EE\big[\nabla \ell(\theta; B_{t,d_0,\theta})^2\big]- \big[\EE \nabla \ell(\theta; B_{t,d_0,\theta})\big]^2.
    \end{align*}
    By taking an outer expectation with respect to $t \stackrel{d}{\equiv} W_\theta$ and noting that $B_{W_\theta,d_0,\theta} \stackrel{d}{\equiv} B_\theta$, we get
    \begin{align}
        & \EE \big[\nabla^2 \ell(\theta; B_{W_\theta,d_0,\theta})\big] + \EE [\nabla \ell(\theta; B_{W_\theta,d_0,\theta})^2] - \EE \Big[ \EE \big(\nabla \ell(\theta; B_{W_\theta,d_0,\theta}) \mid W_{\theta}\big) \Big]^2 \notag \\
        =& \EE \big[\nabla^2 \ell(\theta; B_{\theta})\big] + \Var \nabla \ell(\theta; B_{\theta}) - \Var \Big[ \EE \big(\nabla \ell(\theta; B_{\theta}) \mid W_{\theta}\big) \Big] = - V(\theta). \label{eq:W_theta Fisher information}
    \end{align}
    Here the first equality used the first Bartlett's identity $\EE \nabla \ell(\theta; B_{\theta}) = 0$, and the second equality follows from the second Bartlett's identity $I(\theta) = -\EE \big[\nabla^2 \ell(\theta; B_{\theta})\big] = \Var \nabla \ell(\theta; B_{\theta})$ (see R3).
\end{proof}

For \cref{lem:prior condition checking}, we use the following property for sub-Gaussian mixtures, whose proof is postponed to \cref{sec:proof technical lemmas}.
\begin{lemma}\label{lem:mixture of Gaussians}
    Let $X$ be a $M$-component mixture of sub-Gaussian random variables, where the $m$-th mixture component has mean $\mu_m$ and sub-Gaussian norm $K_m$. Then, $X$ is sub-Gaussian with
    $$\|X\|_{\psi_2}^2 \le K_{\max}^2 + \frac{1}{4} (\mu_{\max} - \mu_{\min})^2,$$
    where $K_{\max} := \max_{m=1}^M K_m, ~ \mu_{\max} := \max_{m=1}^K \mu_m, ~\mu_{\min} := \min_{m=1}^K \mu_m$.
\end{lemma}

\begin{proof}[Proof of \cref{lem:prior condition checking}]
\begin{enumerate}[(a),itemsep=0.5em]
    \item This is immediate as the sub-Gaussian norm of a bounded measure is uniformly bounded.
    
    \item For a $\kappa$-log-concave measure $\nu$, the Bakry-Émery inequality gives the LSI \citep{bakry2006diffusions}, and consequently imply sub-Gaussianity of $\nu$: $\|\nu\|_{\psi_2} \le \frac{1}{\kappa}$. R5 directly follows from this observation, as $B_{\theta_0}$ is sub-Gaussian.

    Noting that $B_{t,d,\theta}$ has density proportional to $e^{tb - \frac{d b^2}{2} + \ell(\theta; b)}$, set $f(b) := tb - \frac{d b^2}{2} + \ell(\theta; b)$. $f$ is $d$-concave (i.e. strongly concave) as $f(b) + \frac{d b^2}{2}$ is concave for all $b$, and hence $B_{t,d,\theta}$ is sub-Gaussian with norm bounded by $\frac{1}{d} < \frac{2}{d_0}$ as long as $d > \frac{d_0}{2}$. Hence, the second claim in R6 holds.

    Proceeding to prove the first claim in R6, define a function $\mathfrak{m}:B_\epsilon(d_0) \times \Theta \to\R$ by setting
    $$\mathfrak{m}(d,\theta) := \EE[B_{0,d,\theta}] = \frac{\int b e^{-db^2/2} d\mu_\theta(b)}{\int e^{-db^2/2} d\mu_\theta(b)}, \quad \forall d \in B_\epsilon(d_0).$$
    As this is a continuous function (of $d, \theta$) on a compact domain (by R1), it is bounded. For example, regarding the supremum over $d$, direct computations show that
    $$|\partial_d \mathfrak{m}(d,\theta)| = \frac{1}{2} |\Cov(B_{0,d,\theta}, B_{0,d,\theta}^2)| \le \frac{1}{2} \sqrt{\Var(B_{0,d,\theta}) \Var(B_{0,d,\theta}^2)} \lesssim_{d_0} 1, \quad \forall d > \frac{d_0}{2}, ~~ \theta,$$
    so $\mathfrak{m}(d,\theta) = \mathfrak{m}(d_0,\theta) +O(1).$
    The final inequality uses the sub-Gaussianity $\|B_{0,d,\theta}\|_{\psi_2} \le \frac{2}{d_0}$ (note that this implies that $B_{0,d,\theta}^2$ is sub-exponential). 

    \item R5 is immediate by applying \cref{lem:mixture of Gaussians}, as each mixture component of $B_{\theta_0}$ is sub-Gaussian by parts (a) and (b). The first part of R6 is also bounded by the tower property. 
    To check the second part of R6, let $p_m(b)$ denote the conditional density (if we consider a degenerate mixture component, consider the base measure $\nu = \delta_0 + \lambda$ where $\lambda$ denotes the Lebesgue measure) of the $m$th mixture component, and observe that $B_{t,d,\theta}$ has density:
    $$\PP(B_{t,d,\theta} = b) \propto \sum_{m=1}^M \pi_m p_m(b) e^{tb-\frac{d b^2}{2}}.$$
    This is a finite mixture of the component-wise tilts $B^{(m)}_{t,d}$, which have density $p_m(b) e^{tb - \frac{db^2}{2}}$. By parts (a) and (b), each $B^{(m)}_{t,d}$ are sub-Gaussian with $\|B^{(m)}_{t,d}\|_{\psi_2} \le K$ for some constant $K>0$. Also, using the properties of exponential tilting (in particular, we appropriately modify \eqref{eq:tilt expectation} in \cref{lem:sub-Gaussianity}), we can show that $\EE B^{(m)}_{t,d} = O(1+K|t|)$. Now, the conclusion follows by applying \cref{lem:mixture of Gaussians}.

\end{enumerate}    
\end{proof}

\subsection{Proof of upper bound lemmas}\label{sec:pfupbdlm}
\begin{proof}[Proof of \cref{lem:unique maximizer}]
    Recall the notation $(B_\theta, W_\theta)$ from \cref{def:information} and denote the corresponding joint measure as $\PP_{B_\theta,W_\theta}$. Then, we have
    $$\frac{d\PP_{B_\theta,W_\theta}}{d\mu_\theta \otimes d\lambda}(b,t) = \frac{1}{\sqrt{2 \pi d_0}} e^{-\frac{t^2}{2d_0} + tb - \frac{d_0 b^2}{2}}.$$
    Here, $\lambda$ denotes the Lebesgue measure on $\R$. By marginalizing out $B_\theta$, we have
    \begin{align}\label{eq:w_theta density}
        \frac{d\PP_{W_\theta}}{d\lambda}(t) = \frac{e^{-\frac{t^2}{2d_0}}}{\sqrt{2 \pi d_0}} \int e^{tb - \frac{d_0 b^2}{2}} d\mu_\theta(b).
    \end{align}
    Now, by the definition of $\dkl$ and $F_0$,
    {\small
    \begin{align*}
        \dkl(\PP_{W_{\theta_0}} \| \PP_{W_\theta}) &= \EE \Big[\log\Big(\int \exp(W_{\theta_0} b - \frac{d_0 b^2}{2}) d\mu_{\theta_0}(b) \Big) - \log\Big(\int \exp(W_{\theta_0} b - \frac{d_0 b^2}{2}) d\mu_{\theta}(b) \Big)\Big] \\
        & = F_0(\theta_0) - F_0(\theta) \ge 0.
    \end{align*}
    }
    The conclusion holds if the final inequality is strict for any $\theta \neq \theta_0$. This is because the parametric family $\{\PP_{W_\theta}\}_{\theta \in \Theta}$ (see \eqref{eq:w_theta density}) is identifiable (i.e., $\PP_{W_\theta} = \PP_{W_{\theta'}}$ implies $\theta = \theta'$), which follows from the injectivity of the Laplace transform for finite measures, alongside the identifiability of $\mu_\theta$ from R0 in \cref{assmp:prior}.
\end{proof}

Next, we prove \cref{lem:LLN}, using some  technical lemmas which are proved in \cref{sec:proof technical lemmas}.

\begin{lemma}\label{lem:change conditioning}
    Let $(X_p, Y_p)$ be a sequence of random variables. Then, $X_p \xp 0$ if and only if $X_p \mid Y_p \xp 0$. %
\end{lemma}

\begin{lemma}[Extension of Lemma B.3 in \cite{lee2025clt}]\label{lem:B3}
    Let $h$ be any $C^1$ function where $h' \in \mcp_r$, (i.e., $h'$ is bounded by a degree $r$ polynomial).
    For random variables $V \sim N(m, d), V_0 \sim N(m_0, d_0)$ with $d \le 2d_0$, we have
    $$|\EE h(V) - \EE h(V_0)| \lesssim_{d_0, r} |m-m_0| (1+ |m_0|^r + |m-m_0|^r) + |d-d_0|.$$
\end{lemma}
    
\begin{proof}[Proof of \cref{lem:LLN}]
\begin{enumerate}[(a)]
    \item We separate the proof into two steps. Throughout, $K$ denotes absolute constants whose exact value may differ line-by-line. 

    \vspace{2mm}
    \emph{Step 1.} Recall that $w \mid X, \beta^\star \sim N(\Sigma \beta^\star, \Sigma)$ where $\Sigma=\sigma^{-2}X^{\top}X$. 
    By the Gaussian Poincare inequality \citep[e.g. Theorem 3.25,][]{van2014probability}, the growth condition, and \cref{lem:convergence for F_i}, we have 
    \begin{align*} \Var\Big(\frac{1}{p} \sumin g (w_i) \mid X, \beta^\star \Big) &\le \frac{1}{p^2} \sumin \EE[g' (w_i)^2 \mid X, \beta^\star] \\
    &\le \frac{1}{p^2} \sumin \Big[1+\EE \big[|w_i|^{2r} \mid X,\bbst\big] \Big] = O_P\Big(\frac{1}{p}\Big). \end{align*}

    For notational simplicity, let $\phi(m, d) := \EE g(N(m, d))$. This allows us to express the conditional mean of $g(w_i)$ as: $$\EE[g(w_i) \mid X, \beta^\star] = \phi(d_i \beta_i^\star - \mathcal{F}_i, d_i),$$ 
    and the variance bound implies
    $$\frac{1}{p} \sumin g(w_i) - \frac{1}{p} \sumin \phi(d_i \beta_i^\star - \mathcal{F}_i, d_i) \mid X, \bbst \xp 0.$$

    \emph{Step 2.} %
    Noting that $\EE g(W_{\theta_0}) = \EE \Big[\EE \big[g(W_{\theta_0}) \mid B_{\theta_0}\big]\Big] = \EE \phi(d_0 B_{\theta_0}, d_0)$, we claim $\frac{1}{p} \sumin \phi(d_i \beta_i^\star - \F_i, d_i) \mid X \xp \EE \phi(d_0 B_{\theta_0}, d_0).$ It suffices to individually show: \begin{align}\label{eq:lln suffices to show} \frac{1}{p} \sumin \big[\phi(d_i \beta_i^\star - \F_i, d_i) - \phi(d_0 \beta_i^\star, d_0)\big] \mid X &\xp 0, \quad \frac{1}{p} \sumin \phi(d_0 \beta_i^\star, d_0) \xp \EE \phi(d_0 B_{\theta_0}, d_0). \end{align} 
    The second claim in \eqref{eq:lln suffices to show} follows from the LLN for i.i.d. random variables. To show the first claim in \eqref{eq:lln suffices to show}, we use \cref{lem:B3} to write 
    \begin{align} 
        &|\phi(d_i \beta_i^\star -\F_i, d_i) - \phi(d_0 \beta_i^\star, d_0)| \\
        \lesssim_{d_0,r} & |d_i-d_0|+|(d_i-d_0)\beta_i^\star - \F_i| \Big(1+ |d_0 \beta_i^\star|^r + |(d_i-d_0)\beta_i^\star - \F_i|^r \Big) \notag \\ 
        \lesssim_{d_0,r} & |d_i-d_0|+\big(|(d_i-d_0)\beta_i^\star|+|\F_i|\big) \Big(1+|\beta_i^\star|^r + |\F_i|^r\Big). \label{eq:di-d0}
    \end{align}
    By Cauchy Schwartz alongside moment bounds, \begin{align*} &\Big[\sumin \big[\phi(d_i \beta_i^\star - \F_i, d_i) - \phi(d_0 \beta_i^\star, d_0)\big] \Big]^2 \\ 
    \lesssim &p\sumin (d_i-d_0)^2 + \sumin \Big(|(d_i-d_0)\beta_i^\star|^2+|\F_i|^2 \Big) \sumin \Big(1+ |\beta_i^\star|^{2r} + |\F_i|^{2r} \Big) \\
    = & O_{P,X}\Big(\frac{p^2}{n}\Big) + O_{P}\Big(\frac{p}{n} + \frac{p^2}{n}\Big) \Big(p + \sumin |\beta_i^\star|^{2r} + \sumin |\F_i|^{2r} \Big) \\ =& O_{P,X}\Big(\frac{p^2}{n}\Big) + O_{P}\Big(\frac{p^3}{n}\Big) = o_{P}(p^2), \end{align*} which implies the first claim. Here, the third line used simple expectation computations: $\EE \sumin (d_i-d_0)^2 = p/n, \EE [\F_i^2 \mid \bbst] \le \frac{\|\bbst\|^2}{n}$. The fourth line used \cref{lem:convergence for F_i}. Finally, the conditional limit in \eqref{eq:lln suffices to show} follows from \cref{lem:change conditioning}.
    
\item Given a function $h \in \mcp(r)$, define a function $H: \R\to \R$ as 
        $$H(t) := \EE h(B_{t}) = \frac{\int h(b) \exp[t b- \frac{d_0 b^2}{2}] d \mu_{\theta_0}(b)}{\int \exp\big[t b - \frac{d_0 b^2}{2}\big] d \mu_{\theta_0}(b)}.$$
    Note that taking the derivative and using \cref{lem:sub-Gaussianity} gives
    $$|H'(t)| = |\Cov(B_t, h(B_t))| \lesssim \EE [|B_t| + |B_t|^{r+1}] \lesssim 1 + |t|^{3(r+1)},$$
    so we have $H'\in \mcp(3(r+1))$.

    Using this notation, we have
    \begin{align*}
        \frac{1}{p}\sumin \EE h(B_{w_i,d_i}) &=\frac{1}{p}\sumin \EE h(B_{w_i,d_0}) + o_{P,X}(1) \\
        &= \frac{1}{p} \sumin H(w_i) + o_{P,X}(1) \xp \EE H(W_{\theta_0}) = \EE h(B_{\theta_0}).
    \end{align*}
    Here, the first equality used \cref{lem:B4} (noting that $h$ has polynomial growth), followed by using \cref{lem:convergence for F_i} to bound the error term. The limit in the last line follows from part (a) of this lemma (noting that $H'$ has polynomial growth). The final equality uses the observation $B_{\theta_0} \mid W_{\theta_0} \stackrel{d}{\equiv} B_{W_{\theta_0}, d_0}$ to write
    $$H(W_{\theta_0}) = \EE \big[h(B_{\theta_0}) \mid W_{\theta_0}\big].$$
    The second identity follows by modifying the above argument by setting
    $H(w) := [\EE h(B_{w})]^2.$
\end{enumerate}
\end{proof}

\begin{remark}\label{rmk:change centering}
    In part (a) of the above proof, by modifying the computations in \eqref{eq:di-d0} and the following line, we can show the following:
    \begin{align*}
        &\Big[\sumin \phi(d_i\beta_i - \F_i, d_i) - \phi(d_0\beta_i^\star - \F_i, d_0) ] \Big]^2 \\
        \lesssim_{d_0,r}& p\sumin (d_i-d_0)^2 + \sumin |(d_i-d_0) \beta_i^\star|^2 \sumin (1+ |\beta_i^\star|^{2r} + |\F_i|^{2r}) = O_P\Big(\frac{p^2}{n}\Big).
    \end{align*}
    Note that by replacing the first argument of each $\phi$ in the summand from $d_i \beta_i^\star - \F_i$ to $d_0 \beta_i^\star - \F_i$ (as opposed to replacing to $d_0 \beta_i^\star$ as in \eqref{eq:di-d0}), we have a better error bound. We will use this fact to prove the CLT \cref{lem:linear clt}, where a tighter bound is required.
\end{remark}

To prove \cref{lem:linear clt}, we require a crucial technical lemma, which provides an improved moment bound compared to Step 2 in the proof of \cref{lem:LLN}. %

\begin{lemma}\label{lem:A7}
    Suppose R5 holds. Let $\Phi : \R \to \R$ be a $\mathcal{C}^3$ function with $\Phi', \Phi'', \Phi''' \in \mcp(r)$. Then, for $\kappa_{\Phi}(\theta) := \frac{1}{2}\EE [B_{\theta}^2] \EE[\Phi''(d_0 B_{\theta})]$, we have
    $$\sumin \Phi(d_0 \beta_i^\star - \mathcal{F}_i) = \sumin \Phi(d_0 \beta_i^\star) + \frac{p^2}{n} \kappa_{\Phi}(\theta_0) + O_{P}\Big(\frac{p}{\sqrt{n}} + \frac{p^{5/2}}{n^{3/2}} \Big).$$
    Here, the hidden constant only depends on $d_0, r$, and the moments of $B_{\theta_0}$.
\end{lemma}

\noindent Finally, we will need a technical result which is a restatement of Lemma A.13 in \cite{lee2025bayesregression}.

\begin{lemma}[Conditional converging together]\label{lem:converge together}
    Let $G_p = G_p(y,X,\bbst), ~ \mu_p(X,\bbst)$ be $k$-dimensional random vectors, and let $T_1, T_2 \succ 0$ be $k \times k$ deterministic matrices. Suppose $G_p \mid X,\bbst \xd N(0, T_1)$ and $\mu_p \mid X \xd N(0, T_2).$ Then, we have $G_p + \mu_p \mid X \xd N(0, T_1+T_2).$    
\end{lemma}

\begin{proof}[Proof of \cref{lem:linear clt}]
(a) We first prove part (a). Here, all results are stated for general $(p,n)$, and we will specify whenever the assumption $\frac{p^{3/2}}{n} \to \delta$ is used. Throughout this proof, $\theta = \theta_0$ is assumed and we abbreviate $B_{t,d} = B_{t,d,\theta_0}$.

Since $\nabla F_i(\theta_0) = \EE \nabla \ell(\theta_0; B_{w_i, d_i})$, using \cref{lem:B4} followed by \cref{lem:convergence for F_i} gives:
\begin{align*}
    \sumin \nabla F_i(\theta_0) = \sumin \EE \nabla \ell(\theta_0; B_{w_i, d_i}) &= \sumin \EE \nabla \ell(\theta_0; B_{w_i, d_0}) + O_{P}\Big(\sumin |d_i-d_0| (1+|w_i|^{3(r+2)}) \Big) \\
    &= \sumin \EE \nabla \ell(\theta_0; B_{w_i, d_0}) + O_{P}\Big(\frac{p}{\sqrt{n}}\Big).
\end{align*}
For any $t \in \R$, define a vector-valued function $h:\R \to \R^k$ by 
\begin{align}\label{eq:def h}
    h(t) = (h_1(t), \ldots,  h_k(t))^\top := \EE \nabla \ell(\theta_0; B_{t, d_0}).
\end{align}
Then, it suffices to prove that
\begin{align}\label{eq:wts clt}
    \frac{1}{\sqrt{p}} \sumin h(w_i) \xd N(\delta \kappa(\theta_0), V(\theta_0)) \quad \text{as} \quad \frac{p^{3/2}}{n} \to \delta.
\end{align}
We separate the proof into two main steps. In the remainder of the proof, the quadratic tilt parameter will always be $d_0$, so we further abbreviate $B_{t} = B_{t,d_0,\theta_0}$.
\vspace{3mm}

\noindent \emph{Step 1:} Use the second-order Poincare inequality to get conditional normality.

Denote the conditional (on $X, \bbst$) mean and variance of the LHS of \eqref{eq:wts clt} as: $$\mu_p (X,\beta^\star) := \frac{1}{\sqrt{p}} \sumin \EE\Big[ h(w_i) \mid X, \beta^\star \Big], \quad \tau^2_p(X,\beta^\star) := \frac{1}{p} \Var \Big[ \sumin 
 h(w_i) \mid X, \beta^\star \Big]$$
 We claim the conditional limit:
$$\frac{1}{\sqrt{p}} \sumik h(w_i) - \mu_p(X,\beta^\star) \mid X,\bbst \xd N\Big(0, \tau_p^2(X,\beta^\star)\Big).$$
We use the Cramér–Wold device to show vector-valued convergence. It suffices to show the following for any constant vector $\alpha \in \R^k$ with $\|\alpha\|=1$:
$$\frac{1}{\sqrt{p}} \alpha^\top \Big[\sumik h(w_i) - \mu_p(X,\beta^\star)\Big] \mid X,\bbst \xd N\Big(0, \alpha^\top \tau_p^2(X,\beta^\star) \alpha\Big).$$

We prove this claim for any fixed vector $\alpha$.
By applying the unidimensional Gaussian Poincare inequality (see Theorem 2.2 in \cite{chatterjee2009fluctuations}), the TV distance between the LHS and the RHS (still conditioned on $X, \bbst$) is bounded above by
$\frac{\kappa_1 \kappa_2}{\tau_p^2(X,\beta^\star)},$
where 
$$\kappa_1^4 := \frac{1}{p^2} \EE\Big[ \Big[\sumin \big[\sum_{a=1}^k \alpha_a h_a'(w_i) \big]^2 \Big]^2 \mid X, \bbst \Big], \quad \kappa_2^4 := \frac{1}{p^2} \EE \Big[\maxin \big[\sum_{a=1}^k \alpha_a h_a''(w_i)\big]^4 \mid X, \bbst \Big].$$
$\alpha^\top \tau_p^2(X,\beta^\star) \alpha$ is bounded away from 0 by computations in the next step (see Step 2-2, which shows that $\tau_p^2(X,\beta^\star)$ converges in probability to a positive definite matrix), and it suffices to show $\kappa_1 = O_P(1), \kappa_2 = o_P(1)$.

To bound $\kappa_1$, use Cauchy-Schwartz (twice, for each square) alongside $\|\alpha\|=1$ to write
\begin{align*}
    \kappa_1^4 \le \frac{1}{p^2} \EE \Big[\big[\sumin \sum_{a=1}^k h_a'(w_i)^2\big]^2 \mid X,\bbst \Big] \le \frac{k}{p} \sumin \sum_{a=1}^k \EE\Big[h_a'(w_i)^4 \mid X,\bbst \Big]
\end{align*}
Recalling the definition of $h$ from \eqref{eq:def h} and moving the differential inside the expectation, we get %
\begin{align*}
    h'(t) &= \Cov(\nabla \ell(\theta_0; B_{t}), B_{t}), \\
    h''(t) &= \Cov(B_{t}\nabla \ell(\theta_0; B_{t}), B_{t}) - \EE [B_{t}] \Cov(\nabla \ell(\theta_0; B_{t}), B_{t}) - \EE [\nabla \ell(\theta_0; B_{t})] \Var(B_{t}).
\end{align*}
Hence, recalling from R4 that $\nabla \ell(\theta_0; b) \in \mcp(r)$ and using \cref{lem:sub-Gaussianity}, we have 
\begin{align}\label{eq:h_a' bound}
    |h_a'(t)| \le \EE|B_{t} \partial_{\theta_a} \ell(\theta_0; B_{t})| \lesssim \EE |B_{t}| + \EE |B_{t}|^{r+1} \lesssim_r 1+|t|^{3(r+1)}, \quad \forall a, t.
\end{align}
Hence, by taking $t = w_i$ and taking an outer expectation over $w_i \mid X, \bbst$, we have
\begin{align}\label{eq:h_a' bound 2}
    \EE\big[h_a'(w_i)^4 \mid X,\bbst\big] \lesssim_r 1 + \EE |w_i|^{12(r+1)}, \quad \forall a.
\end{align}
Plugging in this into the upper bound for $\kappa_1^4$ and using \cref{lem:convergence for F_i} gives
\begin{align*}
    \kappa_1^4 \lesssim \frac{k^2}{p} \sumin (1 + \EE |w_i|^{12(r+1)}) = O_P(1).
\end{align*}

Moving on to bounding $\kappa_2$, a similar application of Cauchy-Schwartz followed by the naive bound $\max(v_1,\ldots, v_p) \le v_1+\ldots+v_p$ gives
\begin{align*}
    \kappa_2^4 &\le \frac{1}{p^2} \EE \Big[\maxin \big[\sum_{a=1}^k h_a''(w_i)^2\big]^2 \mid X,\bbst \Big] \le \frac{k}{p^2} \EE \Big[\maxin \big[\sum_{a=1}^k h_a''(w_i)^4\big] \mid X,\bbst \Big] \\
    &\le \frac{k}{p^2} \EE \Big[\sumin \big[\sum_{a=1}^k h_a''(w_i)^4\big] \mid X,\bbst \Big] = O_P\Big(\frac{1}{p}\Big).
\end{align*}
Here, the final bound follows from repeating the computations \eqref{eq:h_a' bound}-\eqref{eq:h_a' bound 2} for $h_a''$.

\vspace{3mm}

\noindent \emph{Step 2:} Marginalize out $\beta^\star$.

Now, to prove \eqref{eq:wts clt}, it suffices to separately prove the following:
$$\mu_p (X, \beta^\star) \xd N(\delta \kappa(\theta_0), T_1), \quad \tau^2_p(X, \beta^\star) \xp T_2, \quad T_1 + T_2 = V(\theta_0),$$
where $T_1, T_2 \succ 0$ are deterministic $k\times k$ matrices (depending on $\theta_0$) defined as follows:
\begin{align}\label{eq:T_1, T_2}
    T_1 := \Var\Big[\EE \big[h(W_{\theta_0}) \mid B_{\theta_0} \big] \Big], \quad T_2 := \EE\Big[\Var \big[h(W_{\theta_0}) \mid B_{\theta_0} \big] \Big].
\end{align}
Assuming this, \eqref{eq:wts clt} follows from \cref{lem:converge together}.
We separate out the proof of each claim into three sub-steps.

\vspace{3mm}
\noindent \emph{Step 2-1:} $\mu_p (X, \beta^\star) \mid X \xd N(\delta \kappa(\theta_0), T_1)$.

For any $m\in\R$, define a random variable $W_m := N(m, d_0)$ and a vector-valued function $\Phi:\R \to\R^k$ as $\Phi(m) := \EE h(W_m)$. Then, we have
\begin{align}\label{eq:derivatives of phi}
    \Phi'(m) = \EE h'(N(m, d_0)), \quad \Phi''(m) =  \EE h''(N(m, d_0)).
\end{align}
Here, note that for each $a \in [k]$, we have $\Phi'_a\in \mcp(3r+3), \Phi''_a \in \mcp(3r+6), \Phi'''_a \in \mcp(3r+9)$. To see this, by using \eqref{eq:h_a' bound} followed by \cref{lem:sub-Gaussianity} (conditional on $W_m$) and Gaussian moment bounds, we have 
$$|\Phi_a'(m)| = |\EE h_a'(W_m)| \le \EE |h_a'(W_m)| \lesssim_{r,d_0} 1+\EE |W_m|^{3(r+1)} \lesssim_{r,d_0}  1 + |m|^{3(r+1)}.$$
The bound for $\Phi''_a$ follows from similar computations.

Now, by sequentially applying \cref{rmk:change centering} and \cref{lem:A7} for each coordinate of $h$ (with $\Phi$ defined as above), we have
\begin{align}
    \mu_p(X,\beta^\star) &= \frac{1}{\sqrt{p}} \sumin \EE\Big[ h(w_i) \mid X, \beta^\star \Big] = \frac{1}{\sqrt{p}} \sumin \Phi(d_0 \beta_i^\star - \F_i) + \OPX\Big(\sqrt{\frac{p}{n}}\Big) \label{eq:mean term expansion} \\
    &= \frac{1}{\sqrt{p}} \sumin \Phi(d_0 \beta_i^\star) + \frac{p^{3/2}}{n} \kappa(\theta_0) + O_{P,X}\Big(\sqrt{\frac{p}{n}} + \frac{p^2}{n^{3/2}}\Big).\notag
\end{align}

\noindent The limit for $\mu_p(\bbst)$ follows by taking the limit $\frac{p^{3/2}}{n} \to \delta$, and using the usual i.i.d. CLT for $\frac{1}{\sqrt{p}}\sumin \Phi(d_0 \beta_i^\star)$ alongside the moment computations:
$$\EE_{\beta_i^\star \sim B_{\theta_0}} \Phi(d_0 \beta_i^\star ) = \EE \nabla \ell(\theta_0; B_{\theta_0}) =  0, ~ \Var_{\beta_i^\star \sim B_{\theta_0}} \big[\Phi(d_0 \beta_i^\star ) \big] = \Var\Big[\EE \big[h(W_{\theta_0}) \mid B_{\theta_0} \big] \Big] = T_1.$$

\vspace{3mm}
\noindent  \emph{Step 2-2:} $\tau^2_p(\beta^\star) \xp T_2$.

By expanding the variance as
$$\Var \Big(\sumin h(w_i) \mid X, \beta^\star\Big) = \sumin \Var (h(w_i) \mid X, \beta^\star) + \sum_{i \neq j} \Cov(h(w_i), h(w_j) \mid X, \beta^\star),$$
it suffices to separate show the following:
$$\frac{1}{p} \sumin \Var(h(w_i)) = \underbrace{\EE\Big[\Var \big[h(W_{\theta_0}) \mid B_{\theta_0} \big] \Big]}_{=T_2} + o_{P,X}(1), ~ \frac{1}{p} \sumij \Cov(h(w_i), h(w_j)) = o_{P,X}(1).$$
In the remainder of this step, we omit the conditioning on $X,\bbst$ for simplicity.

To show the first claim, define a function $\tilde{\Phi}(m) := \Var[h(N(m, d_0))]$. By similar computations as in \eqref{eq:derivatives of phi}, the first and second derivatives of $\tilde{\Phi}$ are bounded by polynomials of appropriate degree. Hence, we can apply \cref{rmk:change centering} and \cref{lem:A7} to write
\begin{align*}
    \frac{1}{p} \sumin \Var(h(w_i)) & = \frac{1}{p} \sumin \tilde{\Phi}(d_0 \beta_i^\star - \mathcal{F}_i) + \OPX\Big(\frac{1}{\sqrt{n}}\Big) \\
    &= \frac{1}{p} \sumin \tilde{\Phi}(d_0 \beta_i^\star) +\OPX\Big(\frac{1}{\sqrt{n}} + \frac{p}{n}\Big) \xp \EE  \tilde{\Phi}(d_0 \beta_1^\star) = T_2.
\end{align*}
The final equality follows by plugging-in the definition of $\tilde{\Phi}$.

To show the second claim, define $\tilde{w}_j^i := w_j + \frac{A_p(i,j)}{d_i} w_i$ so that $\Cov(\tilde{w}_j^i, w_i)=0$ and $\tilde{w}_j^i$ and $w_i$ are independent. For any $i \neq j$, a Taylor expansion of $h(w_j)$ around $w_j \approx \tilde{w}_j^i$ gives
\begin{align*}
    &  \Cov \big(h(w_i), h(w_j)\big) \\
    =& \Cov \big(h(w_i), h(\tilde{w}_j^i)\big) - \frac{A_p(i,j)}{d_i} \Cov \big(h(w_i), w_i h'(\tilde{w}_j^i)\big) + \frac{A_p(i,j)^2}{2d_i^2} \Cov\big(h(w_i), w_i^2 h''(\xi_{i,j})\big) \\
    =& - A_p(i,j) \EE h'(w_i) \EE h'(\tilde{w}_j^i)^\top + O \Big(A_p(i,j)^2 \big(\EE \|h(w_i)\|^2 + \EE [w_i^2 \|h''(\xi_{i,j})\|^2 ] \big)  \Big) \\
    =& - A_p(i,j) \EE h'(w_i) \EE h'(w_j)^\top + O \Big(A_p(i,j)^2 \big( \EE [\|h(w_i)\|^2 +\|h'(w_i)\|^2]+ \EE [w_i^2 (\|h''(\xi_{i,j})\|^2 \\
    &\quad + \|h''(\eta_{i,j})\|^2) ] \big) \Big) \\
    =& - A_p(i,j) \EE h'(w_i) \EE h'(w_j)^\top + O \Big(A_p(i,j)^2 \EE\big[1 + \EE|w_i|^{r'} + \EE|w_j|^{r'}\big]\Big)
\end{align*}
for $\xi_{i,j}, \eta_{i,j} \in (w_j, \tilde{w}_j^i)$ and some positive integer $r'$ that depends on $r$.
Here, the second equality follows from expanding the covariance and using Stein's identity:
\begin{align*}
    \Cov\big(h(w_i), w_i h'(\tilde{w}_j^i)\big) &= \EE \big[(h(w_i) - \EE h(w_i)) w_i\big] \EE h'(\tilde{w}_j^i) \\
    &= \EE \big[h(w_i) (w_i- \EE w_i) \big] \EE h'(\tilde{w}_j^i) = d_i \EE h'(w_i) \EE h'(\tilde{w}_j^i), 
\end{align*}
and using the basic covariance inequality $\Cov(V,W) \le \sqrt{\Var V^2 \Var W^2} \le \frac{\EE V^2 + \EE W^2}{2}$ for each entry of the covariance matrix. 
The fourth line uses another Taylor expansion of $h'$: $h'(\tilde{w}_j^i) = h'(w_j) + \frac{A_p(i,j) w_i}{d_i} h''(\eta_{i,j})$ followed by a Cauchy Schwartz.
The final line simplifies the error term $\EE \|h'(w_i)\|^2$ using \eqref{eq:h_a' bound}--\eqref{eq:h_a' bound 2}:
$$\EE \|h'(w_i)\|^2 = \sum_{a=1}^k \EE [h_a'(w_i)^2] \lesssim 1 + \EE |w_i|^{6(r+1)}.$$
The remaining error terms can be similarly bounded by $O\big(\EE|{w_i}|^{r'} + \EE|{w_j}|^{r'}\big)$ for a large enough $r'$, using the naive bound $|A_p(i,j)| \le |d_i|$ to write
$|\xi_{i,j}| \le |w_j| + \frac{|A_p(i,j)|}{d_i}|w_i| \le |w_j| + |w_i|$.

Now, by summing up the long display above for $i \neq j$, we get
\begin{align*}
    & \frac{1}{p} \sumij \Cov(h(w_i), h(w_j)) \\
    =&- \frac{1}{p} \sumij A_p(i,j) \EE h'(w_i) \EE h'(w_j)^\top + \frac{1}{p} O \Big(\sumij A_p(i,j)^2 \EE|w_i|^{r'} \Big) \\
    =& \OPX \Big( \sqrt{\frac{p}{n}} \Big) + O_P\Big(\frac{p}{n}\Big) =o_{P,X}(1).
\end{align*}
The last line uses the operator norm bound of $A_p$ for the first term (see \cref{lem:alpha_p}), and uses \cref{lem:convergence for F_i}:
$$\frac{1}{p}\sumij A_p(i,j)^2 \EE |w_i|^{r'} \le \frac{\alpha_p}{p} \sumin \EE |w_i|^{r'} = O_P(\alpha_p) = O_P\Big(\frac{p}{n}\Big)$$
to control the second term (recall the notation $\alpha_p = \maxin \sum_{j \neq i} A_p(i,j)^2$ from \cref{lem:alpha_p}).

\vspace{3mm}
\noindent  \emph{Step 2-3:} $T_1 + T_2 = V(\theta_0)$.

Recalling \cref{def:information}, define a pair of random variables $(B_{\theta_0}, W_{\theta_0})$. 
By recalling the definitions of $T_1, T_2$ in \eqref{eq:T_1, T_2}, and applying the law of total variance, we have
\begin{align*}
    T_1 + T_2 &= \Var\Big[\EE \big[h(W_{\theta_0}) \mid B_{\theta_0} \big] \Big] + \EE\Big[\Var \big[h(W_{\theta_0}) \mid B_{\theta_0} \big] \Big] \\
    &= \Var \big[ h(W_{\theta_0}) \big]\\
    &= \Var \Big[\EE \big[\nabla \ell(\theta_0; B_{W_{\theta_0}, d_0}) \mid W_{\theta_0} \big] \Big] \\
    &= \Var[\nabla \ell(B_{\theta_0})] - \EE \Big[\Var \big[ \nabla \ell(\theta_0; B_{W_{\theta_0}, d_0}) \mid W_{\theta_0} \big] \Big] = V(\theta_0).
\end{align*}
The third line plugs-in the definition of $h$.
The last line follows from noting that $B_{\theta_0} \equiv B_{W_{\theta_0}, d_0}$, and again applying the law of total variance. 
This completes the proof of part (a).

\vspace{2mm}
(b) Now, we work under the scaling $n^{2/3} \ll p \ll n$. Note that steps 1 and 2-2 above only require $p \ll n$, and still holds. Consequently by Slutsky, it only suffices to modify step 2-1 as follows, under the modified asymptotic scaling:
    $$\frac{n}{p \sqrt{p}} \mu_p (X, \beta^\star) \xp \kappa(\theta_0).$$
    Note that this limit has a slower convergence rate compared to the limit in step 1, and dominates the law of $\sumin \nabla F_i(\theta_0)$.
    
To wit, recall the notation $\Phi(m) := \EE[h(W_m)]$ from \eqref{eq:derivatives of phi} and recall that rescaling \eqref{eq:mean term expansion} by $\frac{n}{p\sqrt{p}}$ gives
\begin{align*}
    \frac{n}{p\sqrt{p}}\mu_p(X,\beta^\star) &= \frac{n}{p^2} \sumin \Phi(d_0 \beta_i^\star) + \kappa(\theta_0) + O_{P}\Big(\frac{\sqrt{n}}{p} + \sqrt{\frac{p}{n}}\Big) = \kappa(\theta_0) + o_{P}(1).
\end{align*}
Here, the last line used the fact that $\sumin \Phi(d_0\beta_i^\star) = O_P(\sqrt{p})$ and the asymptotic scaling $p \gg n^{2/3}$. This completes the proof.

\vspace{2mm}
(c) Finally, suppose $p \ll n^{3/4}$. Again, by the computations in part (a), it suffices to show:
$$\mu_p(X,\bbst) - \frac{p\sqrt{p}}{n} \kappa(\theta_0) \mid X \xd N(0, T_2).$$
But this is a direct consequence of \eqref{eq:mean term expansion}, noting that the error term there is $o_{P,X}(1)$.
\end{proof}

\begin{proof}[Proof of \cref{lem:sub-Gaussianity}]
    For simplicity, define the c.g.f. $\psi_{d,\theta}(t) := \log \int e^{b t - \frac{d b^2}{2}} d \mu_\theta(b).$
    Using the exponential tilt properties and the assumption $\|B_{t,d,\theta}\|_{\psi_2} \lesssim 1+|t|$, we have
    $$\psi_{d,\theta}''(t) = \Var (B_{t,d,\theta}) \lesssim 1 + t^2, \quad \forall t \in \R.$$
    Hence, by integrating over $t$, we get 
    \begin{align}\label{eq:tilt expectation}
        \EE B_{t,d,\theta} = \psi_{d,\theta}'(t) \le \EE B_{0,d,\theta} + O(|t| + |t|^3) = O(1+|t|^3)
    \end{align}
    and the second claim holds.
    The simplification in the last equality used the assumption that $\EE B_{0,d,\theta}$ is uniformly bounded.  
    Then, the first claim also follows as
    $$\psi_{d,\theta}(t) \le \psi_{d,\theta}(0) + O(|t|^3+|t|^4) = O(|t|^3+|t|^4).$$
    The last equality uses the fact $\psi_{d,\theta}(0) = \log \int e^{-\frac{d b^2}{2}} d\mu_\theta(b) \le 0.$
    
    It remains to show the claims regarding $B_{t,d,\theta}$. Again by the sub-Gaussianity of $B_{t,d,\theta}$ and \eqref{eq:tilt expectation}, we have
    \begin{align*}
        \EE |B_{t,d,\theta}|^r \lesssim_r (1+|t|)^r + [\EE B_{t,d,\theta}]^r \lesssim_r 1 + |t|^{3r}.
    \end{align*}
    Now, by taking $t = w_i \stackrel{d}{\equiv} N(d_i \beta_i^\star - \F_i, d_i)$ and using the tower property, we can further write
    $$\EE \big[ |B_{w_i, d,\theta}|^r \mid X,\bbst\big] \lesssim_r 1 + \EE \big[ |w_i|^{3r} \mid X,\bbst\big] \lesssim_{r, d_0} 1 + |\beta_i^\star|^{3r} + |\F_i|^{3r}, \quad \forall d, \theta.$$
    Hence, using \cref{lem:convergence for F_i} and the assumption R5 that $\beta_i^\star$ has finite moments, we have $\sumin \EE |B_{w_i, d,\theta}|^r = O_P(p).$
\end{proof}

\begin{proof}[Proof of \cref{lem:B4}]
    Set $G(d) := \log \int e^{t b - \frac{d b^2}{2}} d \mu_\theta(b)$ and recall the random variable $B_{t,d,\theta}$ from \cref{def:information}. Then, we have
    $G'(d) = -(1/2)\EE [B^2_{t,d,\theta}],$ so Taylor expansion gives
    $$|G(d) - G(d_0)| = |d-d_0| |G'(\xi)| \le |d-d_0| (1+|t|^6), \quad \exists \xi \in (d,d_0).$$
    The last bound used \cref{lem:sub-Gaussianity} with $r=2$.

    To show the second claim, for some $h \in \mcp(r)$, set $H(d) := \EE[h(B_{t,d,\theta})]$. Then, by a similar Taylor expansion noting that $|H'(d)| = \Big|\Cov\Big(h(B_{t,d,\theta}), -\frac{B_{t,d,\theta}^2}{2}\Big)\Big| \lesssim \EE|B_{t,d,\theta}|^{r+2}$ and using \cref{lem:sub-Gaussianity}, we have
    $$|H(d) - H(d_0)| = |d-d_0| |H'(\xi)| \le |d-d_0| \big(1+|t|^{3(r+2)}\big).$$
\end{proof}

We prove \cref{lem:convergence for F_i} using an additional concentration result for $\F_i = \sum_{j \neq i} A_p(i,j) \beta_j^\star$, which is proved in \cref{sec:proof technical lemmas}.
\begin{lemma}\label{lem:mgf of F_i}
    There exists some constant $C>0$ such that $\max_{i=1}^p \EE \big[e^{\lambda \mathcal{F}_i} \mid \bbst\big] \le e^{\frac{C p \lambda^2}{n}}$ for all $|\lambda| \le \frac{Cn}{\sqrt{p}}$ with high probability (with respect to $\bbst$).
\end{lemma}
\begin{proof}[Proof of \cref{lem:convergence for F_i}]
    We start by showing the first claim. Fix any constant $K>0$. It suffices to show that 
    \begin{align}\label{eq:c5show}
        \EE\left(\frac{1}{p}\sum_{i=1}^p e^{K|\F_i|} \mid \bbst\right) \xp 1 \qquad \mbox{and} \qquad \EE\left(\frac{1}{p}\sum_{i=1}^p e^{K|\F_i|} \mid \bbst\right)^2 \xp 1.
    \end{align}
    We will only prove the first display of \eqref{eq:c5show} as the second display follows similarly. Note that $e^{|x|} \le e^x + e^{-x}$. We then have 
    $$1 \le \limsup \frac{1}{p} \EE\left[\sumin e^{K |\F_i|} \mid \bbst\right] \le \limsup \max_{i=1}^p \EE \exp(K|\F_i|) \le 1,$$
    with high probability, as a consequence of \cref{lem:mgf of F_i}. The conclusion follows from a standard sandwiching argument.
    
    \noindent We move on to the second claim. Note that \cref{lem:mgf of F_i} implies that $\sum_{i=1}^p |\F_i|^r=O_P(p)$. Using the Gaussianity of $w_i \mid X, \bbst$ yields
    $$\sumin \EE |w_i|^r \lesssim_r \sumin \Big[|d_i\beta_i^\star - \F_i|^r + |d_i|^{\frac{r}{2}} \Big] \lesssim_r p+\sumin \Big[|d_i|^r (1+|\beta_i^\star|^r) + |\F_i|^r \Big] = O_P(p).$$
    This completes the proof.
\end{proof}

\subsection{Proof of lower bound lemmas}
This Section is devoted to the proof of \cref{thm:CLT}. Throughout the following proof, $B = (B_1, \ldots, B_p)$ will denote a sample from the RFIM $\tilde{\PP} = \tilde{\PP}_{\theta}^{A_p,c}$ (see \cref{def:rfim}). The proof proceeds through the use of Stein's method of exchangeable pairs which has served as a key tool for analyzing Ising models (see e.g.~\cite{deb2020fluctuations,Deb2024detecting,chatterjee2011nonnormal} and the references therein).

\begin{proof}[Proof of \cref{thm:CLT}]
    We start by deriving general results that will be used to prove both parts of the statement. 
    Let $g:[-1,1]\to[-1,1]$ be a continuous function.
    Define $$T_p = T_p(B) := \frac{1}{\sqrt{p}} \sumin \Big(g(B_i) - \langle g\rangle_{\mu_{i,s_i+c_i,\theta}} \Big).$$ Given an RFIM sample $B \sim \tilde{\PP}$, generate $B'$ as follows: let $I$ be a randomly sampled index from $\{1,2,\ldots ,p\}$. Given $I=i$, replace $B_i$ with an independently generated $B'_i$ drawn from the conditional distribution of $B_i$ given $(B_j,\ j\neq i)$.

    Recalling the notation $m_i$ from \cref{lem:RFIM lemma}, we can write
    \begin{align}\label{eq:conditional distribution of B_i}
        B_i \mid (B_j,\ j\neq i) \stackrel{d}{\equiv} \mu_{i,m_i+c_i,\theta}.
    \end{align}
    For notational convenience, define
    $$h_i(t) := \langle g \rangle_{\mu_{i,t,\theta}} = \EE_{B_i \sim \mu_{i,t,\theta}} g(B_i), \quad \forall t \in \R,$$
    and note that \eqref{eq:conditional distribution of B_i} implies $\EE[g(B_i) \mid B_j: j\neq i] = h_i(m_i + c_i).$ Also, observe that
    \begin{align}\label{eq:h_i derivatives}
        h_i'(t) = \Cov_{\mu_{i,t,\theta}}(B_i, g(B_i)), \quad h_i''(t) = \Cov_{\mu_{i,t,\theta}}\Big((B_i-\psi_{i,\theta}'(t))^2, g(B_i) \Big)
    \end{align}
    are uniformly bounded.
    By a Taylor's series expansion of  around $m_i \approx s_i$, we get that for some $\{\xi_i\}_{1\le i\le p}$:
	\begin{align}
		\EE(T_p &- T_p' \mid B) = \frac{1}{p \sqrt{p}} \sumin (g(B_i) - h_i(m_i+c_i) ) \notag \\
		&= \frac{1}{p \sqrt{p}} \sumin \left(g(B_i) - h_i(s_i+c_i) - (m_i - s_i)h_i'(s_i+c_i) - \frac{1}{2} (m_i-s_i)^2 h_i''(\xi_i+c_i) \right) \notag \\
		&= {\frac{T_p}{p}} - \underbrace{\frac{1}{p \sqrt{p}} \sumin (m_i - s_i) h_i'(s_i+c_i)}_{:=H_1}  - \underbrace{\frac{1}{2p\sqrt{p}} \sumin (m_i-s_i)^2  h_i''(\xi_i+c_i)}_{:=H_2}. \label{eq:stein}
	\end{align}

    We bound the error terms $H_1, H_2$ in the above expression.
    \begin{itemize}
        \item For $H_1$, we expand $m_i - s_i = \sum_{j \neq i} A_p(i,j)(B_j - u_j)$ and write
        $$H_1 = \frac{1}{p \sqrt{p}} \sumjn \zeta_j(B_j - u_j),$$
        where $\zeta_j := \sum_{i \neq j} A_p(i,j)h_i'(s_i+c_i)$.
        Then, by part (a) of \cref{lem:RFIM lemma}, we have
        $$p^2 \EE H_1^2 \lesssim \frac{1}{p}\big(\sumjn \zeta_j^2 \big) \Big(1 + \frac{p^3}{n^2}\Big) = \OPX \Big(\frac{p}{n}\Big).$$
        The final inequality follows from using the operator norm bound for $A_p$ in \cref{lem:alpha_p} to write $\sumjn \zeta_j^2 \lesssim p \|A_p\|^2 = \OPX(p^2/n)$.
        
        \item For $H_2$, using the fact that $h_i''$ are uniformly bounded followed by part (b) of \cref{lem:RFIM lemma}, we have
        $$p^2 \EE H_2^2 \lesssim \frac{1}{p}  \EE\Big[\sumin(m_i - s_i)^2\Big]^2 = \OPX \Big( \frac{p^3}{n^2} \Big).$$
    \end{itemize}
    Now, we are ready to prove each claims.

    (a) We prove the first moment bound. Since $(T_p, T_p')$ form an exchangable pair, we have $\EE T_p = \EE T_p'$. Hence, taking an outer expectation on \eqref{eq:stein} gives the identity
    $$\EE T_p = p\EE H_1 + p\EE H_2.$$
    The earlier error bounds for $H_1, H_2$ implies that the RHS is $\OPX\left(\sqrt{\frac{p}{n}} + \sqrt{\frac{p^3}{n^2}}\right)$.

    To prove the second conclusion, using a Taylor expansion followed by the operator norm bound of $A_p$ (see \cref{lem:alpha_p}) gives
    \begin{align*}
        &\sumin \Big[h_i(s_i+w_i) - h_i(w_i) \Big] \\
        =& \sumin \Big[s_i h_i'(\xi_i + w_i)\Big] \sumij A_p(i,j) s_i h_i'(\xi_i + w_i) u_j = \OPX\Big(\frac{p\sqrt{p}}{\sqrt{n}}\Big).
    \end{align*}
    Here, we have used the fact that $h_i'$ is uniformly bounded, which follows from the compact support assumption.

    \vspace{2mm}
    (b) Next, we prove the second moment bound.
    Let $\Delta_p := T_p - T_p' = \frac{1}{\sqrt{p}}(g(B_I) - g(B_I'))$ and note that $|\Delta_p|\le \frac{2}{\sqrt{p}}$. Multiplying \eqref{eq:stein} by $p T_p$ and taking an outer expectation gives
    $$p \EE [\Delta_p T_p] = \EE[T_p^2] - (p \EE [T_p H_1] + p\EE [T_p H_2]).$$
    On the other hand, by exchangeability of $(T_p, T_p')$, we can also write
    $$p \EE [\Delta_p T_p] = \frac{p \EE[\Delta_p^2]}{2} \le 2.$$
    Hence, by upper bounding
    \begin{align}\label{eq:H_a bound stein}
        |p \EE [T_p H_a]| \le \sqrt{p^2 \EE H_a^2} \sqrt{\EE[T_p^2]}, \quad a = 1,2,
    \end{align}
    we get the inequality 
    $$\EE T_p^2 \le 2+\sum_{a=1}^2 \sqrt{p^2 \EE H_a^2} \sqrt{\EE[T_p^2]}.$$
    By plugging-in the earlier bounds for $p^2 \EE H_a^2$, we get a preliminary bound $\EE[T_p^2] = \OPX(1)$.

    To compute the precise constant for $\EE[T_p^2]$, we expand the term $p\EE[\Delta_p^2]$ as follows:
    \begin{align*}
        p\EE[\Delta_p^2] &= \frac{1}{p} \EE\Big[ \sumin  [g(B_i) - g(B_i')]^2 \Big]\\
        &= \frac{2}{p} \sumin \Big[\EE g(B_i)^2 - \EE[g(B_i) \langle g \rangle_{\mu_{i,m_i+c_i,\theta}}] \Big],
    \end{align*}
    where the last line follows by the tower property (by conditioning on $(B_j:j\neq i)$ and using \eqref{eq:conditional distribution of B_i}) and using exchangeability.
    Now, using the two conclusions of part (c) in \cref{lem:RFIM lemma} (with $h \equiv g^2, g$, respectively) gives the following moment bounds:
    \begin{align*}
        \EE\Big[\frac{1}{p} \sumin g(B_i)^2 \Big] &= \EE\Big[\frac{1}{p} \sumin \langle g^2 \rangle_{\mu_{i,m_i+c_i,\theta}} \Big] +O\Big(\frac{1}{\sqrt{p}}\Big), \\
        \EE\Big[\frac{1}{p} \sumin g(B_i) \langle g \rangle_{\mu_{i,\theta}(m_i+c_i)} \Big] &= \EE\Big[\frac{1}{p} \sumin \big(\langle g \rangle_{\mu_{i,m_i+c_i,\theta}}\big)^2 \Big] +O\Big(\frac{1}{\sqrt{p}}\Big).
    \end{align*}
    We additionally simplify the RHS expressions using the following fact: %
    for any continuous function $h:[-1,1]\to[-1,1]$, we have
    $$\EE \Big[\frac{1}{p} \sumin h(m_i+c_i)\Big] = \frac{1}{p} \sumin h(c_i) + \frac{1}{p}\sumin|m_i| =  \frac{1}{p} \sumin h(c_i) + \OPX\Big({\sqrt{\frac{p}{n}}}\Big),$$
    which follows from bounding $\sumin |m_i| \le \sqrt{p} \sqrt{\sumin m_i^2} = \OPX\Big(\sqrt{\frac{p^3}{n}}\Big).$
    By taking $h$ appropriately, we can further simplify the individual moment bounds as:
    \begin{align*}
        \EE\Big[\frac{1}{p} \sumin g(B_i)^2 \Big] &= \frac{1}{p} \sumin \langle g^2 \rangle_{\mu_{i,c_i,\theta}} + \OPX\Big(\frac{1}{\sqrt{p}} + \sqrt{\frac{p}{n}}\Big), \\
         \EE\Big[\frac{1}{p} \sumin g(B_i) \langle g \rangle_{\mu_{i,\theta}(m_i+c_i)} \Big] &= \frac{1}{p} \sumin \big(\langle g \rangle_{\mu_{i,c_i,\theta}}\big)^2 + \OPX\Big(\frac{1}{\sqrt{p}} + \sqrt{\frac{p}{n}}\Big).
    \end{align*}

    By combining all computations above, we have the following:
    $$\EE[T_p^2] = \frac{1}{p} \sumin \Big[\langle g^2 \rangle_{\mu_{i,c_i,\theta}} - \big(\langle g \rangle_{\mu_{i,c_i,\theta}}\big)^2 \Big] + (p \EE [T_p H_1] + p\EE [T_p H_2]) + \OPX\Big({\frac{1}{\sqrt{p}} + \sqrt{\frac{p}{n}}}\Big).$$
    Using \eqref{eq:H_a bound stein} along the preliminary bound $\EE T_p^2 = O(1)$, we have
    $$|p \EE [T_p H_1] + p\EE [T_p H_2]| \le \sqrt{\EE T_p^2} \Big( \sqrt{p^2 \EE H_1^2} + \sqrt{p^2 \EE H_2^2}\Big) = \OPX\Big({\sqrt{\frac{p}{n}} + \sqrt{\frac{p^3}{n^2}}}\Big).$$
    This completes the proof. Note that the term $\sqrt{p/n}$ is omitted in the final bound since
    $$\frac{1}{\sqrt{p}} + \sqrt{\frac{p^3}{n^2}} \ge 2 \sqrt{\frac{p}{n}}.$$

\end{proof}

\subsection{Proof of lemmas for \cref{sec:consequences}}
\begin{proof}[Proof of \cref{lem:psi' bound}]
    Fix $t$, an index $i \le p$, a bounded function $h$ with $|h| \le 1$, and write $$\langle h \rangle_{\mu_{i,t,\theta}} = \frac{\int h(b) e^{t b - \frac{d_i b^2}{2} + \ell(\theta; b)} db}{\int e^{t b - \frac{d_i b^2}{2} + \ell(\theta; b)} db} = \frac{A(\theta)}{B(\theta)}.$$
    Recall that we assume in \cref{assmp:cpt support} that $\ell(\theta;b)$ is continuous in both $\theta$ and $b$, which implies that it is Lipschitz in both arguments. Hence, for any  function $g:[-1,1]\to[-1,1]$, we have
    \begin{align*}
        & \Big|\int g(b) e^{t b - \frac{d_i b^2}{2} + \ell(\theta; b)} (1- e^{\ell(\theta'; b) - \ell(\theta; b)}) db\Big|  \\
        \le& \sup_{b \in [-1,1]} \Big|1 - e^{\ell(\theta'; b) - \ell(\theta; b)}\Big| \int e^{t b - \frac{d_i b^2}{2} + \ell(\theta; b)} db \lesssim \|\theta - \theta'\| B(\theta).
    \end{align*}
    Here, the hidden constant only depends on the log-likelihood $\ell$.
    By taking $g(b) = h(b)$ and $1$ respectively, we have the following bounds:
    $$|A(\theta) - A(\theta')| \lesssim \|\theta - \theta'\| B(\theta), \quad |B(\theta) - B(\theta')| \lesssim \|\theta - \theta'\| B(\theta').$$
    Now, the claim follows by using the triangle inequality and plugging in the above bounds:
    $$|\langle h \rangle_{\mu_{i,t,\theta}} - \langle h \rangle_{\mu_{i,t,\theta'}}| = \left|\frac{A(\theta)}{B(\theta)} - \frac{A(\theta')}{B(\theta')}\right| \le \frac{|B(\theta)-B(\theta')|A(\theta)}{B(\theta) B(\theta')} + \frac{|A(\theta)-A(\theta')|}{B(\theta)} \lesssim \|\theta - \theta'\|.$$
    Note that we have also used the basic identity $|\frac{A(\theta)}{B(\theta)}|\le 1$ for the final simplification. %
\end{proof}

To prove \cref{lem:sum q'u approximate}, we use the following fact from the proof of Theorem 3.1, part (a) in \cite{lee2025bayesregression} (see the last two bullet points there).
\begin{lemma}\label{lem:R4 oracle bound}
    Take $\epsilon := A_p q$. For any $\theta$, set $$R_4(\theta) := \sumij A_p(i,j) q_i \big(\psi_{i,\theta}''(w_i) - \upsilon_p) \psi_{j,\theta}'(w_j)\big).$$ Then, fixing $\theta_0$, the following bounds hold:
    $$\sumin \epsilon_i \psi_{i,\theta_0}'(w_i) = o_P(1), \quad R_4(\theta_0) = o_P(1).$$
\end{lemma}

\begin{proof}[Proof of \cref{lem:sum q'u approximate}]
    By the triangle inequality, it suffices to separately prove the following:
    \begin{align}\label{eq:R4 two terms}
        \Big|\sumin q_i \big(\hat{u}_i - \psi_{i,\thetahat}'(w_i)\big)\Big| = o_P(1), \quad \Big|\sumin q_i \big(\psi_{i,\thetahat}'(w_i) - \psi_{0,\thetahat}'(w_i)\big)\Big| = o_P(1).
    \end{align}

    For simplicity, we use the notations $R_1, R_2$ from \cref{lem:posterior clt} and also set $$R_3 := \sumin \Big[\sumjn A_p(i,j) \psi_{j,\thetahat}'(w_j)\Big]^4.$$
    To show the first bound in \eqref{eq:R4 two terms}, apply Theorem 2.5(a) of \cite{lee2025clt} with $\lambda_n \equiv 0$ to write
    \begin{align}\label{eq:R4 expansion}
        \Big|\sumin q_i \hat{u}_i - \sumin q_i \psi'_{i,\thetahat}(w_i)\Big| &\lesssim \sqrt{R_1 R_2} + \sqrt{R_3} + \sqrt{R_2 \|A_p q\|} + R_4(\thetahat) + \big|\sumin \epsilon_i \psi_{i,\thetahat}'(w_i)\big| \notag \\
        &\lesssim R_4(\thetahat) + \big|\sumin \epsilon_i \psi_{i,\thetahat}'(w_i)\big| + o_{P}(1).
    \end{align}
    The simplification in the second line follow from %
    the bounds for $R_1, R_2$ in \eqref{eq:R_1, R_2} (and eq. (A.17) in \cite{lee2025bayesregression} to control $R_3$) to write
    $$R_1 R_2 = O_{P,X}\Big(\frac{p^3}{n^2}\Big), \quad R_3 = o_{P}(1), \quad R_2 \|A_p q\| = O_{P,X}\Big(\frac{p^3}{n^2}\Big).$$

    We show that the first and second term in \eqref{eq:R4 expansion} are both $o_P(1)$. Our main idea is to write
    $$\psi_{i,\thetahat}'(w_i) = \psi_{i,\theta_0}'(w_i) + E_i, \quad \psi_{i,\thetahat}''(w_i) = \psi_{i,\theta_0}''(w_i) + F_i,$$
    where the error terms are bounded as $|E_i|, |F_i| \lesssim \|\thetahat - \theta_0\|$ (by \cref{lem:psi' bound}).

    We first claim that the third term in \eqref{eq:R4 expansion} is $o_P(1)$. By using the triangle inequality, we have
    $$ |\sumin \epsilon_i \psi_{i,\thetahat}'(w_i)| \le  |\sumin \epsilon_i \psi_{i,\theta_0}'(w_i)| + \sumin |\epsilon_i| |E_i| = o_P(1).$$
    In the equality, the first term is controlled by \cref{lem:R4 oracle bound}. The second term is bounded via the assumption $\|\thetahat -\theta_0\|=O_P(1/\sqrt{p})$ alongside $\sumin |\epsilon_i| \le \sqrt{p} \|A_p q\| = o_P(\sqrt{p^2/n}).$
    
    For the middle term in \eqref{eq:R4 expansion}, we have
    {\small
    \begin{align*}
        &\big|\sumij A_p(i,j) q_i(\psi_{i,\thetahat}''(w_i)-\hat{\upsilon}_p) \psi_{j,\thetahat}'(w_j)\big| \\
        \le& R_4 (\theta_0) + \big|\sumij A_p(i,j) q_i(\psi_{i,\thetahat}''(w_i)-\hat{\upsilon}_p) E_j\big| + \big|\sumij A_p(i,j) q_i(E_i + \upsilon_p - \hat{\upsilon}_p ) \psi_{j,\thetahat}'(w_j)\big| \\
        =& o_P(1).
    \end{align*}}
    In the last equality, the second term is bounded using the operator norm bound
    $$|\sumij A_p(i,j) q_i(\psi_{i,\thetahat}''(w_i)-\hat{\upsilon}_p) E_j| \le \|A_p\| \sqrt{\sumjn E_j^2} = O_P\Big(\sqrt{\frac{p}{n}}\Big),$$
    and the third term also follows similarly by plugging in the size of $\hat{\upsilon}_p - \upsilon_p$ computed in part (a) of \cref{thm:EB}. The bound $R_{4}(\theta_0) = o_{P}(1)$ follows from \cref{lem:R4 oracle bound}. Hence, we have shown that the RHS of \eqref{eq:R4 expansion} is $o_P(1)$.

    Finally, the second bound in \eqref{eq:R4 two terms} follows from using \cref{lem:B3} with $m = m_0 = w_i$ and $d = d_i$ to get:
    $$\Big|\sumin q_i \big(\psi_{i,\thetahat}'(w_i) - \psi_{0,\thetahat}'(w_i)\big)\Big| \le \sumin |q_i||d_i-d_0| \le \|q\| \sqrt{\sumin (d_i-d_0)^2} = \OPX\Big(\sqrt{\frac{p}{n}}\Big).$$
    This completes the proof.
\end{proof}

\section{Proof of technical concentration lemmas}\label{sec:proof technical lemmas}

\begin{proof}[Proof of \cref{lem:mixture of Gaussians}]
    Let $\pi_m$ denote the probability of the $m$th mixture component. By the tower property and sub-Gaussianity of mixture components, we have
    $$\EE e^{\lambda X} \le \sum_{m=1}^M \pi_m e^{\lambda \mu_m + \frac{\lambda^2 K_{m}^2}{2}} \le e^{ \frac{\lambda^2 K_{\max}^2}{2}} \sum_{m=1}^M \pi_m e^{\lambda \mu_m}, \quad \forall \lambda \in \R.$$
    Let $Y$ be a categorical random variable with pmf $\PP(Y = \mu_m) = \pi_m$ for all $m \le M$. Then, by using Hoeffding's inequality, we can write
    $$\sum_{m=1}^M \pi_m e^{\lambda \mu_m} = \EE e^{\lambda Y} \le e^{\lambda \EE Y + \frac{\lambda^2 (\mu_{\max} - \mu_{\min})^2}{8}}.$$
    Finally, noting that $\EE X = \EE Y$, we have
    $$\EE e^{\lambda(X-\EE X)} \le \exp\Big[\frac{\lambda^2}{2} \Big(K_{\max}^2 + \frac{1}{4} (\mu_{\max} - \mu_{\min})^2\Big) \Big]$$
    and the proof is complete.
\end{proof}

\begin{proof}[Proof of \cref{lem:change conditioning}]
    Recall that $X_p \mid Y_p \xp 0$ if and only if $\PP(|X_p| > \epsilon \mid Y_p) \xp 0$ for any $\epsilon > 0$. The claim is immediate after using the tower property to write $\PP(|X_p| > \epsilon) = \EE \PP(|X_p| > \epsilon \mid Y_p)$, and using the bounded convergence theorem (for the if part) / Markov inequality (only if part).
\end{proof}

\begin{proof}[Proof of \cref{lem:B3}]
    It suffices to prove the claim under two cases: (1) $m = m_0$, (2) $d = d_0$. First, suppose $m = m_0=0$.
    Noting that $\EE h(V) = \EE h(\sqrt{\frac{d}{d_0}} V_0 )$, we can write
    \begin{align*}
        |\EE h(V) - \EE h(V_0)| &= |\EE h(\sqrt{\frac{d}{d_0}} V_0) - \EE h(V_0)| \\
        &\lesssim \EE\Big|\big(\sqrt{\frac{d}{d_0}} - 1\big)V_0\Big| (1+|V_0|^r) \\
        &\lesssim_r \big|\sqrt{\frac{d}{d_0}} - 1\big| (d_0^{1/2} + d_0^{(r+1)/2}) \lesssim_{d_0,r} |d - d_0|. %
    \end{align*}
    Here, the second line follows from a Taylor expansion alongside the polynomial growth bound $|h'(v)| \lesssim 1+ |v|^r$, where $v \le \max(V_0, \sqrt{d/d_0} V_0) \le \sqrt{2} V_0$.
    The square root term in the final line is simplified using $\sqrt{\frac{d}{d_0}} - 1 = \frac{d-d_0}{\sqrt{d_0}(\sqrt{d}+ \sqrt{d_0})}.$

    Next, suppose $d = d_0 =1$. Then, by similar computations using the bound $$\sup_{v \in [v_0+(m-m_0), v_0]} |h'(v)| \lesssim \sup_{v \in [v_0+(m-m_0), v_0]} \big[1+|v|^r \big] \lesssim_r 1 + |v_0|^r + |m-m_0|^r,$$ we have
    \begin{align*}
        |\EE h(V) - \EE h(V_0)| &= |\EE h(V_0 + (m-m_0) ) - \EE h(V_0)| \\
        &\le |m-m_0| (1+\EE |V_0|^r + |m-m_0|^r) \\
        &\lesssim_r |m-m_0| (1+ m_0^r + |m-m_0|^r).
    \end{align*}
\end{proof}

Finally, we prove \cref{lem:A7} using another technical moment bound as a separate lemma, which will be proved at the end of this appendix.
\begin{lemma}\label{lem:F_i^2 moment}
    Assume $\EEB (\beta_1^\star)^4<\infty$ and let $g$ be a function such that $\EEB g(\beta_1^\star)^2<\infty$. Then, it holds that 
    $$\EE \Big[\frac{n}{p^2}\sumin \mathcal{F}_i^2 g(\beta_i^\star) - \EEB g(\beta_1^\star) \EEB (\beta_1^\star)^2 \Big]^2 = O\Big(\frac{1}{p} + \frac{p}{n}\Big).$$
\end{lemma}

\begin{proof}[Proof of \cref{lem:A7}]
    (a) Let $\Phi$ be any $\mathcal{C}^3$ function such that $|\Phi'(m)|, |\Phi''(m)|, |\Phi'''(m)| = \mcp(r)$ for some constant $r>0$. 
    By a Taylor expansion, we can write
    \begin{align}
        &\sumin \Phi(d_0 \beta_i^\star - \mathcal{F}_i) - \sumin \Phi(d_0 \beta_i^\star) \notag \\
        =& - \sumin \F_i \Phi'(d_0\beta_i^\star) + \frac{1}{2} \sumin \F_i^2 \Phi''(d_0 \beta_i^\star) - \frac{1}{6} \sumin \F_i^3 \Phi'''(d_0 \beta_i^\star + \xi_i),\label{eq:taylor expansion phi}
    \end{align}
    where $\xi_i \in (-\F_i,0)$. We separately control the second moment of each term in the RHS.
    
    We start by control the first term in \eqref{eq:taylor expansion phi}. By computing the moment (c.f. Lemma A.9(a) in \cite{lee2025bayesregression}), we have
    $$\EE_X \Big[ \Big[\sumin \F_i \Phi'(d_0\beta_i^\star)\Big]^2 \mid \bbst\Big] \lesssim \frac{\|\bbst\|^2 \sumin \Phi'(d_0\beta_i^\star)^2 }{n} \lesssim \frac{\|\bbst\|^2 \sumin |\beta_i^\star|^{2r} }{n} = O_{P,\bbst}\Big(\frac{p^2}{n}\Big).$$
    The last bound used assumption R5 that $\beta_i^\star \stackrel{d}{\equiv} B_{\theta_0}$ has finite moments. Hence, we have
    $$\sumin \F_i \Phi'(d_0\beta_i^\star) = O_P\Big(\frac{p}{\sqrt{n}}\Big).$$

    Next, we move on to the second term in \eqref{eq:taylor expansion phi}. By the moment bound in \cref{lem:F_i^2 moment}, we can write 
    $$\frac{1}{2} \sumin \F_i^2 \Phi''(d_0\beta_i^\star) = \frac{p^2}{n} \kappa_{\Phi}(\theta_0) + O_P\Big(\frac{p^{3/2}}{n} + \frac{p^{5/2}}{n^{3/2}}\Big).$$
    Note that the moment requirements in \cref{lem:F_i^2 moment} are immediate by R5.
    
    Finally, we control the third term in \eqref{eq:taylor expansion phi}. By Holder's inequality (with exponents $4/3, 4$), using the growth condition on $\Phi'''$, and recalling $|\xi_i| \le |\F_i|$, we have
    {\small
    \begin{align*}
        \EE_X \Big|\sumin \F_i^3 \Phi'''(d_0 \beta_i^\star + \xi_i) \Big| \le  \Big[\EE_X \sumin \F_i^4 \Big]^{3/4} \Big[\EE_X \sumin \big[ |d_0\beta_i^\star|^{4r} + |\F_i|^{4r} \big] \Big]^{1/4} = O_{P,\bbst}\Big(\frac{p^{5/2}}{n^{3/2}}\Big).
    \end{align*}
    }
    Here, the last bound follows as (i) $\EE_X \Big[\sumin \F_i^4 \Big] = O_{P,\bbst}(p^3/n^2)$ by Lemma A.9(b) in \cite{lee2025bayesregression}, and (ii) $\EE_X \Big[ \sumin |d_0\beta_i^\star|^{4r} + |\F_i|^{4r} \Big] = O_{P,\bbst}(p)$ by the moment assumption for $\beta_i^\star$ and the moment bounds for $\F_i$ in \cref{lem:mgf of F_i}.
    
    The final bound follows by combining the above bounds and noting that 
    $$\frac{p}{\sqrt{n}} + \frac{p^{5/2}}{p^{3/2}}\gg \frac{p^{3/2}}{n}.$$
    Note that the final form in \cref{lem:A7}, which is conditional on $X$, follows from \cref{lem:change conditioning}.
\end{proof}

\begin{proof}[Proof of \cref{lem:mgf of F_i}]
    Throughout the proof, all expectations will be over $X$, and always be conditional on $\bbst$.
    Fix any $i \le p$. For $k \le n$, let $Z_{k,i} := \sqrt{n} X_{k,i}$ and recall from \cref{assmp:random design} that $Z_{k,i}$ are mean zero sub-Gaussian random variables with sub-Gaussian norm bounded by $K>0$. Let us assume
    By definition, we can write $\F_i = \frac{1}{n} \sum_{k=1}^n \sum_{1 \le j \neq i \le p} Z_{k,i} Z_{k,j} \beta_j^\star$. Setting $Y_{k,i} := \sum_{1 \le j \neq i \le p} Z_{k,j} \beta_j^\star$, $Y_{k,i}$ is a sub-Gaussian random variable whose norm is bounded by $K\|\beta^\star\|$ by Hoeffding's inequality. 
    
    Now, the product property for sub-Gaussians (e.g. Lemma 2.8.6 in \cite{vershynin2018high}) implies that $Z_{k,i} Y_k$ is sub-exponential with the corresponding sub-exponential norm bounded by $K^2\|\bbst\|$. This can be equivalently written out as
    $$\EE e^{\lambda' Z_{k,i} Y_k} \le e^{K^4\|\beta^\star\|^2 \lambda'^2}, \quad \forall \,\, |\lambda'| \le \frac{1}{K^2 \|\beta^\star\|}.$$
    Now, by spelling out $\F_i = \frac{1}{n} \sum_{k=1}^n Z_{k,i} Y_{k,i}$ and setting $\lambda = \lambda'n$, we have
    $$\EE e^{\lambda \F_i} = \prod_{i=1}^n \EE e^{\frac{\lambda}{n} Z_{k,i} Y_k} \le e^{K^4 \|\beta^\star\|^2 \lambda^2/n}$$
    as long as $|\lambda|\le \frac{n}{K^2 \|\beta^\star\|}$. The final conclusion follows by noting that $\|\bbst\| = \OPB(\sqrt{p}).$
\end{proof}

\begin{proof}[Proof of \cref{lem:F_i^2 moment}]
    For notational convenience, define a $n \times p$ matrix $Z$ with entries $Z_{k,i} := \sqrt{n} X_{k,i}$, so that each $Z_{k,i}$ are independent sub-Gaussians with mean 0 and variance 1.
    We will write $\EE_{\bZ}$ and $\EEB$ for expectations with respect to the matrix $Z$ and $\bbeta$. Observe that 
    \begin{align}\label{eq:zall}
        \EE_{\bZ} \mathcal{F}_i^2=\EE_{\bZ}\left[\frac{1}{n^2}\sum_{k_1,k_2}\sum_{j_1,j_2\neq i} Z_{k_1,i}Z_{k_1,j_1}Z_{k_2,i}Z_{k_2,j_2}\beta_{j_1}^\star \beta_{j_2}^\star\right]=\frac{1}{n}\sum_{j\neq i} (\beta_j^\star)^2.
    \end{align}
    By using \eqref{eq:zall}, we get: 
    \begin{align*}%
        \frac{n}{p^2}\EEB\EE_{\bZ}\sumin \mathcal{F}_i^2 g(\beta_i^\star) &= \frac{1}{p^2}\EEB\sum_{i\neq j}(\beta_j^\star)^2 g(\beta_i^\star) \\ &=\frac{p-1}{p}\EEB(\beta_1^\star)^2\EEB g(\beta_1^\star)\to \EEB (\beta_1^\star)^2 \EEB g(\beta_1^\star).
    \end{align*}
    Having established the limit for the mean, we proceed to show the variance converges to $0$. Using the law of total variance, we will separately show that 
    \begin{align}\label{eq:tsh}
    \frac{n^2}{p^4}\EEB\Var\left(\sumin \mathcal{F}_i^2 g(\beta_i^\star)\mid \bbeta\right)\to 0 \quad \mbox{and} \quad \frac{n^2}{p^4}\Var\EE_{\bZ}\left(\sumin \mathcal{F}_i^2 g(\beta_i^\star)\mid \bbeta\right)\to 0.
    \end{align}
    By \eqref{eq:zall}, we have: 
    \begin{align*}
        &\;\;\;\;\frac{n^2}{p^4}\Var\EE_{\bZ}\left(\sumin \mathcal{F}_i^2 g(\beta_i^\star)\mid \bbeta\right) \nonumber \\ &=\frac{1}{p^4}\Var \left(\sum_{i\neq j} (\beta_j^\star)^2 g(\beta_i^\star)\right) \nonumber \\ &=\frac{1}{p^4}\sum_{i\neq j} \Var(g(\beta_i^\star)(\beta_j^\star)^2) + \frac{1}{p^4}\sum_{i_!\neq j_1} \sum_{i_2\neq j_2} \Cov(g(\beta_{i_1}^\star)(\beta_{j_1}^\star)^2, g(\beta_{i_2}^\star) (\beta_{j_2}^\star)^2) \nonumber \\ &\lesssim p^{-2}+p^{-1}\to 0,
    \end{align*}
    where the last inequality follows from the fact that the covariance term is $0$ if all the indices $i_1,i_2,j_1,j_2$ are distinct. For the other term in \eqref{eq:tsh}, let us define 
    $$A(i,j_1,j_2):=\beta_{j_1}^\star \beta_{j_2}^\star g(\beta_{i}^\star), \,\,\,\, 1\le i,j_1,j_2\le p.$$
    We then note that 
    \begin{align*}
        &\;\;\;\;\EEB\Var\left(\sumin \mathcal{F}_i^2 g(\beta_i^\star)\mid \bbeta\right) \\ &=\frac{1}{n^4}\sum_{k_1,k_2=1}^n \sum_{i=1}^p \sum_{j_1,j_2\neq i} \EEB [A^2(i,j_1,j_2)]\Var(Z_{k_1,i}Z_{k_1,j_1}Z_{k_2,i}Z_{k_2,j_2}) \\ &\qquad +\frac{1}{n^4}\sum_{k_1,k_2,k_3,k_4=1}^n \sum_{i_1,i_2=1}^p \sum_{\substack{j_1,j_2\neq i_1, \\ j_3,j_4\neq i_2}} \EEB [A(i_1,j_1,j_2)A(i_2,j_3,j_4)]\Cov(Z_{k_1,i_1} Z_{k_1,j_1}Z_{k_2,i_1}Z_{k_2,j_2}, \\ &\qquad \qquad Z_{k_3,i_2} Z_{k_3,j_3}Z_{k_4,i_2}Z_{k_4,j_4}).
    \end{align*}
    As $\EE A^2(i,j_1,j_2)<\infty$ uniformly, for $j_1,j_2\neq i$, we have: 
    $$\frac{1}{n^4}\sum_{k_1,k_2=1}^n \sum_{i=1}^p \sum_{j_1,j_2\neq i} \EEB [A^2(i,j_1,j_2)]\Var(Z_{k_1,i}Z_{k_1,j_1}Z_{k_2,i}Z_{k_2,j_2})=\frac{p^3}{n^2} = o(p^4/n^2).$$
    To bound the covariance part, we will first isolate terms for which 
    \begin{align*}
        \Cov(Z_{k_1,i_1} Z_{k_1,j_1}Z_{k_2,i_1}Z_{k_2,j_2},Z_{k_3,i_2} Z_{k_3,j_3}Z_{k_4,i_2}Z_{k_4,j_4})
    \end{align*}
    is not equal to $0$. 
    If the set $\{k_1,k_2,k_3,k_4\}$ consists of $3$ or more distinct elements, then clearly the above covariance term equals $0$. If $|\{k_1,k_2,k_3,k_4\}|=2$, once again if $k_1=k_2$, $k_3=k_4$, but $k_1\neq k_3$, then the covariance term will be $0$. So let us consider the case $k_1=k_3$, $k_2=k_4$. In that case, the covariance term reduces to 
    $$\Cov(Z_{k_1,i_1}Z_{k_1,j_1}Z_{k_2,i_1}Z_{k_2,j_2},Z_{k_1,i_2}Z_{k_1,j_3}Z_{k_2,i_2}Z_{k_2,j_4}).$$
    Now if $i_1\neq i_2$, then simple calculation shows that the above covariance term can be non-zero only when $j_3=j_4=i_1$ and $j_1=j_2=i_2$. The contribution of such terms is given by 
    \begin{align*}
        \frac{1}{n^4}\sum_{k_1\neq k_2}\sum_{i_1\neq i_2}\EEB[A(i_1,i_2,i_2)A(i_2,i_1,i_1)]\Var(Z_{k_1,i_1}Z_{k_1,i_2}Z_{k_2,i_1}Z_{k_2,i_2})\lesssim \frac{p^2}{n^2}=o(p^4/n^2).
    \end{align*}
    Next suppose $i_1=i_2$. Then we can check that the covariance term will be $0$ if either $j_1\neq j_3$ or $j_2\neq j_4$. Therefore, to get non-zero covariance we must have $j_1=j_3$ and $j_2=j_4$. The contribution of such terms is given by 
 \begin{align*}
     \frac{1}{n^4}\sum_{k_1\neq k_2}\sum_{i=1}^p\sum_{j_1,j_2\neq i}\EEB[A^2(i,j_1,j_2)]\Var(Z_{k_1,i}Z_{k_1,j_1}Z_{k_2,i}Z_{k_2,j_2}) = \frac{p^3}{n^2} = o(p^4/n^2).
 \end{align*}
 The final remaining case is when $|\{k_1,k_2,k_3,k_4\}|=1$. In this case, it is not hard to check that unless $|\{i_1,i_2,j_1,j_2,j_3,j_4\}|\le 4$, the covariance term will be $0$. The contribution of such terms is given by 
 {\small
 \begin{align*}
 &\;\;\;\;\frac{1}{n^4}\sum_{k=1}^n \sum_{\substack{i_1,i_2,(j_1,j_2)\neq i_1, (j_3,j_4)\neq i_2, \\ |\{i_1,i_2,j_1,j_2,j_3,j_4\}|\le 4}} \EEB[A(i_1,j_1,j_2)A(i_2,j_3,j_4)]\Cov(Z_{k,i_1}^2 Z_{k,j_1}Z_{k,j_2},Z_{k,i_2}^2 Z_{k,j_3}Z_{k,j_4})\\ &\lesssim \frac{p^4}{n^3}=o(p^4/n^2).
 \end{align*}
 }
 This establishes \eqref{eq:tsh} and completes the proof.
\end{proof}

\section{Simulation details}\label{sec:sim details}
\subsection{Parameter space}
For each parametric family defined in \cref{sec:simulations} from the main text, we consider the following compact parameter space $\Theta$:
\begin{enumerate}[(1)]
    \item for $\textsf{Ber}(\pi)$, consider $\pi \in [0,1]$,
    \item for $\textsf{Spike-Slab}(\pi,\tau^2)$ consider $\pi \in [0, 1], \tau \in [0.2, 5]$,
    \item for $\textsf{location-GMM}(\theta_1,\theta_2)$, consider $\theta_1, \theta_2 \in [-2, 2]$.
\end{enumerate}

\subsection{Implementation details}
The three methods are implemented as follows.
\begin{itemize}
    \item \textbf{Variational EB:} We computed the vEB estimator by using the \texttt{optimize} (one-dimensional) and \texttt{optim} (higher-dimensional) functions in \texttt{R} to solve the minimization \eqref{eq:est def}. For the latter, we used the ``L-BFGS-B'' method with the above parameter space constraints.

    \item \textbf{Method of Moments:} For each parametric family, we solve the following moment equations using the LSE $\hat{\beta} = (X^\top X)^{-1} X^\top y$:
    \begin{align*}
        &\EE_{\textsf{Ber}} \beta = \pi, 
        \qquad \EE_{\textsf{Spike-Slab}} \big[\beta^2,~ \beta^4\big] = (1-\pi) \big[\tau^2,~ 3 \tau^4 \big], \\
        &\EE_{\textsf{location-GMM}} \big[\beta,~ \beta^2\big] = \frac{1}{2} \big[\theta_1+\theta_2,~ \theta_1^2 + \theta_2^2 + 1.25 \big].
    \end{align*}
    The final estimates are thresholded to take values in the compact parameter space.
    
    \item \textbf{Langevin diffusion:} We adaptively sample the prior parameters $\theta = (\theta_1, \theta_2)$ and regression coefficients $\beta$ $4000$ times via the discretized Langevin algorithm in \cite{fan2025dynamical}. The algorithm is initialized under the constraints $\theta_1 < \theta_2$, with the default learning rate $10^{-2}$ (for both $\theta, \beta$). The final estimates are thresholded to take values in the compact parameter space.
\end{itemize}

The Mean-Field approximation in \eqref{eq:M_theta(u) definition} is computed by iterating the fixed point equation for $u$ in \cref{def:fixed point}.
To evaluate coverage guarantees under the Oracle posterior $\PP_{\theta_0}$, we draw samples from $\PP_{\theta_0}$ as follows.
\begin{itemize}
    \item \textbf{Bernoulli}: We use the \texttt{IsingSampler} package in \texttt{R} \citep{isingsampler}, using the Metropolis-Hastings algorithm. We collect 500 samples.
    \item \textbf{Spike-Slab, GMM}: For both models, we implemented a parallel Gibbs sampler (initialized at the true coefficients $\beta^\star$) and collected 5000 samples. We used the last half of the samples to evaluate the coverage probabilities.
\end{itemize}

\subsection{Computation time} We report the average computation time per scenario in \cref{tab:time} (which corresponds to \cref{tab:mse_theta}). Across all settings, the vEB estimator is the fastest. In particular, the MoM estimator suffers in high-dimensions, which is a consequence of having to evaluate an matrix inversion to compute the least squares estimator. In contrast, the vEB estimator just has to solve a fixed-dimensional optimization.

\begin{table}[H]
    \centering
    \begin{tabular}{c|cc|cc|ccc}
    \toprule
    Prior & \multicolumn{2}{c}{\textsf{Ber}} & \multicolumn{2}{c}{\textsf{Spike-Slab}} & \multicolumn{3}{c}{\textsf{location-GMM}} \\
    $p$ \textbackslash{} Algo & vEB & MoM & vEB & MoM & vEB & MoM    & Langevin \\
    \midrule
    25 & 0.001 & 0.001 & 0.001 & 0.001 & 0.001 & 0.001 & 0.21 \\ %
    50 & 0.001 & 0.009 & 0.002 & 0.005 & 0.001 & 0.006 & 0.20 \\ %
    100& 0.005& 0.069 & 0.009 & 0.103 & 0.007 & 0.093 & 0.34 \\ %
    200& 0.059 & 1.103 & 0.053 & 1.278 & 0.051 & 1.580 & 0.58 \\ %
\bottomrule
\end{tabular}
    \caption{Elapsed time (in seconds) for varying $p$ and $n = p^2$ (averaged over 400 replications). The Langevin estimator is implemented in \texttt{Python} while all other estimators (vEB and MoM) are implemented in \texttt{R}.}
    \label{tab:time}
\end{table}

\subsection{Additional results}
In \cref{tab:mse_theta_second_table}, we additionally report the accuracy of the MoM and Langevin estimators under an uninformative random initialization. As mentioned as a footnote in the main paper, these estimators suffer from fundamental non-identifiability and multiple local maximizers, respectively. We see that the MoM estimator results in an inconsistent result whereas the Langevin estimator only mildly suffers compared to the analogous results in \cref{tab:mse_theta}. 

\cref{tab:mse_theta_second_table} also reports simulations results under the heavy-tailed Cauchy prior with location parameter $\theta$. While our theory does not allow this due to a violation of the moment condition R5, the results indicate consistency but with an empirical convergence rate slower than $\sqrt{p}$, and it would be interesting to investigate whether the convergence rate fundamentally suffers under heavy-tailed priors.
\begin{table}[H]
    \centering
    \begin{tabular}{c|cc|c}
    \toprule
    Prior & \multicolumn{2}{c}{\textsf{location-GMM}} & {\textsf{Cauchy}$(\theta,1)$} \\
    $p$ \textbackslash{} Algo & MoM& Langevin & vEB \\
    \midrule
25 & 0.587 & 0.730 & 0.832 \\
50 & 0.623 & 0.422 & 0.584 \\
100& 0.596 & 0.192 & 0.366 \\
200& 0.589 & 0.093 & 0.342 \\
\bottomrule
\end{tabular}
    \caption{MSE $ = \|\thetaveb - \theta_0\|^2$ for varying $p$ and $n = p^2$ (averaged over 400 replications). The \textsf{location-GMM} results are computed completely uninformatively.}
    \label{tab:mse_theta_second_table}
\end{table}

\end{appendix}

\end{document}